 \documentclass{amsart}
\RequirePackage[colorlinks,citecolor=blue,urlcolor=blue,linkcolor=blue]{hyperref}
\hypersetup{
colorlinks = true,
citecolor=blue,
urlcolor=blue,
linkcolor=blue,
pdfpagemode = UseNone
}
\usepackage[OT2,T1,T2A]{fontenc}
\usepackage[russian,USenglish]{babel}
\usepackage{graphicx,hyperref,xcolor}
\usepackage{enumerate}
\usepackage{geometry,soul,fancybox}
 \usepackage{tikz}
 \usetikzlibrary{patterns}
\newtheorem{theorem}{Theorem}[section]
\newtheorem{lemma}[theorem]{Lemma}
\newtheorem{proposition}[theorem]{Proposition}
\newtheorem{corollary}[theorem]{Corollary}
\theoremstyle{definition}
\newtheorem{definition}{Definition}[section]

\newtheorem{remark}[theorem]{Remark}

\newtheorem{claim}[theorem]{Claim}

\newtheorem{propositiona}{Proposition X}

\theoremstyle{definition}

\newtheorem{remarka}{Remark X}
\numberwithin{equation}{section}
\newcommand{\eps}{\varepsilon}
\newcommand{\la}{\lambda}
\newcommand{\RR}{\mathbb{R}}
\newcommand{\EE}{\mathbb{E}}
\newcommand{\CC}{\mathbb{C}}
\newcommand{\ZZ}{\mathbb{Z}}

\newcommand{\q}{\mathfrak {h}}
\newcommand{\floor}[1]{\lfloor#1\rfloor}
\def\topp#1{^{(#1)}}
\def\topp#1{^{\left(#1\right)}}
\def\ind{\mathbf{1}}

\newcommand{\wt}[1]{\widetilde{#1}}
    \def\re{\textnormal {Re}}
    
     \def\d{{\textnormal d}}
    \def\i{{\textnormal i}}

    \def\vv#1{\mathbf{#1}}
\makeatletter
\renewcommand*\env@matrix[1][\arraystretch]{%
  \edef\arraystretch{#1}%
  \hskip -\arraycolsep
  \let\@ifnextchar\new@ifnextchar
  \array{*\c@MaxMatrixCols c}}
\makeatother

\newcommand{\comment}[1]{\ovalbox{\footnotesize \color{magenta}#1}}
  
\newcommand{\hide}[1]{\comment{\tiny (hidden)} }
 \usepackage{framed,pgf}
\newcounter{oldeq}
\newcounter{usesofarxiv}
 \newcommand{\arxiv}[1]{
\setcounter{oldeq}{\value{equation}}
 \addtocounter{usesofarxiv}{1}
 \setcounter{equation}{0}
\def\theoldeq{\theequation}
\def\theequation{x-\arabic{usesofarxiv}.\arabic{equation}}
\def\theequation{\arabic{section}.\arabic{usesofarxiv}.\arabic{equation}}
\def\theequation{\thesection.\arabic{usesofarxiv}.\arabic{equation}}
  \colorlet{shadecolor}{gray!10}
{
\begin{shaded}
\footnotesize
#1  \normalsize
\end{shaded}
   \setcounter{equation}{\value{oldeq}}
\numberwithin{equation}{section}
}}
\newcommand{\A}{\mathsf{a}}
 
 \newcommand{\UU}{\mathfrak{u}}
 \newcommand{\Cab}{\mathsf{c}}

\newcommand{\Cb}{\mathsf{c}_2}
\newcommand{\Ca}{\mathsf{c}_1}

\renewcommand{\arxiv}[1]{}

\usepackage{CJKutf8}

\title{Geometric last passage percolation on a strip and free Askey--Wilson functionals}%
\title{Free Askey--Wilson functionals and geometric last passage percolation on a strip}%
\author{W{\l}odek Bryc}%
\address
{
W{\l}odzimierz Bryc\\
Department of Mathematical Sciences\\
University of Cincinnati\\
2815 Commons Way\\
Cincinnati, OH, 45221-0025, USA.
}
\email{wlodek.bryc@gmail.com}%

\author{Kamil Szpojankowski}%
\address{Kamil Szpojankowski, Faculty of Mathematics and Information Science,
Warsaw University of Technology, pl. Politechniki 1, 00-661
Warszawa, Poland}
\email{kamil.szpojankowski@pw.edu.pl}%
\author{Jacek  Weso{\l}owski}
\address{Jacek Weso{\l}owski, Faculty of Mathematics and Information Science,
Warsaw University of Technology, pl. Politechniki 1, 00-661
Warszawa, Poland; Statistics Poland, Al. Niepodleglosci 208, 00-925 Warsaw}
\email{jacek.wesolowski@pw.edu.pl}

\begin{document}

\begin{abstract}
Barraquand, Corwin, and Yang \cite{barraquand2024stationary} established that geometric last passage percolation (LPP) on a strip of $\ZZ^2$ has a unique stationary measure. Building on this, Barraquand \cite{Barraquand-2024-integral} derived explicit contour integral formulas for the model’s multipoint probability generating function. In this paper, we introduce {\em free Askey--Wilson functionals}  and use them to   extend these generating function formulas. Our framework yields explicit expressions valid over a broader range of boundary parameters than previously accessible. This generalization allows us to determine the full phase diagram that characterizes how the large-scale asymptotics of the stationary measure depend on the boundary conditions. In addition, we prove a Poisson approximation for the stationary measure when the 
parameters vary with the strip width.
\end{abstract}
\maketitle
\arxiv{This is an expanded version of the paper}
\section{Introduction}\label{Sec:Intro}

Stationary measures of out-of-equilibrium systems in finite domains typically exhibit a nontrivial dependence on boundary conditions. A classical example is the asymmetric simple exclusion process (ASEP) on $N$ sites, coupled to boundary reservoirs (open ASEP). In this work, we demonstrate that analogous phenomena arise in the stationary measure of the directed geometric last passage percolation (LPP) model on a strip, recently introduced in \cite{barraquand2024stationary}. A technical reason underlying this similarity in phase diagrams is that, somewhat unexpectedly, the two models are linked through integral representations of their generating functions, first discovered in the fan region of ASEP in \cite{Bryc-Wesolowski-2015-asep}, and more recently for geometric LPP in \cite{Barraquand-2024-integral}. The objective of the present work is to deepen the theoretical understanding of such representations and to extend their applicability to the full range of boundary parameters.
\subsection{Geometric last passage percolation on a strip}

For $N=1,2,\dots$,
we consider a strip
$$\mathbb{S}_N:=\{(j,n): n\in \ZZ_{\ge 0}, j\in\{n, n+1, \dots, n+N\}\}\subset \ZZ_{\ge 0}^2$$
of width $N+1$  of the $\mathbb{Z}^2$ lattice  with vertices $(j,n)$. Vertices $(n,n)$ for $n\geq 0$ are called the left boundary vertices; vertices $(n+N,n)$ for $n\geq 0$  are called the right boundary vertices and the remaining vertices $(j,n)$ for $n\geq 0$ and $n<j<n+N$ are called the bulk vertices (at the $n$-th   level).

We   write $X\sim Geo(p)$ if a random variable $X$ has the geometric distribution  $P(X=n)=p^n(1-p)$, $n=0,1,\dots$  with parameter  $0\leq p<1$.
\begin{definition}[geometric LPP]
Let $\A\in(0,1)$ be the bulk parameter  and let $\Ca ,\Cb \geq 0$
be boundary parameters  such that $\A \Ca ,\A \Cb <1$.
Let $\left(\omega_{j,n}\right)_{ 1\leq n\leq j\leq n+N}$ be a sequence of independent geometric random variables indexed by the vertices of the strip above level  $0$. We assume that at the boundaries of the strip, $\omega_{n,n}\sim Geo(\A \Ca )$,
 $\omega_{n,n+N}\sim Geo(\A \Cb )$  and that in the bulk  $n<j<n+N$ we have  $\omega_{j,n}\sim Geo(\A^2)$, $n=1,2,\dots$.

For a given initial condition $G(i,0)\ge 0$,  $ 0\leq i \leq N$ with $G(0,0)=0$, the geometric LPP is  the sequence of random variables $G(j,n)$ indexed by the vertices $(j,n)$ of the strip,  which solve the following recursion: for $j\geq 0$, $n\geq 1$,
\begin{equation}
  \label{LPP-Geo+}
  G(j,n)=\omega_{j,n}+\max\{G(j-1,n), G(j,n-1)\}, \quad n\leq j\leq n+N,
\end{equation}
with the convention that  $G(n-1,n)\equiv 0$ and $G(n+N,n-1)\equiv 0$ for the vertices in $\ZZ^2$ that are outside of the strip $\mathbb{S}_N$.
 \end{definition}
In particular, $G(j,1)=\max_{1\leq k\leq j} \{G(k,0)+ \sum_{i=k}^j\omega_{k,1}\}$, $j=1,\dots,N+1$.
We are interested in the laws of the vectors $\vv G_n:=(G(j+n,n)-G(n,n))_{j=0,\dots,N}$  along the horizontal paths at level $n$.
\begin{definition}[stationary measure]
 We say that the law $\nu$ on $\ZZ_{\geq 0}^N$ is a stationary measure for the LPP if
 under the initial condition  $\vv G_0$ that has law $\nu$ and is independent of $\left(\omega_{j,n}\right)_{ 1\leq n\leq j\leq n+N}$,   the law of vector $\vv G_n$ is given by the same law $\nu$ for any $n\geq 1$.
\end{definition}

Barraquand, Corwin and Yang \cite[Theorem 1.3]{barraquand2024stationary} established a  remarkable result that if $\Ca ,\Cb \in(0,1/\A)$,  then the geometric LPP has a   unique stationary measure, which is given as a marginal of the following  two-layer ensemble.
For $\vv m=(m_1,\dots m_N)$, $\vv n=(n_1,\dots,n_N)\in \ZZ_{\geq 0}^N$  consider two sequences of   partial sums   $\vv L_1,\vv L_2$ given   by $\vv L_1(j)=\sum_{k=1}^j m_k$, $\vv L_2(j)=\sum_{k=1}^j  n_k$ for $j=1,\dots,N$,  with $\vv L_1(0)= \vv L_2(0)=0$.  With $\A\in(0,1)$, $\Ca ,\Cb \in(0,1/\A)$.
we assign the probability measure  to   the pair of sequences $(\vv L_1,\vv L_2)$ using the weight function
\begin{equation}
  \label{QQQ}
\mathbb{Q}\topp{\A,\Ca ,\Cb }_N( \vv L_1, \vv L_2)=\A^{\vv L_1(N)} \A^{\vv L_2(N)}(\Ca \Cb )^{\max_{1\leq j\leq N}(\vv L_2(j)-\vv L_1(j-1))}\Cb ^{\vv L_1(N)-\vv  L_2(N)}.
\end{equation}
which according to \cite[Proposition 2.18]{barraquand2024stationary} (see also Lemma \ref{Lem:conv}) can be normalized by the finite constant
\begin{equation}
  \label{ZN}
  \mathcal{Z}\topp{\A,\Ca ,\Cb }(N)=\sum_{\vv L_1,\vv L_2}
\mathbb{Q}\topp{\A,\Ca ,\Cb }_N( \vv L_1, \vv L_2).\end{equation}
\begin{theorem}[{\cite[Theorem 1.6]{barraquand2024stationary}}]
If $\A\in(0,1)$, $\A \Ca , \A \Cb \in(0,1)$ then the marginal law  $\nu$ of $\vv L_1$ under the normalized measure
\begin{equation}\label{P-BCY}
   \mathbb{P}\topp{\A,\Ca ,\Cb }_{\rm Geo}(\vv L_1,\vv L_2)=\frac{1}{\mathcal{Z}\topp{\A,\Ca ,\Cb }(N)} \mathbb{Q}\topp{\A,\Ca ,\Cb }_N( \vv L_1, \vv L_2)
\end{equation} is the unique stationary measure of the geometric LPP.
\end{theorem}
 More recently, Guillaume Barraquand \cite{Barraquand-2024-integral} derived explicit integral formulas for the $d$-point probability generating function. In the simplest case, $d=1$, his formula is equivalent to the following:

 \begin{equation}\label{GB-k=1}
    G_N(t):= \sum_{\vv L_1,\vv L_2}t^{2 \vv L_1(N)}\mathbb{Q}\topp{\A,\Ca ,\Cb }_N( \vv L_1, \vv L_2)=\frac{1-\Ca\Cb}{2\pi}\int_{-2}^2 \, \frac{\sqrt{4-y^2}}{\q_{\A t}^N(y)\,\q_{\Ca/t}(y)\,\q_{\Cb t}(y)}\,\d y,
 \end{equation}
 where
 \begin{equation}
    \label{Jck-q}
    \q_\alpha(y)=1+\alpha^2 - \alpha y.
\end{equation}
provided that $\Ca,\Cb<1$, $\A t<1$, $\Ca<t<1/\Cb$.  The  series on the left-hand side of
 \eqref{GB-k=1} converges when $\A \Ca, \A t, \A \Cb t^2<1$, so according to \cite{Barraquand-2024-integral}   formula \eqref{GB-k=1} holds  in this generality after the right-hand side is replaced by its analytic continuation in parameters $\Ca,\Cb$.
Our goal is to make this analytic continuation more explicit, so that we can use it to prove limit theorems by analyzing the Laplace transform $$\EE[e^{2 s \vv L_1(N)/N}]=G_N(e^{s/N})/\mathcal{Z}\topp{\A,\Ca ,\Cb }=G_N(e^{s/N})/G_N(1).$$
It turns out that formula \eqref{GB-k=1} yields a concise expression that remains valid
   over a larger range of parameters than allowed in \eqref{GB-k=1}.
 \begin{proposition}
   \label{Prop1}
 If $0<\A, \A t, \A \Ca ,\A \Cb t^2 <1$, then
   \begin{equation}
   \label{GenG}
   \sum_{N=1}^\infty z^{N-1} G_N(t)=\frac{B}{\A (t-\Ca  B)(1-\Cb t B)},
 \end{equation}
 where $B=B(z,\A t)<1$ is the smaller root of the quadratic equation
 $$
B+\frac1B=\frac{1-z}{\A t}+\A t
 $$
 with $0\leq z<(1-\A t)^2$ which is small enough to ensure that $B\Ca <t, B\Cb  t <1$.
 \end{proposition}

\arxiv{Explicitly, we have
$$
B(z,\A t)=\frac{1+\A^2 t^2-z-\sqrt{\left((1-\A t)^2-z\right) \left((1+\A t)^2-z\right)}}{2 \A t}\leq 1.
$$
Since $\lim_{z\to 0}B(z,\A t)=\A t$ and $\A \Ca ,\A \Cb t^2 <1$, one can find $\delta>0$ such that $B(z,\A t)\Ca ,B(z,\A t)\Cb <1$ for all $0\leq z<\delta$.
 }

In this paper, however, we adopt a different approach, working directly with
\eqref{QQQ}. For parameters satisfying $0<\A \Ca, \A t, \A \Cb t^2<1$,
Theorem \ref{Thm1-} allows us to replace \eqref{GB-k=1}
with
\begin{equation}\label{GenG-L}
    G_N(t) = \mathbb{L}\topp{\Ca/t,\Cb t}[\q_{\A t}^{-N}],
\end{equation}
where
\begin{equation}\label{L-def0}
    \mathbb{L}\topp{a,b}[f]=\frac{1}{2\pi \i}\oint_{|w|=\rho} f(w+1/w)\frac{1-w^2}{(1-a w)(1-b w)} \frac{\d w}{w}, \quad \rho,|a|\rho,|b|\rho<1,
\end{equation}
is the moment functional for the Askey-Wilson polynomials \cite{Askey-Wilson-85} with parameters $(a,b,0,0)$ and $q=0$, extended to an appropriate class of analytic functions.

By analyzing the representation \eqref{GenG-L} for the Laplace transform $\EE\left[e^{2s \vv L_1(N)/N}\right] = G_N(e^{s/N})/G_N(1)$, we derive the phase diagram for the geometric LPP, which closely parallels the phase diagram known for the ASEP \cite{derrida93exact},
as predicted by universality; see Fig.~\ref{Fig-PhD}.
(A similar phase diagram also appears in \cite{Zongrui-2024-sic-vertex} for a different model.)
The large-scale limit of $\vv L_1(N)/N$ depends on the boundary parameters $\Ca$ and $\Cb$ as detailed below.
 \begin{figure}[hbt]
  \begin{tikzpicture}[scale=1.45]
\draw[scale = 1,domain=6.8:8,smooth,variable=\x,dotted,blue] plot ({\x},{1/((\x-7)*1/3+2/3)*3+5});
\draw[scale = 1,domain=8:10,smooth,variable=\x,dotted,red] plot ({\x},{1/((\x-7)*1/3+2/3)*3+5});

 \draw[->] (5,5) to (5,11);
 \draw[-,thick,magenta] (5,5) to (5,10);
  \draw[-,thick,red] (5,5) to (10,5);
 \draw[->] (5.,5) to (11,5);
   \draw[-, dashed] (5,8) to (8,8);
   \draw[-, dashed] (8,8) to (8,5);
   \draw[-,red,thick] (8,8) to (10,10);
   \node [left] at (4.9,8) { $1$};
   \node[below] at (8,4.9) {  $1$};
     \node [below] at (10.7,5) {$ {\Ca } $};
   \node [left] at (5,10.7) {$ {\Cb }$};

  \draw[-] (8,4.9) to (8,5.1);
   \draw[-] (4.9,8) to (5.1,8);

 \node [below] at (5,5) {\scriptsize$(0,0)$};
    \node [above] at (6.5,8.5) {\color{red}H};
    \node[above] at (8,9) {\color{red}H};
         \node [below] at (6.5,8.5){\color{red}$\rho=\tfrac{\A \Cb }{1-\A \Cb }$};
    \node [below] at (9,6) {\color{blue}L}; %
      \node [above] at (9,8) {\color{blue}L};
       \node [below] at (9,6.5)  {\color{blue}$\rho=\tfrac{\A }{\Ca -\A}$};
    \node [below] at (6.5,6.5) {$\rho=\tfrac{\A}{1-\A}$};
 \node [below] at (6.5,6) {I};%

\node [below] at (10,4.8){$\frac{1}{\A}$};
  \draw[-] (10,4.9) to (10,10);
\node [left] at (4.8,10){$\frac{1}{\A}$};
\draw[-] (4.9,10) to (10,10);
\end{tikzpicture}
  \caption{The limit
  $\frac{1}{N}\vv L_1(N)\to \rho=\rho(\A,\Ca ,\Cb )$ and the boundary parameters $\Ca ,\Cb \in(0,1/\A)$.
  Below the hyperbola $\Ca \Cb =1$, the functional $\pi^t$ in \eqref{M-Trans+} is given as an integral with respect to a positive Askey--Wilson measure. For $(\Ca,\Cb)\in I:=(0,1]^2$ the limit does not depend on the boundary parameters.
  The  line between the points (1,1) and $(1/\A,1/\A)$ separates the regions $L$ and $H$ with the low and high average and the limit of  $\frac{1}{N}\vv L_1(N)$ on that line is random, \eqref{U-lim}.  }
  \label{Fig-PhD}
\end{figure}
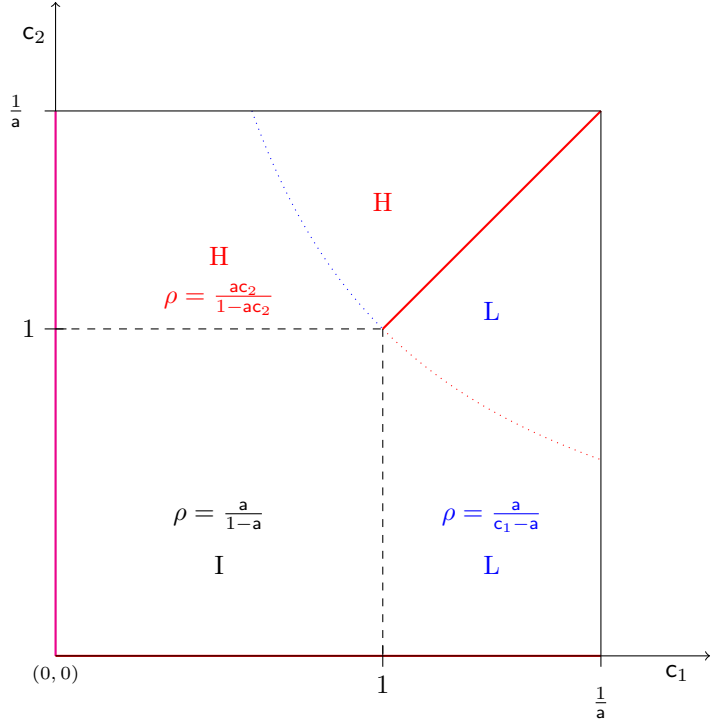

\begin{theorem}\label{Thm-PhD}
Let $\A\in(0,1)$ and $\Ca ,\Cb \in(0,1/\A)$.
\begin{enumerate}
  [(i)]
\item    If %
$\Ca ,\Cb \leq 1$,  then $\vv L_1(N)/N\to \A  /(1-\A)$ in probability.

\item   If $\Ca >1$ %
and $\Cb <\Ca <1/\A$, then $\vv L_1(N)/N\to {\A }/(\Ca -\A)$ in probability.

\item   If $\Cb >1$ %
and $\Ca <\Cb <1/\A$, then $\vv L_1(N)/N\to {\A \Cb }/({1-\A \Cb })$ in probability.
    \item   If  $\Ca =\Cb =\Cab\in (1,1/\A)$,  then  $\vv L_1(N)/N$ converges in distribution to  a random limit that interpolates between the high and low density limits from (ii) and (iii).  We have
    \begin{equation}
      \label{U-lim}  \frac{1}{N}\vv L_1(N)\Rightarrow   \frac{\A}{\Cab-\A} U+\frac{\A \Cab}{1-\A \Cab}(1- U)
    \end{equation}   with  $U$ distributed uniformly on $[0,1]$.
\end{enumerate}
\end{theorem}
The proof of Theorem \ref{Thm-PhD} is in Section \ref{Sec:ProofPhD}.

Formula \eqref{GenG-L} admits a natural generalization, providing a representation for the generating function
\begin{equation}\label{GG}
   \mathcal{G}\topp{\A,\Ca ,\Cb }_N(\vv t)
   = \sum_{\vv L_1, \vv L_2} \prod_{j=1}^N t_j^{2( \vv L_1(j) - \vv L_1(j-1))} \, \mathbb{Q}\topp{\A,\Ca ,\Cb }_N,
   \quad \vv t=(t_1, \dots, t_N)\in(0,\infty)^N . %
\end{equation}
in terms of a three-parameter family \( \mathbb{L}\topp{a,b,c} \) of functionals acting on analytic functions, introduced in Definition~\ref{Def-J-ext0}.
We focus on two specific types of such functionals, which are parametrized by \( \Ca \), \( \Cb \), \( 0 < s \leq t \), and a complex number \( x \). Since \( \Ca \) and \( \Cb \) remain fixed throughout, we omit their explicit appearance in the notation:
 \begin{equation}\label{M-Trans+}
   \pi^{t }  =\mathbb{L}\topp{\Cb t,\frac{\Ca}t} ,\quad
P_{x}^{s,t }  =\mathbb{L}\topp{\Cb t,\frac{s \UU(x)}t,\frac{s}{t \UU(x)}}
\end{equation}
with $\UU(x)$ satisfying $\UU(x)+1/\UU(x)=x$; for definiteness we will choose $\UU(x)$  given by \eqref{u(z)}.

We will show that for $0<t_1\leq t_2\le \dots\le  t_d$, these functionals can be composed to define  functionals that act on  tensor products of analytic functions. \color{black}
Recall that the tensor product   $f_1\otimes\dots\otimes f_d$ is defined by
$$\left(f_1\otimes\dots\otimes f_d\right)(x_1,\dots,x_d)=\prod_{j=1}^d\,f_j(x_j).$$
Define $\pi^{t_1,\dots,t_d}$  starting with $\pi^{t_1}$ as defined in \eqref{M-Trans+}, and proceeding recursively \color{black} by
\begin{equation}\label{pii}
\pi^{t_1,\dots,t_d }\left[\bigotimes_{j=1}^d f_j\right]:= \pi^{t_1,\dots,t_{d-1} }\left[\left(\bigotimes_{j=1}^{d-2}f_j\right)\otimes \wt f_{d-1} \right],\end{equation}
where
\begin{equation}
    \wt f_{d-1}(x)=f_{d-1}(x)P_{x}^{t_{d-1},t_d } [f_d].\end{equation}
\color{black}

Our main result is the following representation for the generating function \eqref{GG}.
\begin{theorem}%
\label{Thm1-} Fix $0<\A<1$ and  $\Ca ,\Cb >0$ such that $\A \Ca  <1$. Let
$0<t_1\leq t_2 \dots\leq t_N$ be such that $\A t_N<1$ and $\A t_N^2 \Cb <1$.
Then
  \begin{equation}\label{S2I}
  \mathcal{G}\topp{\A,\Ca ,\Cb }_N(\vv t)
  =\pi^{t_1,\dots,t_N}\left[\bigotimes_{j=1}^N\,\frac1{\q_{\A t_j}}\right].
\end{equation}
\end{theorem}
Our proof does not rely on \cite{Barraquand-2024-integral} and  is  based on  recursions \eqref{Gc2Ga} and \eqref{Ic2Ia}.
The proof of Theorem \ref{Thm1-} appears in Section \ref{Sec:Proof-T1}.

\arxiv{
\textbf{Remark.}
Under additional restrictions on the parameters, functionals \eqref{M-Trans} are have integral representations that already appeared in the literature, see Proposition \ref{pi2Int} and Proposition \ref{Pst2Int}.
}

We remark that if the argument  $\vv t$ of the generating function has only $d$ distinct  components, as in \eqref{GB-k=1},
where $d=1$,  then  formula \eqref{S2I} can be expressed using fewer functionals.
To state this formula we need additional  notation. Fix integers  $n_1,\dots,n_d\geq 1$ such that $n_1+\dots+n_d=N$  and real numbers $0<t_1<\dots<t_d$ satisfying  $\A t_d, \A\Ca, \A\Cb t_d^2<1$.
Denote
\begin{equation}\label{tn}
  \vv t^{\vv n}=(\underbrace{t_1,\dots,t_1}_{n_1},\underbrace{t_2,\dots,t_2}_{n_2},\dots, \underbrace{t_d,\dots, t_d}_{n_d}).
\end{equation}
Then, due to the property   $P_{x}^{t,t}[f]=f(x)$,
formula \eqref{S2I}  takes the following form:
\begin{equation}
    \label{G2I-r} \mathcal{G}\topp{\A,\Ca,\Cb}_N(\vv t^{\vv n})
\;=\pi^{t_1,\dots,t_d}\left[\bigotimes_{j=1}^d \frac{1}{\q_{\A t_j}^{n_j}}\right].
\end{equation}
The explicit formula for the right hand side is in Proposition \ref{P:comp}.
When  the boundary parameters \( \Ca, \Cb \) lie in the interval \( (0,1) \) and
 $\Ca<t_1<\dots<t_d<\min\{1/\A,1/\Cb\}$  the last expression becomes a $d$-multiple integral  with respect to the  $d$-variate density of the free Askey-Wilson process, and in this form expression \eqref{G2I-r} can be read out from the  contour integral   expression   discovered by Barraquand \cite[Theorem 1.10]{Barraquand-2024-integral}. Our Theorem \ref{Thm1-} was inspired by this remarkable   result. %

An application of %
formula \eqref{G2I-r}  gives:
\begin{theorem}\label{Prop:Poiss} Fix $\la>0$.
 Suppose    that the parameters of the geometric LPP  vary with the  width $N$ of the strip in one of the following ways:
 \begin{enumerate}[(a)]
     \item $\Ca=\Ca(N)=\sqrt{N/\la}$, $\A=\sqrt{\la/(N+1)}$, $\Cb$ is fixed.
     \item $\Ca$ is fixed, $\A=1/N^{\theta+1}$, $\Cb=\Cb(N)=\la N^\theta$ for some $\theta>0$.
 \end{enumerate}
 Then as $N\to\infty$, process  $(\vv L_1(\floor{Nx}))_{0\leq x\leq 1}$ converges in finite dimensional distributions  to  the Poisson process
$(\vv N_x)_{0\leq x\leq 1}$   with parameter $\la$.

\end{theorem}
 The analytic proof of this result appears  in Section \ref{Sec:Proof:Poiss} and is unexpectedly involved.

\medskip

 The paper is organized as follows. In Section \ref{Sec:FreeAW}, we introduce the fundamental concept of free Askey--Wilson functionals. This section is comprehensive and begins with a discussion of monic free Askey--Wilson polynomials and their complex-valued moment functionals. We  extend these moment functionals from polynomials to analytic functions through contour integration. Additionally, we explore the use of polynomial approximation to derive key identities. Integral representations for specific parameter choices are derived, and we show how the functionals can be composed by allowing two of the parameters to depend on a variable upon which the subsequent functional acts. Asymptotic expansions required for subsequent sections are also discussed. In Section \ref{Proofs-Intro}, we provide proofs for  Theorem \ref{Thm1-}, Proposition \ref{Prop1}, Theorem \ref{Thm-PhD}, Theorem \ref{Prop:Poiss}, and several additional lemmas. Appendix \ref{Sec:PoTL} contains elementary but lengthy proofs for two auxiliary results needed in Section \ref{Proofs-Intro}, including the proof of the key recurrence relation \eqref{Gc2Ga} for the generating function. Appendix \ref{Sec:I-rep-of-AW} outlines sufficient conditions on the parameters that allow for an explicit representation of \( \pi_y^t \) as an integral with respect to the signed Askey--Wilson measure and provides a representation of \( P_{s,t}^{x,y} \) as a complex Askey--Wilson integral.

\section{Free Askey--Wilson  functionals}
\label{Sec:FreeAW}
\subsection{Free Askey--Wilson polynomials}
Our starting point is the four-parameter family of the "free" Askey--Wilson polynomials with $q=0$ which we introduce   as the (complex) linear combinations of the Chebyshev polynomials of the second kind.

For $n=-1,0,2\dots$, let $U_{n}$ be the $n$-th monic Chebyshev polynomial (also known as the Vieta–Fibonacci polynomial) defined by
\begin{equation}\label{U-def}
 U_n(y)=\frac{\sin ((n+1)\theta)}{\sin \theta}, \quad y=2\cos \theta.
\end{equation}
These polynomials satisfy  the three step recurrence  relation
\begin{equation}\label{U-rec}
  y U_n(y)=U_{n+1}(y)+U_{n-1}(y), \quad n\in\ZZ,
\end{equation}
 with initial conditions $U_{0}\equiv 1$ and $U_{-1}\equiv 0$.

\begin{definition}\label{Def:AW0}
  Let $a,b,c,d$ be real or complex such that $abcd\ne 1$.
Denote by $s_j$ the $j$-th elementary symmetric function reduced (specialized) to four non-zero arguments,  $s_j=e_j(a,b,c,d,0,0,\dots)$. Thus $s_0=1$, $s_1=a+b+c+d$, $s_2=ab+ac+ad+bc+bd+cd$, $s_3=abc+abd+acd+bcd$, $s_4=abcd$.
The {\em free Askey--Wilson polynomials} in indeterminate $y$ are
\begin{eqnarray}
W_0(y;a,b,c,d)&=&U_0(y) \label{W0} \\
W_1(y;a,b,c,d)&=&(1-s_4) U_1(y)+(s_3-s_1)U_0(y)  %
\label{W1}\\
  W_2(y;a,b,c,d)&=&U_2(y)-s_1U_1(y)+(s_2-s_4)U_0(y)  %
  \label{W2}\\
W_n(y;a,b,c,d)&=&U_n(y)-s_1U_{n-1}(y)+s_2U_{n-2}(y)-s_3U_{n-3}(y)+s_4U_{n-4}(y) \mbox{ for } n\geq 3. \label{Wn}
\end{eqnarray}

\end{definition}

Note that $W_n\in\CC[y]$ has complex coefficients as we allow complex parameters $a,b,c,d$.  However, if parameters $a,b,c,d$ are real, or if complex, appear in conjugate pairs, then
$s_1,s_2,s_3,s_4$ are real so $W_n\in\RR[y]$ is a real polynomial for all $n=0,1,\dots$, and then the free Askey--Wilson polynomials $\{W_n\}$ as defined here  are essentially the (monic)  Askey--Wilson polynomials with parameter $q=0$, compare \cite[formula (4.29)]{Askey-Wilson-85}, who discuss polynomials $W_n(x/2)$.

Since formula \eqref{U-def}   defines polynomials $U_n$ for all $n\in\ZZ$,   an equivalent but less  explicit version of Definition \ref{Def:AW0} is
\begin{equation}
  \label{U2W}W_n(y;a,b,c,d)=U_n(y)-s_1U_{n-1}(y)+s_2U_{n-2}(y)-s_3U_{n-3}(y)+s_4U_{n-4}(y), \quad n=1,2,\dots
\end{equation}
with  $W_0(y;a,b,c,d)=U_0=1$.
(Formula \eqref{U-def} gives $U_{-n}(y)=-U_{n-2}(y)$  for $n\in\ZZ$.)
\arxiv{
To see that the formulas are equivalent, we only need to  check the two exceptional entries \eqref{W1} and \eqref{W2}.
Recalling that $U_{-n}(y)=-U_{n-2}(y)$, we have $W_1(y;a,b,c,d)=(1-s_4) U_1(y)+(s_3-s_1)U_0(y) =U_1(y)-s_1 U_0(y)+s_2 U_{-1}(y)-s_3 U_{-2}(y)+s_4 U_{-3}(y)$
and $ W_2(y;a,b,c,d)=U_2(y)-s_1U_1(y)+(s_2-s_4)U_0(y)   =  U_2(y)-s_1U_1(y)+s_2U_0(y)-s_3 U_{-1}(y)+s_4 U_{-2}(y)$.   Assumption $abcd\ne 1$   ensures that polynomial $W_1$ is of degree 1.
}
 \subsection{The functionals}\label{Sec:Mon=mF}

We introduce a four-parameter family of functionals $\mathcal{L}\topp{a,b,c,d}$ defined on polynomials in the variable $y$. The action of $\mathcal{L}\topp{a,b,c,d}$ is first specified on monomials by prescribing the values $\mathcal{L}\topp{a,b,c,d}[y^n]$, and then extended to arbitrary polynomials by linearity. The functional is defined recursively, starting with the initial condition $\mathcal{L}\topp{a,b,c,d}[1] = 1$, together with the requirement that $\mathcal{L}^{a,b,c,d}[W_n(y)] = 0$ for all $n \geq 1$.
Since  $abcd\ne 1$, this determines $\mathcal{L}\topp{a,b,c,d}:\CC[y]\to\CC$  uniquely, and  $\mathcal{L}\topp{a,b,c,d}$ is  invariant under the permutations of the parameters $a,b,c,d$.
In view of permutation invariance,    we can put any zeros at the end and  shorten the notation by omitting trailing zeros, writing $\mathcal L$ for $\mathcal{L}\topp{0,0,0,0}$, $\mathcal{L}\topp{a}$  for $\mathcal{L}\topp{a,0,0,0}$, $\mathcal{L}\topp{a,b}$  for $\mathcal{L}\topp{a,b,0,0}$ and so on.
The above mimics  the usual approach to the  moment functional, although we consider complex-valued functionals and complex polynomials.
In particular,  $\mathcal{L}$ is the moment functional of the Chebyshev polynomials, and is given by an explicit
 integral
\begin{equation}
    \label{L0}
    \mathcal{L}[f]=\frac{1}{2\pi}\int_{-2}^2 f(y) \sqrt{4-y^2} \d y.
\end{equation}

Definition  \eqref{U2W}  gives a simple linear recurrence relation for the sequence
 $\la_n:=\mathcal{L}\topp{a,b,c,d}[U_n]$, $n=0,1,\dots$. We have
 \begin{equation}
   \label{JU-rec}\la_n=s_1 \la_{n-1}-s_2\la_{n-2}+s_3\la_{n-3}-s_4 \la_{n-4}, \quad n\geq 2.
 \end{equation}
 (Note that since $\la_{-k}=-\la_{k-2}$, the recurrence relation for  $n=2$ takes an irregular form, see \eqref{W1} and \eqref{W2}.)

Using the initial conditions $\la_{-2}=-1$, $\la_{-1}=0$, $\la_{0}=1$, and $\la_1 =\frac{s_1-s_3}{1-s_4}$,  which corresponds to the irregular term \eqref{W1}, recurrence \eqref{JU-rec} can be solved explicitly.    If all the parameters are distinct, the solution is
\begin{multline}\label{Jnabcd}
\mathcal{L}\topp{a,b,c,d}[U_n]=\frac{(1-b c) (1-b d) (1-c d)}{(a-b) (a-c) (a-d) (1-a b c d)} a^{n+3}+\frac{(1-a c) (1-a d) (1-c d)}{(b-a) (b-c) (b-d) (1-a b c d)} b^{n+3}
\\+ \frac{(1-a b) (1-a d) (1-b d) }{(c-a)
   (c-b) (c-d) (1-a b c d)} c^{n+3}+ \frac{ (1-a b) (1-a c) (1-b c)
 }{(d-a) (d-b) (d-c) (1-a b c d)} d^{n+3}.
\end{multline}
This formula is valid for all $n\geq -2$, but we only need it  for $n\geq 0$.
In this paper we focus on the case with only three non-zero parameters. Then
\begin{equation}\label{Jabc}
\mathcal{L}\topp{a,b,c}[U_n] = \frac{ (1-b c)}{(a-b) (a-c)}a^{n+2}+\frac{(1-a c) }{(b-a) (b-c)}b^{n+2}+\frac{(1-a b)}{(c-a)
   (c-b)}c^{n+2},
\end{equation}
which is valid for $n\geq -1$ and distinct parameters $a,b,c$.

We also note that for $n\geq -1$ we have
\begin{equation}
  \label{Jab00}
  \mathcal{L}\topp{a,b}[U_n]=\sum_{k=0}^n a^{n-k}b^k=\begin{cases}
    \frac{a^{n+1}-b^{n+1}}{a-b} & a\ne b \\
    (n+1)a^n & a= b,
  \end{cases}
\end{equation}
which is  the  case when $c=0$ in \eqref{Jabc}, and for $n\geq 0$ we have
\begin{equation}\label{Ja00}
\mathcal{L}\topp{a}[U_n] = a^n,
\end{equation}
which is the case when $b=0$ in \eqref{Jab00}.

These formulas give the following identity: %
\begin{equation}\label{Jabc2JaJbJc}
\mathcal{L}\topp{a,b,c}  =
\frac{ a^2(1-b c)}{(a-b) (a-c)}\mathcal{L}\topp{a}+\frac{b^2(1-a c) }{(b-a) (b-c)}\mathcal{L}\topp{b}+\frac{c^2(1-a b)}{(c-a)
   (c-b)}\mathcal{L}\topp{c}.
\end{equation}
For repeated parameters,  we obtain the following limiting case of \eqref{Jabc}:
\begin{equation}\label{Jabb}
\mathcal{L}\topp{a,b,b}[U_n]=\frac{b^{n+1} (n (b-a) (1-a b)+b+a^2b-2a )}{(a-b)^2}+\frac{\left(1-b^2\right) a^{n+2}}{(a-b)^2}.\end{equation}
 This gives  the following identity
\begin{equation}   \label{Dabb}
\mathcal{L}\topp{a,b,b}%
=\frac{ a^2(1-b^2)}{(a-b)^2}\mathcal{L}\topp{a} +\frac{b(b-a-a (1-a b)) }{(a-b)^2}\mathcal{L}\topp{b}+\frac{ b^2  (1-a b)}{b-a }\frac{\partial}{\partial b} \mathcal{L}\topp{b}.
\end{equation}
We also have
   \begin{equation}\label{Jaaa}
\mathcal{L}\topp{a,a,a}[U_n]=\frac{1}{2} (n+1) a^n \left(\left(1-a^2\right) n+2\right).
\end{equation}
\arxiv{which gives
\begin{equation}
    \label{Daaa}
    \mathcal{L}\topp{a,a,a} %
=\mathcal{L}\topp{a} + a(2-a^2)  \frac{\partial}{\partial a} \mathcal{L}\topp{a} + \frac{a^2}{2} (1-a^2)  \frac{\partial^2}{\partial a^2} \mathcal{L}\topp{a}
\end{equation}

}

\arxiv{
(To get \eqref{Jab00} directly from recurrence \eqref{JU-rec}, we  need to discard the initial condition $\la_{-2}=-1$.)

With $n\geq 0$, formulas \eqref{Jabb}  and  \eqref{Jab00} work also  for $b=0$, where $W_n=U_n-a U_{n-1}$, $n=0,1,2,\dots$ and  $\mathcal{L}\topp{a}[U_n]=a^n$. }
 From recursion \eqref{JU-rec}, or from explicit solutions, we  obtain the following version of  \cite[Theorem 4.1]{KS-2010}.
\begin{proposition} For   $z\in\CC$ such that   $|z|<1/\max\{1,|a|,|b|,|c|,|d|\}$,
 we have
\begin{multline}\label{KS-Thm4.1}
\sum_{n=0}^\infty z^n \mathcal{L}\topp{a,b,c,d}[U_n] =\frac{z ( s_1  s_4- s_3)}{( 1-s_4) (1-a z) (1-b z) (1-c z) (1-d
   z)}+\frac{ 1+s_4 z^2}{(1-a z) (1-b z) (1-c z) (1-d z)}
   \\
   =
  \frac{ z (abcd(a+b+c+d)- abc-abd-acd-bcd)  }{(1-a b c d)(1-az)(1-bz)(1-cz)(1-dz)}+\frac{1+z^2 abcd}{(1-az)(1-bz)(1-cz)(1-dz)}.
\end{multline}
In particular,  for $|z|<1/\max\{1,|a|,|b|,|c|\}$  we get
\begin{equation}\label{KS-Thm4.1abc}
 \sum_{n=0}^\infty z^n \mathcal{L}\topp{a,b,c}[U_n] =  \frac{1-a b c  z}{(1- a z)(1-bz)(1-cz)}.
\end{equation}
\end{proposition}

\begin{proof}
If all the parameters are distinct, \eqref{KS-Thm4.1} follows immediately from \eqref{Jnabcd}.
If some of the parameters coincide, we rely on continuity.
To obtain formula \eqref{KS-Thm4.1abc}, we set \(d = 0\) in \eqref{KS-Thm4.1}.
\end{proof}
 Since the generating function $\sum_{n=0}^\infty z^n U_n(y)$ is  $1/\q_z(y)$,
formula \eqref{KS-Thm4.1abc} can be interpreted heuristically as the expression for $\mathcal{L}\topp{a,b,c}\left[\q_z^{-1} \right]$.
\subsubsection{Reduction formulas}
We now provide formulas that lower the number of parameters   of functionals. For every polynomial $f$, the following reduction formula holds:
 \begin{equation}
     \label{Redukcja}
      \mathcal{L}\topp{a,b,c}\left[\q_c f\right] =(1- a c)(1-b c)\mathcal{L}\topp{a,b}[f].
 \end{equation}
In particular, by setting \( a = 0 \) and \( a = b = 0 \) respectively, and relabeling the parameters, we obtain
\begin{equation}
     \label{Redukcja2}
      \mathcal{L}\topp{a,b}\left[\q_b f\right] =(1- a b)\mathcal{L}\topp{a}[f], \quad  \mathcal{L}\topp{a}\left[\q_af\right] =\mathcal{L}[f],
 \end{equation}
 where $\mathcal{L}$ is given by \eqref{L0}.
 (The reverse operation of adding a parameter in Lemma \ref{L:anti-reduce} is more intricate.)
\arxiv{
Combining the above formulas, we get
\begin{equation}\label{Jck-*}
  \mathcal{L}\topp{a,b,c}\left[\q_a\q_b\q_c f\right] =(1- a c)(1-b c)\mathcal{L}\topp{a,b}[\q_a\q_bf]
=(1- a b)
(1- a c)(1-b c)\mathcal{L}\topp{0}[f]
\end{equation}
}

To verify   formulas   \eqref{Redukcja}   and \eqref{Redukcja2}, we take $f=U_n$, $n=0,1,\dots$, in which case this is a calculation based on recurrence relation \eqref{U-rec} and formulas \eqref{Jabc},
 \eqref{Jabb}, \eqref{Jaaa}. %
\subsection{Extension to analytic functions}
Under appropriate additional assumptions on $a,b,c,d$,  functional $\mathcal{L}\topp{a,b,c,d}$ extends  from polynomials to more general functions and is given as an integral with respect to probability measure, as in \eqref{L0}. Under fewer assumptions on $a,b,c,d$, it is given as an integral with respect to a signed measure, compare \cite[Section 2.4]{WWY-2024}.  However,  without any additional conditions on the parameters, the functional $\mathcal{L}\topp{a,b,c,d}$ extends to analytic functions, but the resulting functional is not always given as an  integral, see Remark \ref{R:J-rep}. The extension essentially can be based on formula \eqref{KS-Thm4.1}; in application to geometric LPP   one of the parameters  vanishes,   so instead of developing the general theory we will
rely on much simpler formula
\eqref{KS-Thm4.1abc}.
\subsubsection{Joukowsky (Zhukovskii) map $w+1/w$}
The transformation $z= w+1/w$ maps the unit disk $|w|<1$ injectively onto  $\CC\setminus[-2,2]$.  It transforms a circle $|w|=\rho<1$, oriented counter-clockwise,   into the ellipse
\begin{equation}
  \label{ellipse}
  \gamma_\rho=\left\{z=x+\i y: \frac{x^2}{\left(\tfrac1\rho+\rho\right)^2}+\frac{y^2}{\left(\tfrac1\rho-\rho\right)^2}=1\right\},
\end{equation}
oriented clockwise.
The transformation maps the unit circle $|w|=1$ continuously onto $[-2,2]$, covering it twice, and it is an increasing function $[1,\infty)\to [2,\infty)$. It also maps  both  the exterior  and the interior of the unit circle injectively onto  $\CC\setminus[-2,2]$  (see Fig. \ref{Fig:w+1/w}).
 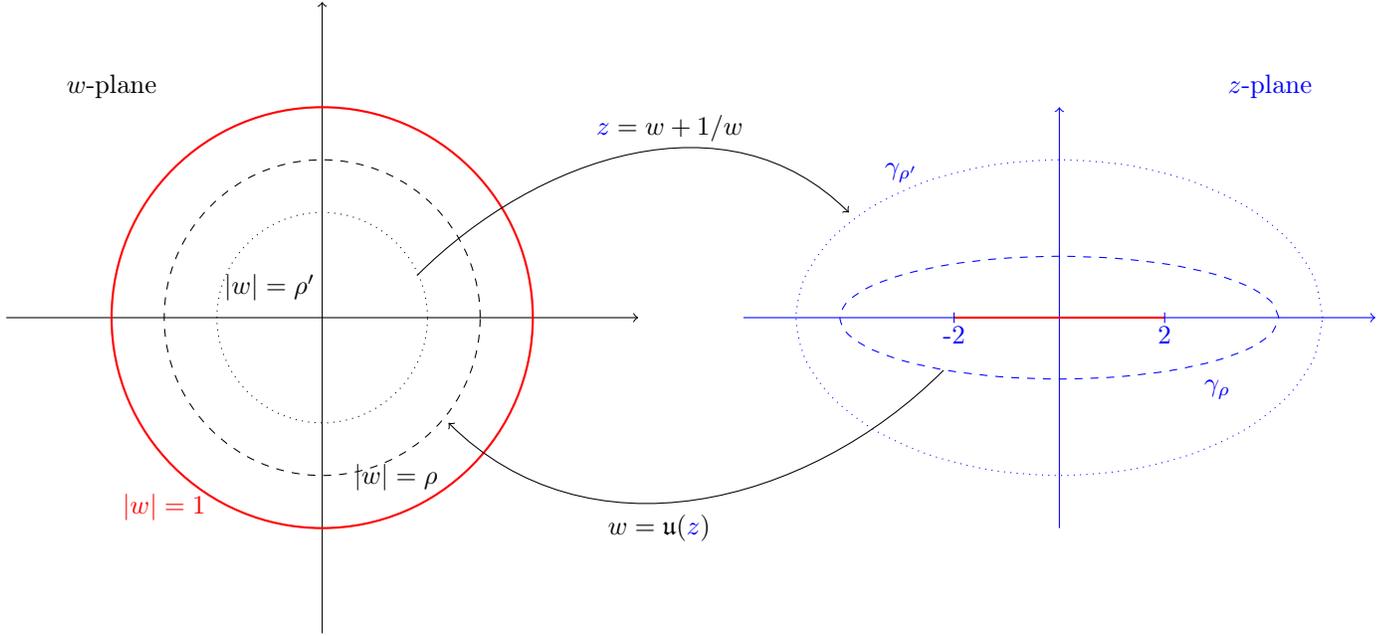
\begin{figure}[hbt]
  \begin{tikzpicture}[scale=1.4]
   \draw[->] (0,-3) to (0,3);
    \draw[->] (-3,0) to (3,0);
   \draw[thick, red] (0,0) circle (2);
   \node  at (-1.5,-1.8) {\color{red}$|w|=1$};
    \node [below] at (.7,-1.3) {$|w|=\rho$};
   \draw[dotted] (0,0) circle (1);
   \node [below] at (-.5,.5) {$|w|=\rho'$};
    \node [above] at (5.5,1.2) {\color{blue}$\gamma_{\rho'}$};
     \node [below] at (8.5,-.5) {\color{blue}$\gamma_{\rho}$};

   \draw[dashed] (0,0) circle (1.5);
  \node [above] at (9,2) {\color{blue}$z$-plane};
  \node [above] at (-2,2) {$w$-plane};
    \draw[->,blue] (7,-2) to (7, 2);
 \draw[->,blue] (4,0) to (10,0);
 \draw[-,thick,red] (6,0) to (8,0);
 \draw[-,blue] (6,-.05) to (6,.05);
  \draw[-,blue] (8,-.05) to (8,.05);
 \node [below] at (6,0) {\color{blue}-2};
 \node [below] at (8,0) {\color{blue}2};

  \draw[scale = 1,domain=0:360,smooth,variable=\x,dotted,blue] plot ({(7+2.5*cos(\x))},{1.5*sin(\x)});

    \draw[scale = 1,domain=0:360,smooth,variable=\x,dashed,blue] plot ({(7+25*cos(\x)/12)},{7*sin(\x)/12});
  \draw[->] (.9,.4) to [out=45,in=135] (5,1);
\node   at (3.3,1.8) {${\color{blue}z}=w+1/w$};

   \draw[<-,thin] (1.2,-1) to [out=-45,in=-135] (5.9,-.5);
\node   at (3.2,-2) {$w=\UU({\color{blue}z})$};
  \end{tikzpicture}
  \caption{Mapping $z=w+1/w$ and one of its two inverses $w=\UU(z)$. Note that with $\rho'<\rho$,
  ellipse $\gamma_{\rho}$ lies inside  $\gamma_{\rho'}$. }
  \label{Fig:w+1/w}
\end{figure}

 It is convenient to introduce notation $D_\rho$ for the image of the %
  annulus $\rho<|w|\leq 1$
 under the transformation $z=w+\tfrac1w$; $D_\rho$ is the %
 open set
 enclosed by the ellipse $\gamma_\rho$   defined by the  equation \eqref{ellipse} with $0<\rho<1$.

The inverse mapping
\begin{equation}
  \label{u(z)}
  \UU(z)=\tfrac12(z-\sqrt{z^2-4})
\end{equation}
maps injectively $\CC\setminus[-2,2]$ onto the unit disk $|w|<1$.  We extend $\UU$   to $z\in[-2,2]$ by taking $\UU(x)= e^{\i \theta}$ for $x=2\cos \theta$ with $0\leq \theta \le \pi$.
Then $\UU$ is the right inverse of the Joukowsky map: $\UU(z)+\frac{1}{\UU(z)}=z$ for all $z\in\CC$.
The second inverse mapping, $1/\UU(z)=(z+\sqrt{z^2-4})/2$  maps injectively $\CC\setminus[-2,2]$ onto the exterior of the unit circle $|w|=1$.
(This is a standard  material on mappings of the complex plane, see e.g. \cite[Vol. 1, Ch. 10, Section 51]{Markushevich-77} or \cite[Section 4.2]{Ahlfors}.)

 In what follows we adopt the standard  convention that the integrals  over the closed curves are in counter-clockwise orientation.
\arxiv{In particular, if $0<\rho<1$ then due to positive orientation convention, Joukowsky transformation $z=w+1/w$ with $\d z=(w^2-1)/w^2$ gives
\begin{equation}\label{dz2dw}
   \oint_{\gamma_\rho}  h(z)\d z=\oint_{|w|=\rho} h(w+1/w) \frac{1-w^2}{w^2}\d w,
\end{equation}
with both curves  oriented counter-clockwise.
}
\subsubsection{Extension to analytic functions}
For $0<r\leq 1$, let $\bar D_r$ be the closed region enclosed by the ellipse $\gamma_r$, with the convention that $\bar D_1=[-2,2]$.
Let
\begin{equation}
    \label{Def-A}  \mathcal{A}_{r}=\left\{f:\bar D_r\to\CC: \mbox{$f$ extends analytically to an open  set $D_{\rho'}$ for some $0<\rho'<r$, where $\rho'=\rho'(f)$ depends on $f$}\right\}.
\end{equation}
We note that if  $r'<r\leq 1$ then $\bar D_r\subset\bar{D}_{r'}$ and if  $f\in\mathcal{A}_{r'}$, then  its restriction to $\bar D_{r}$ is in $\mathcal{A}_{r}$. %
Therefore, when $r'<r$ we can treat $\mathcal{A}_{r'}$ as a subset of $\mathcal{A}_{r}$.

 \arxiv{ In order to clarify the extension procedure, we start with some formal and heuristic formulas.
With $|w|=\rho<1/R$, we rewrite \eqref{KS-Thm4.1abc} as a formal expression
\begin{equation}
    \label{***}
      \mathcal{L}_y\topp{a,b,c}\left[\frac{1}{w+1/w-y}\right]=\frac{w(1-a b c  w)}{(1- a w)(1-b w)(1-c w)}.
\end{equation}

and let $z=w+1/w\in\gamma_\rho $. We get
$$
 \mathcal{L}_y\topp{a,b,c}\left[\frac{1}{z-y}\right]=\frac{\UU(z)(1-a b c  \UU(z))}{(1- a \UU(z))(1-b \UU(z))(1-c \UU(z))}.
$$
This formal expression and the Cauchy formula   $f(x)=\frac{1}{2\pi\i}\oint_C \frac{f(z)}{z-x}\d z$ form a basis for our extension $ \mathbb{L}\topp{a,b,c}$  of $\mathcal{L}\topp{a,b,c}$ to analytic functions.
The heuristic is that for real $a,b,c$, and $f$ that are real on the real line, the value of the functional  $\mathbb{L}\topp{a,b,c}[f]$ should only depend   on the values of $f$ on the interval $[-2,2]$ and at some isolated points of the form $a+1/a$ when $a>1$, so the closed contour $C$ should enclose such points.  The ellipse $\gamma_\rho$ with $\rho<1/R$ works as such a contour.

{If $R=1$, then any ellipse with $\rho<1$ encloses $[-2,2]$. If $R=a>1$ and $\rho<1/R$ then point $a+1/a$ is enclosed by $\gamma_\rho$ because
$\UU(a+1/a)=1/a> \UU(\rho+1/\rho)=\rho$. }
}

\begin{definition}
  \label{Def-J-ext0}   For $a,b,c\in\CC$, let
  \begin{equation}
  \label{R} R=R(a,b,c) %
  =\max\{1,|a|,|b|,|c|\}.
\end{equation}
For $f\in\mathcal{A}_{1/R}$,  we define the {\em free Askey--Wilson} functional
\begin{equation}
  \label{J[f]++}
  \mathbb{L}\topp{a,b,c}[f]=\frac{1}{2\pi\i}\oint_{\gamma_\rho}
\frac{f(z)\UU(z)(1-a b c  \UU(z))}{(1- a \UU(z))(1-b \UU(z))(1-c \UU(z))}\,  \d z,
\end{equation}
with arbitrary $\rho\in(\rho',1/R)$ where $\rho'$ is from  \eqref{Def-A} and
the ellipse $\gamma_\rho$ is oriented counterclockwise.

\end{definition}
We remark that the choice of $\rho'=\rho'(f)$ and $\rho$ may depend on $f$ but the value $ \mathbb{L}\topp{a,b,c}[f]$ of the functional depends only on $f\Big|_{\bar D_{1/R}}$ and does not depend on the choice of $\rho, \rho'$. Moreover,  $\mathbb{L}\topp{a,b,c}:\mathcal{A}_{1/R}\to\CC$ is a linear  functional.
We also note that functional $\mathbb{L}\topp{0}$ is in fact defined for all
bounded measurable functions $f:[-2,2]\to \RR$, see \eqref{L0}.
\arxiv{
To see that the value of the functional does not depend on the value of $\rho$,
suppose $\rho_1<\rho_2<1/R$ both define ellipses that lie within $D$ together with their interior regions. Since $\rho_2<1$, the ellipse $\gamma_{\rho_2}$ lies within the ellipse $\gamma_{\rho_1}$. Since $1-a \UU(z)=0$ for $z=a+1/a$ and $|\UU(z)|<1$, the poles arise only from $a,b,c$ that lie outside of the unit circle and then they all lie inside the ellipse $\gamma_{1/R}$,  which is enclosed by the ellipse $\gamma_{\rho_2}$. Therefore, there are no poles of the integrand in the regions between the ellipses $\gamma_{\rho_1}$ and $\gamma_{\rho_2}$.   The singularities of $\UU(z)$ lie in $[-2,2]$ which is within all ellipses $\gamma_\rho$. Thus, \eqref{J[f]++} gives the same value for $\rho=\rho_1$ and for $\rho=\rho_2$.
}
\begin{remark}
   Since $\UU(w+1/w)=w$ for $|w|<1$,
  change of variable $z=w+1/w$ in
    \eqref{J[f]++} with    $\d z= -(1-w^2)/{w^2} \,\d w$
  gives an equivalent expression
  \begin{equation}\label{J[f]-w}
     \mathbb{L}\topp{a,b,c}[f]=\frac{1}{2\pi\i} \oint_{|w|=\rho} \frac{f\left(w+\tfrac{1}{w}\right)(1-w^2)(1-a b c w) }{w(1-a w)(1-bw)(1-cw)}\,\d w,
  \end{equation}
   where the circle is   oriented counter-clockwise.
   \end{remark}
\arxiv{
 \begin{remarka}
   When $a,b,c$ are all distinct, we have
  \begin{multline}\label{J[f]}
  \mathbb{L}\topp{a,b,c}[f]=\frac{1}{2\pi\i}\frac{a (1-b c)}{(a-b) (a-c)}\oint_{\gamma_\rho}\frac{f(z)\d z}{1-a \UU(z)}
  +\frac{1}{2\pi\i}\frac{b (1-a c)}{(b-a)
   (b-c)}\oint_{\gamma_\rho}\frac{f(z)\d z}{1-b \UU(z)}+\frac{1}{2\pi\i}\frac{c (1-a b)}{(c-a) (c-b)}\oint_{\gamma_\rho}\frac{f(z)\d z}{1-c \UU(z)}.
\end{multline}
An equivalent expression is
   \begin{multline}\label{J[f]+}
  \mathbb{L}\topp{a,b,c}[f]=\frac{1}{2\pi\i}\frac{a (1-b c)}{(a-b) (a-c)}\oint_{|w|=\rho}\frac{f(w+1/w)(1-w^2)\d w}{w^2(1-a w)}
  \\+\frac{1}{2\pi\i}\frac{b (1-a c)}{(b-a)
   (b-c)}\oint_{|w|=\rho}\frac{f(w+1/w)(1-w^2)\d w}{w^2(1-b w)}+\frac{1}{2\pi\i}\frac{c (1-a b)}{(c-a) (c-b)}\oint_{|w|=\rho}\frac{f(w+1/w)(1-w^2)\d w}{w^2(1-c w)}.
\end{multline}
 \end{remarka}
}
\arxiv{
  Note that in general, functional $\mathbb{L}$ takes complex values. However, our focus is on analytic functions $f$   with real coefficients in the series expansion so that $f(x)$ is real for real $x$.
  One of such functions is  $f(z)=1/(1+a^2-a z)^k$ for $0<a<1$, which is real-analytic on the half-plane   $\re (z) < a+1/a$ (which incudes the domain $D_a$ enclosed by  the ellipse $\gamma_a$). Proposition \ref{Prop:J-real} addresses this issue.
}
We now verify that the functional $\mathbb{L}\topp{a,b,c}$ is indeed an extension of the functional $\mathcal{L}\topp{a,b,c}$.
\begin{proposition}
  \label{P:ext}  If $p$ is a polynomial then
  \begin{equation}\label{J=J}
    \mathbb{L}\topp{a,b,c}[p] =\mathcal{L}\topp{a,b,c}[p].
  \end{equation}
\end{proposition}
\begin{proof}
  It suffices to verify \eqref{J=J} for $p=U_n$, the $n$-th Chebyshev polynomial. We use the well know formula
  \begin{equation}\label{U(w+1/w)}
  U_n\left(w+\tfrac 1w\right)=\begin{cases}
    \frac{w^{n+1}-1/w^{n+1}}{w-1/w} & w\ne \pm 1 \\
    n+1 & w=1 \\
    (-1)^n(n+1) & w=-1
  \end{cases}
  \end{equation} which can be checked from the three step recurrence relation \eqref{U-rec}.

  Consider the integral
  \eqref{J[f]-w}.     For $n=1,2,\dots$ we have
   \begin{multline}\label{JabcUn}
 \mathbb{L}\topp{a,b,c}[U_n]= \frac{1}{2\pi\i} \oint_{|w|=\rho}\frac{U_n\left(w+\frac 1w\right)(1-w^2)(1-abcw)}{w(1-a w)}\,\d w
  \\= \frac{1}{2\pi\i} \oint_{|w|=\rho}\frac{\frac{1/w^{n+1}-w^{n+1}}{1/w-w}(1-w^2){(1-a b c w)}}{w(1-a w)(1-bw)(1-cw)}\,\d w
  = \frac{1}{2\pi\i} \oint_{|w|=\rho}\frac{(1/w^{n+1}-w^{n+1})(1-abcw)}{(1-a w)(1-bw)(1-cw)}\,\d w\\
  =\frac{1}{2\pi\i} \oint_{|w|=\rho}\frac{1/w^{n+1}(1-abcw)}{(1-a w)(1-bw)(1-cw)}\,\d w-\frac{1}{2\pi\i} \oint_{|w|=\rho}\frac{w^{n+1}(1-abcw)}{(1-a w)(1-bw)(1-cw)}\,\d w.
  \end{multline}
The second integral vanishes as  the integrand %
is an analytic function  on the disk $|w|<|\rho|$. In the first integral,   we use the fact that with $\max\{|a|,|b|,|c|\}\rho\le \rho R<1$ the geometric series $\sum_{k\ge 0}\,a^k w^k$, $\sum_{k\ge 0}\,b^k w^k$ and $\sum_{k\ge 0}\,c^k w^k$ converge. Expanding the integrand into these series we obtain
$$
\mathbb{L}\topp{a,b,c}[U_n]=\frac{1}{2\pi\i} \oint_{|w|=\rho}\frac{{ (1-a b c w)}}{w^{n+1}}\sum_{r=0}^\infty {w^r}\left(\sum_{\substack{i+j+k=r\\
  i,j,k\ge 0} }\,a^ib^jc^k\right)\,\d w
 =\sum_{\substack{i+j+k=n\\
  i,j,k\ge 0} }\,a^ib^jc^k   - a b c\sum_{\substack{i+j+k=n-1\\
  i,j,k\ge 0} }\,a^ib^jc^k, \quad n\geq 1.
$$
Clearly, $\mathbb{L}\topp{a,b,c}[U_0]=1$.

We thus see that the generating function
\begin{multline*}
   \sum_{n=0}^\infty z^n\mathbb{L}\topp{a,b,c}[U_n]=
 \sum_{n=0}^\infty \sum_{\substack{i+j+k=n\\
  i,j,k\ge 0} }\,(az)^i(bz)^j(cz)^k- abc z\sum_{n=1}^\infty \sum_{\substack{i+j+k=n-1\\
  i,j,k\ge 0} }\,(az)^i(bz)^j(cz)^k
\\ =(1-abc z)\sum_{i=0}^\infty(az)^i\sum_{j=0}^\infty(bz)^j\sum_{k=0}^\infty(cz)^k=\frac{1-abcz}{(1-az)(1-bz)(1-cz)}
\end{multline*}
matches the generating function \eqref{KS-Thm4.1abc}.  This shows that
\eqref{J=J} holds   for $p=U_n$, $n=0,1,\dots$.
\end{proof}
The following result will allow us to derive several properties of $\mathbb{L}\topp{a,b,c}$ by elementary calculations with polynomials.
\begin{proposition}\label{P:poly-appr}
Fix    $a,b,c\in\CC $ and let $R$ be given by \eqref{R}.
If $f\in \mathcal{A}_{1/R}$, then there exists a sequence $(p_n)$ of polynomials such that $p_n\to f$ uniformly on $\bar D_{1/R}$.
Furthermore, if a sequence of polynomials  $p_n\to f$ uniformly on $\bar D_{1/R}$, then $\mathbb{L}\topp{a,b,c}[f]=\lim_{n\to\infty}    \mathcal{L}\topp{a,b,c}[p_n]$.
\end{proposition}

\begin{proof}Choose  $\rho'<\rho<1/R$ with $\rho'=\rho'(f)$ as in \eqref{Def-A}. By Mergelyan's theorem \cite[Theorem 20.5]{Rudin-87}, we can find a sequence of polynomials $p_n$ such that $p_n\to f$ uniformly on $\bar{D}_{\rho}$ and hence also on $\bar{D}_{1/R} \subset\bar{D}_{\rho}$.  Since the other factors in the integrand in \eqref{J[f]++} are continuous on the ellipse $\gamma_\rho$, we get
$\mathbb{L}\topp{a,b,c}[p_n]\to\mathbb{L}\topp{a,b,c}[f]$.  Identity \eqref{J=J}  ends the proof.
\end{proof}
 We remark that in view of Proposition \ref{P:poly-appr}, identity \eqref{Redukcja} implies
  \begin{equation}
     \label{Redukcja+}
      \mathbb{L}\topp{a,b,c}\left[\q_c f\right] =(1- a c)(1-b c)\mathbb{L}\topp{a,b}[f].
 \end{equation}
Another useful property is continuity in parameters.
\begin{proposition}%
\label{P:cont-abc}   Fix $R_0\geq 1$,  and consider $a,b,c$ in the
disk $\bar D=\{|z|\leq R_0\}$.
If  $f\in\mathcal{A}_{1/R_0}$, then  the mapping $(a,b,c)\mapsto \mathbb{L}\topp{a,b,c}[f]$ is continuous on $\bar D^3$.
\end{proposition}
\begin{proof} %
   With $\rho'<1/R_0$   from \eqref{Def-A}, choose  $\rho$ such that  $\rho'<\rho<1/R_0$ so that $f$ is continuous  on $\gamma_\rho$.  Then with $R=R(a,b,c)$ given by \eqref{R}, we have $R\leq R_0$, so $\rho<1/R(a,b,c)$ for all $a,b,c\in\bar D\color{black}$. %
    Since $|w|=\rho$, the integrand in \eqref{J[f]-w} is a continuous function of $a,b,c\in \bar D$  and is  a  bounded  function of variable  $w$ on the circle $|w|=\rho$; the integrand is bounded   by
    $$\sup_{|w|=\rho} |f(w+1/w)| \frac{(1+\rho^2)(1+R_0^3 \rho)}{\rho (1-\rho R_0)^3}.$$
    Thus, the conclusion holds by the dominated convergence theorem.
\end{proof}

The following  confirms heuristic interpretation of formula \eqref{KS-Thm4.1abc} as $\mathcal{L}\topp{a,b,c}\left[\q_z^{-1}\right]$ mentioned before.
 \begin{proposition}\label{P:KS-rig} Recall \eqref{R}. If $0<|v|<1/R$, where $R$ is given by \eqref{R}, then
    \begin{equation}\label{KS-T4.1abc++}
       \mathbb{L}\topp{a,b,c}\left[\q_v^{-1}\right]
       =\frac{1-a b c  v}{(1- a v)(1-b v)(1-c v)}.
     \end{equation}
 \end{proposition}
\arxiv{ \begin{proof}[Combinatorial proof of Proposition \ref{P:KS-rig}]
     The generating function $\q_v^{-1}(y)=\sum_{n=0}^\infty v^n U_n(y)$ gives $q_v^{-1}\in \mathcal{A}_{1/R}$  with explicit polynomials approximation.
     Indeed, every  point $y\in D_{|v|}$   can be written as $y=w+1/w$
 with $1\geq |w|>|v|$.  We see
that $$\q_v^{-1}(w+1/w)=\frac{1}{(1-v w)(1-v /w)}=\frac{w}{(1- v w)(w-v )}$$ has no zeros in the denominator, as $|v w|=|v||w|\leq |v|<1/R\leq 1$  and $|v/w|=|v|/|w|<1$. So $1/\q_v$ is analytic on $D_{|v|}\supset D_{1/R}$.

     From Proposition  \ref{P:poly-appr} we obtain
     $$\mathbb{L}\topp{a,b,c}\left[\q_v^{-1}\color{black}\right]=\sum_{n=0}^\infty v^n \mathcal{L}\topp{a,b,c}[U_n] $$
so \eqref{KS-T4.1abc++} follows from \eqref{KS-Thm4.1abc}.

  \end{proof}
 }
 \begin{proof}
 The function  $1/\q_v$ is an analytic function of variable $y$ on the domain
 { $D_{|v|}$ } enclosed by the ellipse $\gamma_{|v|}$. Indeed, every  point $y\in D_{|v|}$   can be written as $y=w+1/w$
 with $1\geq |w|>|v|$.  We see
that
\begin{equation}
    \label{q-fac}
    \frac{1}{\q_v(w+1/w)}=\frac{1}{(1-v w)(1-v /w)}=\frac{w}{(1- v w)(w-v )}
\end{equation}
  has no zeros in the denominator, as $|v w|=|v||w|\leq |v|<1/R\leq 1$  and $|v/w|=|v|/|w|<1$.

Choose $\rho$ such that $|v|<\rho<1/R$. We  apply the definition \eqref{J[f]++} in the equivalent form  \eqref{J[f]-w}.
We get %
\[
     \mathbb{L}\topp{a,b,c}\left[\q_v^{-1}\color{black}\right] =
     \frac{1}{2\pi\i} \oint_{|w|=\rho} \frac{ (1- a b c w)(1-w^2)}{(1-vw)(w-v) (1-a w)(1-b w)(1-c w)} \,\d w .
\]
We note  that  the last three factors in the  denominator do not have singularities: $|aw|=|a| \rho<|a|/R<1$ and the same holds for $b$ and $c$. Since $|v w|\leq \rho^2<1$, we see that  in the   disk $|w|<\rho$  the integrand has only one pole at $w=v$ (recall that  $|v|<\rho$).  We get
$$
   \mathbb{L}\topp{a,b,c}\left[\q_v^{-1}\color{black}\right] = \textnormal{Res}_{w=v}\left( \frac{ (1- a b c w)(1-w^2)}{(1-vw)(w-v) (1-a w)(1-b w)(1-c w)}\right)
=\frac{ (1- a b c v)}{ (1-a v)(1-b v)(1-c v)},
$$
 which is  \eqref{KS-Thm4.1abc}.
 \end{proof}

\subsection{Askey--Wilson functionals and complex Askey--Wilson measures} In \cite{WWY-2024}, the authors
  worked out { an integral} representation of the moment functional for  the  family of Askey--Wilson polynomials with arbitrary parameter $q\in[0,1)$,
  using a signed measure of finite variation.
For our purposes, it is necessary to allow complex parameters, which gives rise to integral representations involving complex measures under suitable assumptions. We note that in general, free Askey--Wilson functional is not given as an integral, see formula \eqref{J2c2a} in Remark \ref{R:J-rep} for the simplest example.

The following formula together with \eqref{Jabc2JaJbJc}  gives a useful expression for $\mathbb{L}\topp {a,b,c} [f]$ when the parameters are distinct.
 \begin{lemma}\label{L:anti-reduce}
  Recall notation \eqref{Jck-q}.  Fix $a\in\CC$ such that $|a|\ne 1$ and let $R=\max\{1,|a|\}$. Then for  $f\in\mathcal{A}_{1/R}$,  we have
    \begin{equation}
        \label{Jck2}
          \mathbb{L}\topp {a} [f]= \mathbf{1}_{|a|>1} \left(1-\tfrac{1}{a^2}\right) f\left(a+\tfrac1a\right)+
          \frac{1}{2\pi}\int_{-2}^2\,f(y)\, \frac{\sqrt{4-y^2}}{\q_a(y)}\,\d y
          .
    \end{equation}
\end{lemma}
\begin{proof} Writing $y=2 \cos \theta =e^{\i \theta}+e^{-\i \theta}$ and $a=|a| e^{\i \alpha}$, from \eqref{q-fac} we see that
$$
|\q_a(y)|=(1+|a|^2-2 |a| \cos(\alpha+\theta))((1+|a|^2-2 |a| \cos(\alpha-\theta)))
\geq (1-|a|)^4>0.
$$
Therefore, when using Proposition \ref{P:poly-appr}, we can pass  to the limit under the integral  on the right hand side of \eqref{Jck2}. Hence, it suffices to show that \eqref{Jck2} holds for polynomials, and we verify this formula for $f=U_n$, the $n$-th Chebyshev polynomial. The left hand side is $a^n$.
Noting that for $y\in[-2,2]$ we have  $|U_n(y)|\leq U_n(2)=n+1$, on the right hand side of \eqref{Jck2} we get one of the two  uniformly convergent series
\begin{equation}
    \label{Cheb2q}
    \frac{1}{\q_a(y)}=\begin{cases}
    \sum_{k=0}^\infty a^k U_k(y) & |a|<1 \\ \\
    \frac{1}{a^2}\sum_{k=0}^\infty \frac{1}{a^k} U_k(y) & |a|>1.
\end{cases}
\end{equation}
For $\mathcal{L}$, the moment functional of  polynomials  $\{U_n\}$ given by \eqref{L0}, we have $\mathcal{L}[U_kU_n]=\delta_{k,n}$. Thus we get \color{black}
$$\mathcal{L}\left[\tfrac{U_n}{\q_a}\right]=\begin{cases}
    a^n & |a|<1 \\
    1/a^{n+2} & |a|>1.
\end{cases}$$
This yields the identity for $|a|<1$. If $|a|>1$, we need to take into account the first term on the right hand side of \eqref{Jck2}. In this case, we use identity \eqref{U(w+1/w)}
 to complete the proof.
\end{proof}

We now give explicit integral expressions for $\mathbb{L}\topp{a,b,c}[f]$ for the cases we need.
\begin{proposition} \label{P:J-rep}
{ In each of the cases below, $f\in\mathcal{A}_{1/R}$, where $R=R(a,b,c)$ is given by \eqref{R}. }
Then we have the following explicit representations for the Askey--Wilson  functional. %
\begin{enumerate}[(i)]

  \item If $a,b,c$ are distinct  and %
   $|a|,|b|,|c|\ne 1$, then
   \begin{multline}\label{J2c1a}
  \mathbb{L}\topp{a,b,c}[f]=\frac{(1-ab)(1-ac)(1-bc)}{2\pi}\int_{-2}^2 f(y)\frac{\sqrt{4-y^2}}
{ \q_a(y)\q_b(y)\q_c(y)\color{black}} \,\d y  %
  \\+ \mathbf 1_{|a|>1} \frac{\left(a^2-1\right)
    (1-b c)}{(a-b)
   (a-c)}f\left(a+\tfrac{1}{a}\right) +\mathbf 1_{|b|>1}\frac{\left(b^2-1\right) (1-a c)
  }{(b-a) (b-c)} f\left(b+\tfrac{1}{b}\right)
   +
   \mathbf 1_{|c|>1} \frac{\left(c^2-1\right) (1-a b)
  }{(c-a) (c-b)}  f\left(c+\tfrac{1}{c}\right).
  \end{multline}
    In particular,  if $a,b,c$ are real or include a complex conjugate pair, then $\mathbb{L}\topp{a,b,c}[f]$ is an integral with respect to the signed measure of bounded variation on $\RR$.

        \item  If $|a|,|b|,|c|<1$ (and some may be equal), then \eqref{J2c1a} holds. That is,
   \begin{equation}\label{J2c0a}
    \mathbb{L}\topp{a,b,c}[f]=
    \frac{(1-ab)(1-ac)(1-bc)}{2\pi}\int_{-2}^2 f(y)\frac{\sqrt{4-y^2}}
    {\q_a(y)\q_b(y)\q_c(y)}\,\d y  %
    \color{black}.
  \end{equation}
  In particular if $a,b,c$ are real or include a complex conjugate pair, then $\mathbb{L}\topp{a,b,c}[f]$ is an integral with respect to the absolutely continuous   compactly supported probability measure on $\RR$.
  \item Suppose $a,b,c\in\CC$ are such that $ab=1$.  Then
       \begin{equation}\label{J2m0b-conj}
      \mathbb{L}\topp{a,b,c}[f]=f(a+\tfrac1a).
    \end{equation}
   In particular, if $b=\bar a$ with $|a|=1$, or if  $a,b$ are real such that $ab=1$, then \eqref{J2m0b-conj} is an integral with respect to a degenerate probability measure on $\RR$.
\arxiv{
\item If $a=b=c$   and $|a|\ne 1$,  then
     \begin{multline}\label{J3aaa}
  \mathbb{L}\topp{a,a,a}[f]=  \frac{(1-a^2)^3}{2\pi}\int_{-2}^2 f(y)\frac{\sqrt{4-y^2}\d y}{(1+a^2-ay)^3}
  \\+\mathbf{1}_{|a|>1} \left( 2 f(a+\tfrac1a)-\frac{\left(a^2-1\right)^2 \left(a^2+1\right)}{a^3} f'(a+\tfrac1a)-\frac{(a^2-1)^4}{2a^4} f''(a+1/a)\right).
  \end{multline}
  Note that  by \eqref{J2m0b-conj} we have  $ \mathbb{L}\topp{1,1,1}[f]= f(2)$ not $2f(2)$, as the integral term in
  \eqref{J3aaa} contributes  $-f(2)$ to the limit $a\to 1$.
  \item If $a=\pm 1$, $b=-a$, $|c|\ne 1$, then \eqref{J2c1a} holds. That is,
      \begin{equation}\label{J2m0a=1b=-1}
      \mathbb{L}\topp{a,b,c}[f]= \frac{1-c^2}{\pi}\int_{-2}^2 f(y) \frac{\d y}{(1+c^2-c y)\sqrt{4-y^2}} +2  f(c+\tfrac1c)\mathbf{1}_{|c|>1}.
    \end{equation}
    Proofs of formulas \eqref{J3aaa} and \eqref{J2m0a=1b=-1} are in Section \ref{Sec:An-Pr}.
    }
\end{enumerate}
\end{proposition}
\begin{remark}\label{R:J-rep}
  If $a=b\ne c$   and $a,c$ are not on the unit circle, then
   \begin{multline}\label{J2c2a}
 \mathbb{L}\topp{a,a,c}[f]= %
 \frac{(1-a^2)(1-a c)^2}{2\pi}\int_{-2}^2 f(y)\frac{\sqrt{4-y^2}}
 {\q_a(y)\q_c(y)}\,\d y  %
  \\+\mathbf{1}_{|a|>1}  \left(\frac{
   (a-c)^2+(1-a c)^2
   }{(a-c)^2} f(a+\tfrac1a)+ \frac{\left(a^2-1\right)^2   (1-a c) }{a^2(a-c)} f'(a+\tfrac1a) \right) + \mathbf{1}_{|c|>1} \frac{(c^2-1)(1-a^2)}{(c-a)^2}f(c+\tfrac1c).
  \end{multline}
 Formula \eqref{J2c2a} with $|a|>1$  gives an example of the free Askey--Wilson functional
    that is not represented as an integral with respect to a signed measure of finite variation.
\end{remark}
\arxiv{ Proof of formula \eqref{J2c2a} is at the end of Section \ref{Sec:An-Pr}.
Section \ref{Sec:An-Pr} has also complex-analytic proofs of the other identities in Proposition \ref{P:J-rep}.
}
\begin{proof}[Proof of \eqref{J2c1a} in Proposition \ref{P:J-rep}(i)] Let $R$ by defined by \eqref{R} and $f$ be analytic in $
D_{\rho}$ for some $0<\rho<1/R$.  Since $a,b,c$ are distinct, %
by Proposition \ref{P:poly-appr} and \eqref{Jabc2JaJbJc} we have
$$
\mathbb{L}\topp{a,b,c} [f] =
\frac{ a^2(1-b c)}{(a-b) (a-c)}\mathbb{L}\topp{a}[f]+\frac{b^2(1-a c) }{(b-a) (b-c)}\mathbb{L}\topp{b}[f]+\frac{c^2(1-a b)}{(c-a)
   (c-b)}\mathbb{L}\topp{c}[f],$$
Therefore, from  \eqref{Jck2}, expressed in notation \eqref{L0},  we get
\begin{multline}
   \mathbb{L}\topp{a,b,c} [f] =
\mathcal{L} \left[\frac{ a^2(1-b c)}{(a-b) (a-c)}\frac{f}{\q_a}+\frac{b^2(1-a c) }{(b-a) (b-c)} \frac{f}{\q_b}+\frac{c^2(1-a b)}{(c-a)
   (c-b)} \frac{f}{\q_c}\right]
   \\+
  \frac{ (a^2-1)(1-b c)}{(a-b) (a-c)}\mathbf{1}_{|a|>1}  f\left(a+\tfrac1a\right)+\frac{(b^2-1)(1-a c) }{(b-a) (b-c)} \mathbf{1}_{|b|>1}   f\left(b+\tfrac1b\right) +\frac{(c^2-1)(1-a b)}{(c-a)
   (c-b)}\mathbf{1}_{|c|>1}  f\left(c+\tfrac1c\right)
   \\=
  (1-a b)(1-bc)(1-ac) \mathcal{L} \left[\frac{f}{\q_a\q_b\q_c}\right]  +
  \frac{ (a^2-1)(1-b c)}{(a-b) (a-c)}\mathbf{1}_{|a|>1}  f\left(a+\tfrac1a\right)
  \\+\frac{(b^2-1)(1-a c) }{(b-a) (b-c)} \mathbf{1}_{|b|>1}   f\left(b+\tfrac1b\right) +\frac{(c^2-1)(1-a b)}{(c-a)
   (c-b)}\mathbf{1}_{|c|>1}  f\left(c+\tfrac1c\right),
\end{multline}
where  in the last step, in the argument of $\mathcal{L}$ %
we collected together the terms in the partial fraction decomposition
 $$\frac{(1-a b)(1-bc)(1-ac)}{\q_a\q_b\q_c}=\frac{ a^2(1-b c)}{(a-b) (a-c)}\,\frac{1}{\q_a}+\frac{b^2(1-a c) }{(b-a) (b-c)} \,\frac{1}{\q_b}+\frac{c^2(1-a b)}{(c-a)(c-b)}\,\frac{1}{\q_c}.$$
   Formula \eqref{L0} gives \eqref{J2c1a}.

\end{proof}

\begin{proof}[Proof of \eqref{J2c0a} in  Proposition \ref{P:J-rep}(ii)]
If $|a|,|b|,|c|<1$ are distinct then the formula holds by \eqref{J2c1a} and \eqref{R} gives $R=1$. By Proposition \ref{P:cont-abc} we can pass to the limit as $a\to b$ on the left hand side of \eqref{J2c0a} and also on the right hand side under the integral.
\end{proof}

\begin{proof}[Proof of \eqref{J2m0b-conj} in Proposition \ref{P:J-rep}(iii)] Note that if $ab=1$ then
$\mathbb{L}\topp{a,b,c}[f]$ does not depend on $c$ (except through the domain.) Thus  by  Proposition \ref{P:poly-appr}, it is enough to verify the formula for Chebyshev polynomials. If $a\ne 1/a$, {  i.e. $a\ne\pm 1$},
$$
\mathbb{L}\topp{a,1/a,c}[U_n]=\mathbb{L}\topp{a,1/a}[U_n]\stackrel{\eqref{J=J}}{=}\mathcal{L}\topp{a,1/a}[U_n]
\stackrel{\eqref{Jab00}}{=}\frac{a^{n+1}-(1/a)^{n+1}}{a-1/a}
\stackrel{\eqref{U(w+1/w)}}{=}U_n\left(a+\tfrac1a\right).
$$
Formulas \eqref{Jab00} and \eqref{U(w+1/w)}  give the same identity $\mathbb{L}\topp{a,b,c}[U_n]=U_n\left(a+\tfrac1a\right)$ also in the two exceptional cases $a=\pm 1$.
\end{proof}

 \subsection{Composing the functionals} %
Our next goal is to  consider  parameters which  depend on a complex variable so that we can  introduce composition of functionals. To accomplish this, we will carefully track the domains $D_\rho$ on which the arguments of the functionals are analytic.

 Recall function $z=w+1/w$ and its two inverses $\UU(z)=(z-\sqrt{z^2-4})/2$ and $(z+\sqrt{z^2-4})/2=1/\UU(z)$ on $\CC\setminus[-2,2]$,
with  $\UU$ extended to $z\in[-2,2]$ by  $\UU(x)= e^{\i \theta}$ for $x=2\cos \theta$ with $0\leq \theta \le \pi$.
Thus  $\UU:\CC\to \{|z|<1\}\cup \{z=e^{\i \theta}: 0\leq \theta\leq \pi\}$ is the right-inverse of Joukowsky map on $\CC$.
Recall that $D_\rho$ denotes the domain enclosed by the ellipse $\gamma_\rho$ with equation \eqref{ellipse}.

 \begin{proposition} \label{Pr:pre-Markov} Fix $b\geq 0$ and  $0\leq r\leq 1$. %
 Let $0<\rho'<1$ be such that $\rho' b<1$.
 Suppose that $f$  is  analytic on  $D_{\rho'}$. Then
  $$ g(x)=\mathbb{L}\topp{b,r \UU(x),r/\UU(x)}[f]$$
  is a well-defined  analytic function of  complex variable $x$  on
   $D_{r\rho'}$.
  \end{proposition}

\begin{proof}  The statement is trivial if $r=0$. If $r=1$ then  \eqref{J2m0b-conj} gives
\(g(x)=f(\UU(x)+1/\UU(x))=f(x)\),
which proves the conclusion for $r=1$.
It remains  consider the case $0<r<1$.

Choose $\rho$ and $\rho''$ such that  $0<\rho'<\rho<\rho''<1$ with $\rho'' b<1$.
If $x\in D_{r\rho''}$, then $x=u+1/u$ with $r \rho ''< |u|\leq 1$.
Thus \eqref{R} applied with $a=r \UU(x)$ and $c=r /\UU(x)$ gives
$R=R(x)=\max\{1,b,r/|u|\}$ (recall that  $|a|=r |u|\leq r<1$). Since $b<1/\rho''<1/\rho$ and  $r/|u|<1/\rho''<1/\rho$, the parameter $\rho$  that we chose upfront, does not depend on $x\in D_{r\rho''}$  and satisfies  $\rho\in(\rho',1/R(x))$ for all  $x\in D_{r\rho''}$. Therefore, for all $x\in D_{r\rho''}$ we have
  \begin{multline}\label{JrarG+}
  \mathbb{L}\topp{b ,r\UU(x),r/\UU(x)}[f] =
\frac{1}{2\pi\i}\oint_{|w|=\rho} f(w+1/w) \frac{\left(1-b w r^2\right)(1-w^2)}{(1- b w) (1-w
   r /\UU(x)) (1-w r \UU(x))}\frac{\d w}{w}
   \\ \stackrel{\eqref{q-fac}}{=} \frac{1}{2\pi\i}\oint_{|w|=\rho} f(w+1/w)
   \frac{\left(1-b w r^2\right)(1-w^2)}{(1- b w) \,\q_{wr}(x)}\frac{\d w}{w}.
\end{multline}

 We now invoke a complex analysis result cited as Lemma \ref{M:Th,17.19}, which we apply with
   \[\Phi(x,w):=  f(w+1/w)
   \frac{\left(1-b w r^2\right)(1-w^2)}{w(1- b w)\,\q_{wr}(x)},\]
the curve $\mathcal{C}$ being the circle $|w|=\rho$,  and the domain $\mathcal{D}=D_{r\rho''}$. Since $D_{r\rho'}=\bigcup_{\rho''\in(\rho',1)}D_{r\rho''}$, this concludes the proof.
\end{proof}
  \begin{lemma}[{\cite[Vol. 1, Section *87, Theorem 17.19]{Markushevich-77}}]\label{M:Th,17.19}

   Let $\mathcal{D}$  be a domain in the complex $x$-plane and let $\mathcal{C}$ be a rectifiable curve in the complex $w$-plane. Suppose $\Phi(x,w)$ is a function of two complex variables with the following properties:

   \begin{enumerate}
       \item For every $w\in \mathcal{C}$, function $\Phi(x,w)$ is analytic in $x$ on the domain $\mathcal{D}$
       \item For every $x\in \mathcal{D}$, function $\Phi(x,w)$ is continuous in $w$ on the curve $\mathcal{C}$
       \item The family of functions
       $
       \{\Phi(x,w): w\in \mathcal{C}\}
       $
       is uniformly bounded on compact subsets of $\mathcal{D}$.
   \end{enumerate}
   Then $f(x)=\oint_{\mathcal{C}} \Phi(x,w)\,\d w$ is analytic on $\mathcal{D}$.
   \end{lemma}
\arxiv{ We now verify the assumptions of Lemma   \ref{M:Th,17.19}:
\begin{enumerate}
  \item  If $x\in D_{r\rho''}$ then $x=u+1/u$ with $r\rho''<|u|\leq 1$.
Thus $(1+w^2 r^2 - w r x)=(1-r w u)(1-w r/u) \ne 0$ for $|w|=\rho$, as $|w r u|<1$ and $|w r/u|=r\rho/|u|<\rho/\rho''<1$. So $1/(1+w^2 r^2 - w r x)$ as a function of variable $x$, is analytic on the domain $D_{r\rho'}$ enclosed by the ellipse $\gamma_{r\rho''}$.
 \item
 As a product of continuous functions, function $\Phi(x,w)$ is continuous on the circle $|w|=\rho$ for $x\in D_{r\rho''}$.
\item   By the previous estimates,
$ \sup\{ |\Phi(x,w)|: \; |w|=\rho, x\in \bar D_{r\rho''}\}<\infty$ as $1/|1-bw|\leq 1/(1-b\rho)$, $\sup_{|w|=\rho}|f(w+1/w)|<\infty$,  and
\[
    \sup_{x\in\bar D_{r\rho''}, |w|=\rho}\frac{1}{|1+w^2 r^2-w r x|}=\sup_{|w|=\rho,r\rho''\leq |u|\leq 1}\frac{1}{|1-w r u| \; |1-w r/u|}
    \leq \frac{1}{(1-r \rho)(1-\rho/\rho'')}.
\]

\end{enumerate}
}

We now consider composition of functionals \eqref{M-Trans+}. However,
we will need to apply \eqref{M-Trans+}  with  varying $\Cb$  in Proposition \ref{P:M-F}, so we need to make the dependence on this parameter explicit.
\color{black}
We therefore denote:
 \begin{equation}\label{M-Trans}
   \pi^{t;\Cb }  =\mathbb{L}\topp{\Cb t,\frac{\Ca }{t}},
P_{x}^{s,t;\Cb }  =\mathbb{L}\topp{\Cb t,s \frac{\UU(x)}{t},\frac{s}{t \UU(x)}},
\end{equation}
and we denote by $\pi^{t_1,\dots,t_d;\Cb }$  the expression defined in \eqref{pii}.

\begin{proposition}\label{P:comp}
Suppose   $0<\A<1$, $\Ca ,\Cb >0$, $\Ca  \A<1$, $0<t_1\leq t_2\leq \dots\leq  t_d<1/\A$, and $\A \Cb  t_d^2<1$. If functions $f_j$ are analytic on $D_{\A t_j}$ then  the composition \eqref{pii} is well defined.

Furthermore, if $0=t_0<t_1 < t_2 < \dots < t_d$ are distinct, then there exist constants $0 < \rho_1 < \rho_2 < \dots < \rho_d < 1$ satisfying $\rho_j > \A t_j$, $t_{j-1}\rho_j < t_j \rho_{j-1}$, $\Cb t_j \rho_j < 1$, and $\Ca \rho_1 / t_1 < 1$, such that
\begin{multline}\label{Mult-int}
    \pi^{t_1,\dots,t_d;\Cb }\left[\bigotimes_{j=1}^d f_j\right]
\\=   \frac{1}{(2 \pi \i)^d}
   \oint_{|w_1|=\rho_1} \dots \oint_{|w_d|=\rho_d} \frac{1 }{1-\tfrac{\Ca w_1}{t_1}}
   \prod_{j=1}^{d}  \frac{f_j\left(w_j+\tfrac{1}{w_j}\right)(1-w_j^2)\left(1-\Cb \frac{w_j t_{j-1}^2}{t_j}\right)}{(1-\Cb w_j t_j)\q_{\frac{t_{j-1}w_j}{t_j}}\left(w_{j-1}+\frac1{w_{j-1}}\right)}\,\frac{\d w_d}{w_d}  \dots \frac{\d w_1}{w_1}.
\end{multline}
\end{proposition}

\begin{proof}
    In this proof we rewrite definition \eqref{pii} as follows: given $f_j$ which are  analytic on $D_{\A t_j}$, $j=1,\dots,d$ we introduce auxiliary functions
    $g_1,\dots,g_d$ by taking $g_d=f_d$ and
  \begin{equation}
      \label{g-rec-0}  g_{j}(x):=f_{j}(x) P_{x}^{t_{j}, t_{j+1}; \Cb}[g_{j+1}], \quad j=d-1,d-2\dots,1.
  \end{equation}
Proposition~\ref{Pr:pre-Markov} applied to $P_{x}^{t_{j},t_{j+1};\Cb}$ with $b = \Cb t_{j+1}$ and $\rho' = \A t_{j+1}$ (so that $b \rho' = \A \Cb t_{j+1}^2 < 1$) implies that the function $g_{j}$ is analytic on $D_{\A t_{j}}$, so the  recursion
\eqref{g-rec-0} can proceed. The recursion ends with function
$g_1$ that is analytic on $D_{\A t_{1}}$. In particular, $g_1$ is in the domain of $\pi^{t_1;\Cb}$ (as $(\A t_1)(\Ca/t_1)<1$ and $(\A t_1)(\Cb t_1)<1$) and     we obtain
  $$\pi^{t_1,\dots,t_d;\Cb }\left[\bigotimes_{j=1}^d f_j\right]:=\pi^{t_1}[g_1].$$
 This shows that  the composition \eqref{pii} is well defined.

To verify formula \eqref{Mult-int}, we apply \eqref{J[f]-w} at each step of the iteration  \eqref{g-rec-0}.    To do so, it is necessary to carefully track the parameters $\rho' < \rho < \rho''$ at each stage. For the $j$-th step, we denote these parameters by $\rho_j' < \rho_j < \rho_j''$, $j = 1, \dots, d$.

To define them, we choose constants
\[
\A = \A_{d+1}'' < \A_d < \A_d'' < \A_{d-1} < \dots < \A_2 < \A_2'' < \A_1<\A_1''=1,
\]
sufficiently close to $\A$ and to each other, so that
$\A_1 \Ca ,\A_1 t_d, \A_1 \Cb t_d^2<1 $ and $ \A_j t_j < \A_{j+1} t_{j+1}$ for all $j = 1, \dots, d-1.$
\arxiv{Indeed, choose $\A_1<1$ so that the first three conditions hold and  let $q=\max_{j\leq d-1}t_j/t_{j+1}$. Then $q<1$ so given $\A_j$ for $j=1,\dots,d-1$, we can choose $\A_{j+1}\in(q \A_j,\A_j)$. }

For each $j = 1, \dots, d$, we define
\[
\rho_j = \A_j t_j, \quad \rho_j'' = \A_j'' t_j, \quad \text{and} \quad \rho_j' =  \A_{j+1}''t_j.
\]
It is clear that  we constructed $0 < \rho_1 < \rho_2 < \dots < \rho_d < 1$ such that $\rho_j > \A t_j$, $t_{j-1}\rho_j < t_j \rho_{j-1}$, $\Cb t_j \rho_j < 1$, and $\Ca \rho_1 / t_1 < 1$.
\arxiv{Indeed,
$\rho_j<\rho_{j+1}<1$ is a consequence of  $\A_jt_j<\A_{j+1}t_{j+1}<\A_1t_d<1$. Inequality $\rho_j>\A t_j$ is a consequence of $\A_j>\A$.
Clearly, $t_{j-1}\rho_j < t_j \rho_{j-1}$ is equivalent to $\A_j<\A_{j-1}$. $\Ca \rho_1 / t_1$ is $\Ca\A_1<1$.
Finally, $\Cb t_j \rho_j < 1$ follows from
$\Cb t_j^2 \A_j<\Cb t_d^2\A_1<1$
}
We also have
$\A t_j < \rho_j' < \rho_j < \rho_j''$, and for $j \geq 2$, the relation
\[
\rho_{j-1}' = \frac{t_{j-1}}{t_j} \rho_j''
\]
holds.

Starting with $j=d-1$ and $g_d(y_d) = f_d(y_d)$, which is analytic in the variable $y_d$ on $D_{\rho_d'} = D_{\A t_d}$, we proceed recursively. For $1\leq j \leq d-1$, we apply \eqref{JrarG+} with parameters $r = t_j / t_{j+1}$, $b = \Cb t_{j+1}$, and $\rho = \rho_j$ to the function $g_{j+1}$, which is analytic on $D_{\rho_{j+1}'}$, using $\rho'' = \rho_{j+1}''$ such that
\[
\rho_{j+1}'' \Cb t_{j+1} = \A_{j+1}'' \Cb t_{j+1}^2 < \A_1 \Cb t_d^2 < 1.
\]
Thus, for $j=d-1, d-2,\dots,1$,   following \eqref{g-rec-0},   we define recursively the functions %
\[
g_j(y_j) := f_j(y_j) P_{y_j}^{t_j, t_{j+1}; \Cb} \left[ g_{j+1}\right],
\]
where each $g_j$ is a function of the variable $y_j=w_j+1/w_j$ that is   given by
$$
 g_j(w_j+1/w_j) =
 \frac{1}{2\pi\i}\oint_{|w_{j+1}|=\rho_{j+1}} g_{j+1}\left(w_{j+1}+\tfrac1{w_{j+1}}\right) \frac{ (1- \Cb w_{j+1} t_{j}^2/t_{j+1})(1-w_{j+1}^2) }{(1-\Cb w_{j+1}t_{j+1})\q_{\frac{t_j}{t_{j+1}}w_{j+1}}\left(w_j+\frac1{w_j}\right)} \frac{\d {w_{j+1}}}{w_{j+1}}.
$$
From the previous step, $g_{j+1}(y_{j+1})$ is a function of variable $y_{j+1}$ analytic on $D_{\rho'_{j+1}}$  so with $0<\rho'_{j+1}<\rho_{j+1}<\rho''_{j+1}$, inspecting the proof of Proposition \ref{Pr:pre-Markov} we see that \eqref{JrarG+}, on noting that $t_{j} \rho''_{j+1}/t_{j+1}=t_j\A_{j+1}''=\rho_j'$, gives $g_j$ that is analytic on $D_{\rho_j'}$. \color{black} Thus, the recursion can proceed until we get $g_1$ which is analytic on $D_{\rho_1'}=D_{t_1 \A''_2}$.

To apply the last functional $\pi^{t_1;\Cb}=\mathbb{L}\topp{\Cb t_1,\Ca/t_1}$, we use \eqref{J[f]++} in the form \eqref{J[f]-w} with $\rho=\rho_1=\A_1t_1$ with $\rho'=\rho_1'=t_1\A_2''<\rho$, noting that
$\rho=\A_1 t_1<1/R(\Ca/t_1,\Cb t_1)$, because by construction we have:
 $\rho=\A_1 t_1<\A_1t_d<1$,
     $\rho \Ca/t_1=\A_1 \Ca<1$ and
    $\rho \Cb t_1<\A_1\Cb t_d^2<1$.
Thus
\[\pi^{t_1;\Cb}[g_1]=\frac{1}{2\pi \i} \oint_{|w_1|=\rho_1} g_1\left(w_1+\tfrac1{w_1}\right)\frac{1-w_1^2}{(1-\Ca w_1/t_1)(1-\Cb w_1 t_1)} \frac{\d w_1}{w_1}
\]
contributes the final (exterior) integral to formula \eqref{Mult-int}.

\end{proof}

\begin{proposition}\label{P:M-F}
Recall \eqref{Jck-q}. %
  Fix $s\leq t$, $0<\A<1$, $\Ca ,\Cb >0$ such that $\A t<1$, $\A \Ca <1$, $\A \Cb  t^2<1$.
\begin{enumerate}[(i)]
    \item     If   $f$ is an analytic function of variable $y$ on the domain $D_{\A s}$  enclosed by the ellipse $\gamma_{\A s}$, then
\begin{align} \label{Lq0+}
   \pi^{s;\Cb } \left[f \frac{\q_{\Cb s}}{\q_{\A s}} \right] %
   &= \frac{1-\Ca \Cb }{1-\A \Ca } \pi^{s;\A} [f].
\end{align}
\item
 If   $f$ is an analytic function of variable $y$ on the domain $D_{\A t}$,
 then for $x\in D_{\A s}$ and $s<t$ we have
 \begin{equation}   \label{Lq-step+}
 P_{x}^{s,t;\Cb } \left[ f \frac{\q_{\Cb t}}{\q_{\A t}}\right]
 =\frac{\q_{\Cb s}(x)}{\q_{\A s}(x)}   P_{x}^{s,t;\A} \left[ f \right]
\end{equation}
and the above expression is an analytic function of variable $x$ on the domain $D_{\A s}$ enclosed by the ellipse $\gamma_{\A s}$.
\end{enumerate}

\end{proposition}

\begin{proof}[Proof of (i)]
If $\Ca \Cb = 1$, then $\pi\topp{s;\Cb}[f] = f(\Cb s + 1/(\Cb s))$, so both sides of \eqref{Lq0+} vanish.
Thus, from now on, we assume $\Ca \Cb \neq 1$.

Next, observe that if $|\A|<1$ and $f$ is analytic in $D_{\A s}$, then \eqref{Lq0+} is equivalent to
\begin{equation}\label{Lq0-equiv}
  \frac{1}{1-\Ca\Cb} \, \pi\topp{s;\Cb}\left[ \q_{\Cb s}f \right] = \frac{1}{1-\A\Ca} \, \pi\topp{s;\A}\left[ \q_{\A s}f \right].
\end{equation}
Indeed, setting $g(y) := \q_{\A s}(y) f(y)$, we see that $g$ is analytic on $D_{\A s}$ if and only if $f$ is analytic on $D_{\A s}$, and applying \eqref{Lq0+} to $g$ yields \eqref{Lq0-equiv}.

Now, applying \eqref{Redukcja} with parameters $a = 0$, $b = \Ca/s$, and $c = \Cb s$, we see that the left-hand side of \eqref{Lq0-equiv} is independent of $\Cb$; that is, it remains unchanged if we replace $\Cb$ by $\A$.

\end{proof}

\begin{proof}[Proof of (ii)] Let $x \in D_{\A s}$.
First, observe that if %
$\q_{\Cb s}(x)=0$, then, see \eqref{q-fac}, we get
\[
(1 - (\Cb t) s \UU(x)/t)\left(1 - (\Cb t) s/(t \UU(x))\right) = 0,
\]
so the product of the parameter $\Cb t$ with one of the parameters of the last functional in \eqref{M-Trans}  equals one. Therefore, by \eqref{J2m0b-conj}, we have
\[
P_{x}\topp{s,t;\Cb}[f] = f\left( \Cb t + \frac{1}{\Cb t} \right),
\]
and thus both sides of \eqref{Lq-step+} vanish.

Next, we consider the case $\q_{\Cb s}(x)\ne 0$. We note that since $x \in D_{\A s}$, factorization  \eqref{q-fac} gives
\[
\q_{\A s}(x) = (1 - \A s \UU(x))\left(1 - \frac{\A s}{\UU(x)}\right) \neq 0,
\]
since $\A s |\UU(x)| \leq \A s < 1$ and $\A s/|\UU(x)| < 1$.
Therefore, %
we observe that \eqref{Lq-step+} is equivalent to
\begin{equation}\label{Lq-Step+equiv}
\frac{\mathbb{L}\topp{\Cb t, \frac{s \UU(x)}{t}, \frac{s}{t \UU(x)}}\left[ f \; \q_{\Cb t}\right]}{\q_{\Cb s}(x)}
= \frac{\mathbb{L}\topp{\A t, \frac{s \UU(x)}{t}, \frac{s}{t \UU(x)}}\left[ f \;  \q_{\A t} \right]}{\q_{\A s}(x)}.
\end{equation}

Finally, applying \eqref{Redukcja+}   with parameters $a = s \UU(x)/t$, $b = s/(t \UU(x))$, and $c = \Cb t$, we see that the left-hand side of \eqref{Lq-Step+equiv} is independent of $\Cb$ (in particular, it has the same value for $\Cb = \A$). In applying \eqref{Redukcja+}   with these $a,b,c$,  we note that
\[
(1 - a c)(1 - b c) = 1 + \Cb^2 s^2 - \Cb s x=\q_{\Cb s}(x),
\]
as required.
\end{proof}

Writing  \eqref{pii}    with explicit dependence on parameter $\Cb$,
from Proposition \ref{P:M-F} we get:
\begin{corollary}
     Suppose $0<t_1\leq t_2\leq \dots\leq  t_d$,  $0<\A<1$, $\Ca ,\Cb >0$, $\Ca  \A<1$,  and $\A \Cb  t_d^2<1$.
  If $f_j$ is analytic function on  $ D_{\A t_j}$, $j=1,\dots,d$,
     then
\begin{equation}\label{magic-formula+}
 \pi ^{t_1,\dots,t_d;\Cb } \left[
\left(\bigotimes_{j=1}^{d-1} f_j\right)\otimes\left(f_d \frac{\q_{\Cb t_d} }{\q_{\A t_d}}\right)\right]
 =
 \frac{1-\Ca \Cb }{1-\A \Ca } \pi^{t_1,\dots,t_d;\A}\left[
 \bigotimes_{j=1}^d f_j\right].
\end{equation}
\color{black}
\end{corollary}
That is, up to a simple multiplicative factor,  multiplying $f_d$  by a single factor $\q_{\Cb t_d}/{\q_{\A t_d}}$ replaces all occurrences of parameter  $\Cb $ in the composition \eqref{pii} of the Askey--Wilson functionals by parameter  $\A$.

\begin{proof}
     Formula \eqref{magic-formula+} follows from \eqref{Lq0+} and \eqref{Lq-step+} by induction on $d$.
For the base case of the induction, we use \eqref{Lq0+} with $s=t_1$, and $f=f_1$.

For the induction step, we assume that   \eqref{magic-formula+} holds for some $d\geq 1$ and arbitrary $f_j$ analytic on $D_{\A t_j}$ for $j=1,\ldots,d$.
From \eqref{pii} we obtain
\begin{equation}
    \label{magic-otimes1}
    \pi ^{t_1,\dots,t_d,t_{d+1};\Cb } \left[
\left(\bigotimes_{j=1}^{d} f_j\right)\otimes\left(f_{d+1} \frac{\q_{\Cb t_{d+1}} }{\q_{\A t_{d+1}}}\right)\right]
 =\pi ^{t_1,\dots,t_d;\Cb } \left[
\left(\bigotimes_{j=1}^{d-1} f_j\right)\otimes\tilde f_{d} \right],
\end{equation}
where
$$\tilde f_{d}(y) =f_d(y) P_{y}^{t_d, t_{d+1};\Cb}\left[f_{d+1} \frac{\q_{\Cb t_{d+1}} }{\q_{\A t_{d+1}}}\right].$$

From \eqref{Lq-step+}  with $s=t_d$ and $t=t_{d+1}$  we
see that
$$\tilde f_{d} =  \frac{\q_{\Cb t_{d}}
}{\q_{\A t_{d}}}\wt f_d^*$$
  with
$$ \wt f_d^*(y):= f_d(y) P_{y}^{t_d, t_{d+1};\A}\left[f_{d+1} \right].
$$
From \eqref{Lq-step+}  with $s=t_d$ and $t=t_{d+1}$, we see that
  function $\wt f_d^*$ is analytic on $D_{\A t_d}$.
Therefore, we can use the inductive hypothesis \eqref{magic-formula+}, and from \eqref{magic-otimes1} we obtain
\begin{multline*}
  \pi ^{t_1,\dots,t_d,t_{d+1};\Cb } \left[
\left(\bigotimes_{j=1}^{d} f_j\right)\otimes\left(f_{d+1} \frac{\q_{\Cb t_{d+1}} }{\q_{\A t_{d+1}}}\right)\right]
 =\pi ^{t_1,\dots,t_d;\Cb } \left[
\left(\bigotimes_{j=1}^{d-1} f_j\right)\otimes \left(\frac{\q_{\Cb t_{d}} }{\q_{\A t_{d}}}\wt f_d^*\right) \right]
\\= \tfrac{1-\Ca \Cb }{1-\A \Ca  } \pi^{t_1,\dots,t_d;\A}\left[
 \bigotimes_{j=1}^{d-1} f_j \otimes \wt f_d^*\right]\stackrel{\eqref{pii}}{=}\tfrac{1-\Ca \Cb }{1-\A \Ca  } \pi^{t_1,\dots,t_d,t_{d+1};\A}\left[
 \bigotimes_{j=1}^{d+1} f_j  \right].
\end{multline*}
This completes the induction step and ends the proof.
\end{proof}
\subsection{Asymptotic expansions}
In our application to geometric LPP, we require a few technical results on asymptotics. The list of cases presented is not exhaustive, but it covers all the situations that we needed. %
 \begin{proposition}     \label{P:ASE} Fix  $0<v<1/R$, where $R$ is given by \eqref{R}.
 Then, as $n\to\infty$, we have  the following asymptotics.
\begin{enumerate}[(i)]
  \item  If  $|a|,|b|,|c|<1 $, and the parameters are real or form a complex conjugate pair, then %
\begin{equation}\label{P1.17iii}
     \mathbb{L}\topp{a,b,c}\left[\q_v^{-n}\right]
     \sim
    C(a,b,c)
     \;\frac{1}{ n^{3/2} v^{3/2} (1-v)^{2n-3}},
    \end{equation}
    where $C(a,b,c)= \tfrac{(1-ab)(1-ac)(1-bc)}{2\sqrt{\pi} (1-a)^2(1-b)^2(1-c)^2}$.
\item %
If $a>1$ and  $0\leq |b|,|c|\leq 1$ or $0\leq |c|<1\leq b<a$, then %
\begin{equation}\label{P1.17i}
      \mathbb{L}\topp{a,b,c}\left[\q_v^{-n}\right] \sim  C(a,b,c) \; \left(\frac{a }{(a-v) (1-a v)}\right)^n,
    \end{equation}
    where $C(a,b,c)= \tfrac{\left(a^2-1\right) (1-b c)}{(a-b) (a-c)} $.

    \item  If $ a=b>1$,    $0\leq c<a$,  then %
    \begin{equation}\label{P1.17ii}
      \mathbb{L}\topp{a,b,c}\left[\q_v^{-n}\right] \sim
     C(a,c) \ v \;n\; \left(\frac{a }{(a-v) (1-a v)}\right)^{n+1},
    \end{equation}
    where $C(a,c)= \tfrac{\left(a^2-1\right)^2 (1-a c)}{a^2 (a-c)} $.
    \item
If $a=1$, and $0\leq |b|,|c|<1$ are real or form a complex conjugate pair, then
\begin{equation}\label{P1.17vi}
 \mathbb{L}\topp{a,b,c}\left[\q_v^{-n}\right] \sim
 \frac{1-bc}{\sqrt{\pi } (1-b)(1-c)}\;  \frac{1}{  \sqrt{n}\sqrt{v} (1-v)^{2n-1}}.
\end{equation}
 \arxiv{   \item  If $ a=b=c>1$,  then
\begin{equation}\label{P1.17v}
     \mathbb{L}\topp{a,b,c}_y\left[\frac{1}{(1+v^2-v y)^n}\right] \sim
     C(a)\;v^2\;n^2\; \left(\frac{a }{(a-v) (1-a v)}\right)^{n+3},
    \end{equation}
    where $C(a)= \tfrac{\left(a^2-1\right)^4  }{2a^5} $.
}
 \arxiv{
 \item If  $ a=\bar b, |a|=1, c<1$, then
\begin{equation}\label{P1.17iv}
 \mathbb{L}\topp{a,b,c}\left[\q_v^{-n}\right] \sim
\frac{1}{(1-v)^{2n}}.
\end{equation}

}
\end{enumerate}
 \end{proposition}

 \begin{proof}[Proof of Proposition \ref{P:ASE}]
 The proof of \eqref{P1.17iii} relies on explicit integral representation. We have  $|a|,|b|,|c|<1$ where
 all three parameters are real or there is one complex-conjugate pair.
In this case, we use  \eqref{J2c0a}
with %
$f=\q_v^{-n}$.
We have
\[  \mathbb{L}\topp{a,b,c}\left[ \q_v^{-n}\right]%
=
\frac{(1-ab)(1-ac)(1-bc)}{2\pi} \int_{-2}^2 \q_v^{-n}(y)\frac{\sqrt{4-y^2}}{\q_a(y)\q_b(y)\q_c(y)}\d y.\]

Function $\q_v^{-n}$ is increasing on the interval $[-2,2]$, with maximum $1/(1-v)^{2n}$. On the interval $[-2,1]$
function $\q_v^{-n}$ is bounded by $1/((1-v)^2+\delta v)^n$, which is of the lower order   than the contribution of the order of $n^{-3/2}/(1-v)^{2n}$ that comes from the integral over the interval $[1,2]$.
We now check that the contribution of the integral over $[1,2]$ is indeed of this higher order.
Changing the variable $y=2-u/n$ we get
\begin{multline*}
 \frac{1}{2\pi} \int_{1}^2 \frac{\sqrt{4-y^2}\,\d y}{(1+v^2-vy)^n(1+a^2-ay)(1+b^2-by)(1+c^2-cy)}
  \\=\frac{1}{2\pi n^{3/2}} \int_{0}^{ n} \frac{\sqrt{u}\sqrt{4-u/n}}{((1-v)^2+v u/n)^n((1-a)^2+a u/n)((1-b)^2+b u/n)((1-c)^2+c u/n)}\,\d u
  \\
  \sim \frac{1}{\pi n^{3/2}(1-v)^{2n}(1-a)^2(1-b)^2(1-c)^2}\int_0^\infty \sqrt{u}e^{-v u/(1-v)^2}\,\d u =
  \frac{1}{2\sqrt{\pi} n^{3/2} v^{3/2} (1-v)^{2n-3}(1-a)^2(1-b)^2(1-c)^2}.
\end{multline*}
(More formally, we re-normalize and pass to the limit under the integral, which can be justified using \eqref{E-ineq}. Compare Section \ref{SubSec:PT14}.) This establishes \eqref{P1.17iii}.

 The proof of \eqref{P1.17vi} also relies on explicit integral representation: %
 \[
\mathbb{L}\topp{1,b,c}\left[ \q_v^{-n}\right]=
 \frac{(1-b)(1-c)(1-bc)}{2\pi} \int_{-2}^2 \q_v^{-n}(y)\frac{1}{\q_b(y)\q_c(y)}\sqrt{\frac{2+y}{2-y}}\,\d y.\]
 This follows from Proposition \ref{P:cont-abc} by taking limit $a\nearrow 1$ in \eqref{J2c0a}.
As in the previous argument, only the integral over $[1,2]$ contributes to the limit.
Changing the variable $y=2-u/n$ we get
\begin{multline*}
  \frac{(1-b)(1-c)(1-bc)}{2\pi} \int_{1}^2 \frac{(1+v^2-vy)^{-n}}{ (1+b^2-by)(1+c^2-cy)}\sqrt{\frac{2+y}{2-y}}\,\d y
  \\
  =  \frac{(1-b)(1-c)(1-bc)}{2\pi \sqrt{n}} \int_{0}^n \frac{((1-v)^2+vu /n)^{-n}}{ ((1-b)^2+bu/n)((1-c)^2+cu/n)}\sqrt{\frac{4+u/n}{u}}\,\d u
  \\\sim   \frac{(1-b)(1-c)(1-bc)}{\pi \sqrt{n}(1-v)^{2n}} \int_{0}^n \frac{\left(1+\frac{vu}{(1-v)^2n}\right)^{-n} }{ ((1-b)^2+bu/n)((1-c)^2+cu/n)}\frac{\d u}{\sqrt{u}}
  \\\sim
   \frac{1-bc}{\pi \sqrt{n}(1-v)^{2n}(1-b)(1-c)} \int_{0}^\infty e^{- v u/(1-v)^2 } \frac{\d u}{\sqrt{u}}
   = \frac{1-bc}{\sqrt{\pi } \sqrt{v} \sqrt{n}(1-b)(1-c)} \frac{1}{ (1-v)^{2n-1}}.
\end{multline*}
(Here we omitted some details.)

 \medskip
 To determine the asymptotic for the remaining cases, we use the generating function

\begin{equation}
  \label{H-series}
  H(z):=\sum_{n=1}^\infty z^n \mathbb{L}\topp{a,b,c}\left[\q_v^{-n}\right].
\end{equation}
From the discussion at the end of the proof of Proposition \ref{P:KS-rig}, we know that
$1/\q_v$ is analytic on  the interior  of the ellipse $\gamma_{v}$.
 To proceed, we fix $\rho<1/R$ such that $|v|<\rho$ and   $\mathbb{L}\topp{a,b,c}\left[\q_v^{-n}\right]$ is well defined. We now consider $|z|<\rho$ small enough to ensure   convergence of the series
 $$\sum_{n=1}^\infty  \frac{z^{n-1}}{(1+v^2-v y)^n}= \tfrac{1}{(1+v^2-v y)}\frac{1}{1-\tfrac{z}{(1+v^2-v y)}}= \frac{1}{1-z+v^2-v y  }=\frac{1}{v}\frac{1}{\tfrac{1-z}{v}+v-y},$$
  uniformly in $y$ in the closed region $\bar{D}_\rho$ enclosed by the ellipse $\gamma_\rho$.  This series converges for $|z|<(1-v)^2$. Indeed, write $y\in \bar{D}_\rho$ as $y=w+1/w$ with $|w|\geq \rho$. Then
 $|1+v^2-v y|=|(1-v w)(1-v/w)|\geq (1-v|w|)(1-v/|w|)=1+v^2-v(|w|+1/|w|)\geq 1+v^2-2v$ as $|w|\leq 1$.

 Let $B(z)$ be the smaller root  of the quadratic equation  $B+1/B=\tfrac{1-z}{v}+v$. That is,
 $B(z)=\UU(\tfrac{1-z}{v}+v)$.
 Explicitly, 
 \begin{equation}
  \label{B(z)=}
   B(z)=\frac{1+v^2-z-\sqrt{\left((1-v)^2-z\right)\left( (1+v)^2-z\right)}}{2 v}.
\end{equation}

 Since $|v|<1/R$ and $B(z)\to v$ as $z\to 0$,   we can choose $\delta>0$ with $2\delta<(1-v)^2$ such that for $|z|<2\delta$, we have $|B(z)|<1/R$ and the series \eqref{H-series} converges.
Thus, noting that
\begin{equation}\label{q2B}
  \frac{1}{\tfrac{1-z}{v}+v-y}=\frac{1}{B(z)+1/B(z)-y}=B(z)\frac{1}{\q_{B(z)}(y)}
\end{equation}
we get
$$H(z)=\frac{B(z)}{v}\mathbb{L}\topp{a,b,c}\left[\frac{1}{\q_{B(z)}}\right].$$
 By Proposition \ref{P:KS-rig}
 we obtain an explicit formula
 \begin{equation}
   \label{H-expl}
   H(z)= \frac{B(z)}{v} \frac{1-a b c B(z)}{(1-a B(z))(1-bB(z))(1-c B(z))}.
 \end{equation}

 To determine the asymptotics, we write
 \begin{equation}\label{J-oint}
    \mathbb{L}\topp{a,b,c}\left[ \q_v^{-n}\right]
    =\frac{1}{2\pi \i}\oint_{|z|=\delta} \frac{H(z)}{z^{n}}\,\d z
 \end{equation}
 and analyze the asymptotics using \eqref{H-expl}. Since $|B(z)|<1$, from \eqref{H-expl} we see that the integrand does not have poles when $|a|,|b|,|c|\leq 1$. If  a parameter, say $a>1$ is real, then there is a pole at
 $z_a=(a-v)(1-v a)/a$.
 This pole is in the disk $|z|<(1-v)^2$, as $z_a>0$ and $$z_a-(1-v)^2= -\frac{(a-1)^2 v}{a}<0.$$

 In the following proofs, we further decrease $\delta$ so that $\delta<z_a$ if $a>1$.

Note that if there is another pole $z_b$ with $1<b<a$ then $z_b-z_a=v \left(a+\frac{1}{a}-b-\frac{1}{b}\right)>0$, so $z_a$ is closer to the origin.

 \begin{itemize}
     \item[]{Proof of \eqref{P1.17i}:} %
     With $a>1$, function $H(z)$ given by \eqref{H-expl}
     has a simple pole at  $z=z_a$, at the root of equation $1-a B(z)=0$ and no other poles in the disk $|z|<(1-v)^2$.
    Thus
 $\Psi(z)=(z-z_a)H(z)$ has a series expansion in the disk $|z|<(1-v)^2$  and
$$\Psi(z_a)=\frac{\left(a^2-1\right) (b c-1)}{(a-b) (a-c)}.$$
 Inserting the series  $\Psi(z)=\sum_{k=0}^\infty a_k z^k$ and $\frac{1}{z-z_a}=\frac{-1}{z_a}\sum_{j=0}^\infty  (z/z_a)^j$ into \eqref{J-oint}, we get
\begin{multline}
    \mathbb{L}\topp{a,b,c}\left[ \q_v^{-n}\right]=\frac{1}{2\pi \i}\oint_{|z|=\delta} \frac{\Psi(z)}{(z_a-z) z^{n}}\,\d z
 =\frac{-1}{2\pi \i z_a}\oint_{|z|=\delta}\,\frac1{z^n}\, \left(\sum_{k=0}^\infty a_k z^k\right)\sum_{j=0}^\infty (z/z_a)^j\,\d z\\
 =\frac{-1}{  z_a}\sum_{k+j=n-1}a_k/z_a^j=
 \frac{-1}{  z_a^{n}}\sum_{k=0}^{n-1}a_kz_a^{k}
 \sim \frac{-1}{  z_a^{n}} \Psi(z_a)=z_a^{-n}\frac{\left(a^2-1\right) (1-b c)}{(a-b) (a-c)} .
 \end{multline}
\item[]{Proof of \eqref{P1.17ii}:} Here $a=b>1$ with $0\leq c<a$. %
Then
 \begin{equation}
   H(z)=\frac{ B(z)}{v} \frac{1-a^2 c B(z)}{(1-a B(z))^2(1-cB(z))}
 \end{equation}
has a pole of order 2 at $z_a=( a -v)(1-v a)/a$.

As discussed above, the pole lies within the disk $|z|<(1-v)^2$. Since a potential pole $z_c$ with $c<a$ is further away from the origin,
$
\Psi(z):=(z-z_a)^2H(z)
$ has a series expansion at $z_a$ and
$$\Psi(z_a)=\frac{\left(a^2-1\right)^2 v (1-a c)}{a^2 (a-c)} .$$
Inserting the series  $\Psi(z)=\sum_{k=0}^\infty a_k z^k$ and $\frac{1}{(z-z_a)^2}=\frac{1}{z_a^2}\sum_{j=0}^\infty (j+1) (z/z_a)^j$ into \eqref{J-oint}, we get
\begin{multline}
    \mathbb{L}\topp{a,b,c}\left[ \q_v^{-n}\right]=\frac{1}{2\pi \i}\oint_{|z|=\delta} \frac{\Psi(z)}{(z_a-z)^2 z^{n}}\,\d z
 =\frac{1}{2\pi \i z_a^2}\oint_{|z|=\delta}\,\frac1{z^n}\, (\sum_{k=0}^\infty a_k z^k)\sum_{j=0}^\infty(j+1)(z/z_a)^j\\
 =\frac{1}{  z_a^2}\sum_{k+j=n-1}a_k(j+1)/z_a^j=
 \frac{1}{  z_a^{n+1}}\sum_{k=0}^{n-1}a_k(n-k)z_a^{k}
 \\=\frac{n}{  z_a^{n+1}}\sum_{k=0}^{n-1}a_kz_a^{k}
 -  \frac{1}{z_a^{n+1}}\sum_{k=0}^{n-1}k a_kz_a^{k}
 \sim \frac{n}{  z_a^{n+1}} \Psi(z_a)-o(n/z_a^n).
 \end{multline}

\arxiv{
 \item[Proof of \eqref{P1.17v}:] If $a=b=c>1$, then %
 \begin{equation}
   H(z)=\frac{B(z)}{v} \frac{1-a^3 B(z)}{(1-a B(z))^3}
 \end{equation}
Function $H$ has a simple pole of order 3 at $z_a=(a-v)(1-av)/a$, which is the root of the equation $aB(z)=1$. Denoting $\Psi(z)=(z-z_a)^3 H(z)$,  we readily check
  that $\Psi$ is analytic  and is given by a convergent series in the open disk $|z|<(1-v)^2$.
   We also need the value %
  $$\Psi(z_a)=\lim_{z\to z_a} \Psi(z)=\frac{\left(a^2-1\right)^4 v^2}{a^4}$$
  From \eqref{J-oint} we have
  \[
   \mathbb{L}\topp{a,b,c}_y\left[\frac{1}{(1+v^2-v y)^n}\right]=\frac{1}{2\pi \i}\oint_{|z|=\delta} \frac{\Psi(z)}{(z-z_a)^3 z^{n}}\,\d z.
  \]
  Inserting the power series  $\Psi(z)=\sum_{k=0}^\infty a_k z^k$ i  and expanding the series for
  $$\frac{1}{(z-z_a)^3}=\frac{-1}{z_a^3}\frac{1}{(1-z/z_a)^3} =\frac{-1}{2 z_a^3}\sum_{j=0}^\infty (j+1)(j+2) \frac{z^j}{z_a^{j}},$$ we get
\begin{multline*}
 \mathbb{L}\topp{a,b,c}_y\left[\frac{1}{(1+v^2-v y)^n}\right]
  =\frac{-1}{2z_a^3 2\pi \i}\oint_{|z|=\delta} \frac{1}{z^{n}}\left(\sum_{k=1}^\infty a_k z^k\right)\left(\sum_{j=0}^\infty(j+1)(j+2)\left(\frac{z}{z_a}\right)^j\right)\, \d z
 \\=\frac{-1}{2z_a^3}  \sum_{j+k=n-1}^\infty a_k  (j+1)(j+2)\left(\frac{1}{z_a}\right)^j
 =\frac{-1}{2z_a^3}  \sum_{k=0}^{n-1} a_{k}  (n-k)(n+1-k) z_a ^{k+1-n}
\\ =\frac{-n^2}{2z_a^{n+2}}  \sum_{k=0}^{n-1} a_{k}  z_a ^{k}
 +\frac{n}{2z_a^{n+2}}  \sum_{k=0}^{n-1} (2k-1)a_{k}  z_a ^{k}
 +\frac{-1}{2z_a^{n+2}}  \sum_{k=0}^{n-1} k(k-1)a_{k}  z_a ^{k}
 =S_1+S_2+S_3.
\end{multline*}
Since $|z_a|<(1-v)^2$ is within the radius of convergence, the series $\sum_{k=0}^{\infty} a_{k}  z_a ^{k}=\Psi(z_a)$  and the first term gives the dominant term in the asymptotics:
$$
S_1\sim \frac{-n^2}{2z_a^{n+2}} \Psi(z_a)=  \frac{\left(a^2-1\right)^4  v^2}{2a^2 (a-v)^2(1-a v)^2} n^2 \left(\frac{a }{(a-v) (1-a v)}\right)^n.
$$
The other two terms are negligible:
\[\frac{z_a^{n+2}}{n^2}S_2=\frac1{n}  \sum_{k=0}^{n-1} (2k-1)a_{k}  z_a ^{k} \to 0, \]
as the series $\sum_{k=0}^{\infty} (2k-1)a_{k}  z_a ^{k} = 2 z_a\Psi(z_a)'-\Psi(z_a)$ converges. Similarly
\[-2z_a^{n+2}S_3=   \sum_{k=0}^n a_k k(k+1) z_a^{k}\to 2 z_a \Psi'(z_a))+z_a^2\Psi''(z_a),\]
 so $S_3/S_1\to 0$ is of lower order.

 \arxiv{
Note that  $\frac{a }{(a-v) (1-a v)}>\frac{1}{(1-v)^2}>1$ as $z_a<(1-v)^2$.
 }
 }
\arxiv{\item [Proof of \eqref{P1.17iv}] is omitted.}
 \end{itemize}
 \end{proof}

\section{Proofs of Theorems in Section \ref{Sec:Intro}}\label{Proofs-Intro}

\subsection{Proof of Theorem \ref{Thm1-}}\label{Sec:Proof-T1}
The proof relies on several technical lemmas. The first lemma refines \cite[Proposition 2.18]{barraquand2024stationary}; see also \cite[page 20]{Barraquand-2024-integral}, where an elegant domination argument is presented.
For completeness, we provide a direct elementary argument in Appendix \ref{Sec:PoTL}.
\begin{lemma}[{\cite[Section 2.5]{Barraquand-2024-integral}}]\label{Lem:conv}
  \label{lem-AnEx0}
Fix $0<\A<1$.  With $t_*=\max \vv t$, let $\Ca ,\Cb >0$   be such that  $a \Ca <1$,  $\A \Cb  t_*^2 <1$.
Then  the series \eqref{GG}
  converges and defines a real-analytic function of $\Ca $  and $\Cb $.
\end{lemma}
We also need the following recurrence relation, which we will use together with the companion recurrence relation in Lemma \ref{Lem-Irec}.
\begin{lemma}
  \label{Lem:S-rec}
Let  $0<\A \Ca , \A \Cb <1$ and  $\{t_j\}$ be a sequence of positive numbers such that $\A t_j<1$ and $\A t_j^2\Cb <1$. If $\A \ne \Cb  $ then
  \begin{multline}\label{Gc2Ga}
    \mathcal{G}\topp{\A,\Ca ,\Cb }_{N+1}(t_1,\dots,t_{N+1}) =\frac{\Cb }{(\Cb -\A)(1-\A \Cb  t_{N+1}^2)}
      \mathcal{G}\topp{\A,\Ca ,\Cb }_{N}(t_1,\dots,t_{N})
     \\ + \frac{\A (\Ca \Cb -1)}{(\Cb -\A)(1-\A \Ca )(1-\A \Cb t_{N+1}^2)}
        \mathcal{G}\topp{\A,\Ca ,\A}_{N}(t_1,\dots,t_{N}).
  \end{multline}
\end{lemma}
To maintain the continuity of the exposition, we defer the elementary but tedious proof to Appendix \ref{Sec:PoTL}.

 For a non-decreasing sequence $0<t_1\leq t_2\leq \dots< %
 1/\max\{\A,\A\Cb\}$, denote the expression on the right-hand side of  \eqref{S2I}
 by
  \begin{equation}\label{I-rec}
   I\topp{\Ca ,\Cb }_N(\vv t)=
   \pi^{t_1,\dots,t_N;\Cb}\left[\bigotimes_{j=1}^N \frac{1}{\q_{\A t_j}}\right].
 \end{equation}

Then $I\topp{\Ca, \Cb}_N(\vv t)$ satisfies the following recurrence relation, compare \eqref{Gc2Ga}.
 \begin{lemma}\label{Lem-Irec} If $\A \Cb  t_{N+1}^2<1$, $\A t_{N+1}<1$ and $\Cb \ne \A$, then
  \begin{multline}\label{Ic2Ia}
    I\topp{\Ca ,\Cb }_{N+1}(t_1,\dots,t_{N+1}) =\frac{\Cb }{(\Cb -\A)(1-\A \Cb  t_{N+1}^2)}
     I\topp{\Ca ,\Cb }_{N}(t_1,\dots,t_{N})
     \\ + \frac{\A(\Ca \Cb -1)}{(\Cb -\A)(1-\A \Ca )(1-\A \Cb t_{N+1}^2)}
       I\topp{\Ca ,\A}_{N}(t_1,\dots,t_{N}).
  \end{multline}
 \end{lemma}

 \begin{proof}
    Denote
 $F_N=  \bigotimes_{j=1}^N\q_{\A t_j}^{-1}$. Since   $P_{x}^{s,t;\Cb}[1]=1$,
 from \eqref{pii} we get  $\pi^{t_1,\dots,t_{N+1};\Cb}[F_N\otimes 1]=\pi^{t_1,\dots,t_{N};\Cb}[F_N]$. Therefore, by linearity
\begin{multline}\label{I-rec-b}
    I\topp{\Ca ,\Cb }_{N+1}(t_1,\dots,t_{N+1})- \frac{\Cb }{(\Cb -\A)(1-\A \Cb  t_{N+1}^2)}
     I\topp{\Ca ,\Cb }_{N}(t_1,\dots,t_{N})\\
     =\pi^{t_1,\dots,t_{N+1};\Cb}\left[ F_N\otimes\left(\frac{1}{\q_{\A t_{N+1}}}\right)\right]-
     \pi^{t_1,\dots,t_{N+1};\Cb}\left[ F_N\otimes\left(\frac{\Cb }{(\Cb -\A)(1-\A \Cb  t_{N+1}^2)}\right)\right]
     \\=\pi^{t_1,\dots,t_{N+1};\Cb}\left[ F_N\otimes\left(\frac{1}{\q_{\A t_{N+1}}}-\frac{\Cb }{(\Cb -\A)(1-\A \Cb  t_{N+1}^2)}\right) \right] .
\end{multline}
    Noting that
 $$\frac{1}{\q_{\A t_{N+1}}(y)}-\frac{\Cb }{(\Cb -\A)(1-\A \Cb  t_{N+1}^2)}=\frac{\A}{\left(\Cb -\A\right) \left(\A \Cb  t_{N+1}^2-1\right)}\frac{\q_{\Cb t_{N+1}}(y)}{ \q_{\A t_{N+1}}(y)},$$
we see that we can write the last expression in \eqref{I-rec-b} as
     $$   \frac{\A}{\left(\Cb -\A\right) \left(\A \Cb  t_{N+1}^2-1\right)}  \pi^{t_1,\dots,t_{N+1};\Cb}\left[
      F_N \otimes \left(\frac{\q_{\Cb t_{N+1}}}{ \q_{\A t_{N+1}}}\right)
    \right]$$
Thus from \eqref{magic-formula+} we get
\begin{multline*}
\frac{\A (1-\Ca \Cb )}{\left(\Cb -\A\right) \left(\A \Cb  t_{N+1}^2-1\right)(1-\A \Ca )}  \pi^{t_1,\dots,t_{N+1};\A}\left[   F_N\otimes 1\right]=
 \frac{\A (1-\Ca \Cb )}{\left(\Cb -\A\right) \left(\A \Cb  t_{N+1}^2-1\right)(1-\A \Ca )}  \pi^{t_1,\dots,t_{N};\A}\left[   F_N \right]
\\ =\frac{\A (\Ca \Cb-1 )}{\left(\Cb -\A\right) (1-\A \Ca )\left(1-\A \Cb  t_{N+1}^2\right)}  I\topp{\Ca ,\A }_{N}(t_1,\dots,t_{N}).
  \end{multline*}
This proves \eqref{Ic2Ia}.
 \end{proof}

 \begin{lemma}\label{Lem-I-con}
 Suppose $0<\A<1$ and $0<\Ca <1/\A$ are fixed. For fixed monotone $\vv t$ in $(0,1/\A)$  and $N=1,2,\dots$, function
   $(0,1/(\A t_N^2)) \ni \Cab\mapsto I\topp{\Ca ,\Cab}_{N}(\vv t)$ is  left-continuous at $\Cab=\A$.
 \end{lemma}
 \begin{proof}
 We prove the following somewhat more general statement:
 If $\A t_1\leq \A t_2\leq \dots \leq \A t_N<1$ and functions $f_j$ are analytic on $D_{\A t_j}$ then the function
 $\Cab\mapsto \pi^{t_1,\dots,t_N;\Cab}\left[\bigotimes_{j=1}^N f_j\right]$ is left-continuous at $\Cab=\A$.

 If $N=1$, then for any $f_1$ analytic on $D_{\A t_1}$ and $\Cab\nearrow \A$  we have
\(\left|(\pi^{t_1;\Cab}-\pi^{t_1,\A})[ f_1]\right|\to 0\)
by Proposition \ref{P:cont-abc}.

Next, consider $N\geq 2$. We first note that since $\A t_N<1$ and $0<\A\Ca<1$, for all $\Cab\leq \A$ close enough to $\A$ we have $\Cab t_N<1$,  $0<\Cab\Ca<1$ and $\A \Cab t_N^2<1$. Therefore, applying formula \eqref{magic-formula+},  we obtain
\begin{equation}\label{pi(c/a)}
\pi^{t_1,\dots,t_N;\Cab}\left[\bigotimes_{j=1}^N f_j\right]=
\frac{1-\Ca \Cab}{1-\A \Ca}\pi^{t_1,\dots,t_N;\A}\left[\bigotimes_{j=1}^{N-1} f_j
\otimes \left(f_N\frac{\q_{\A t_N} }{\q_{\Cab t_N}}\right)
\right].
\end{equation}
For the repeated values $t_j=t_{j+1}$  we have $P_x^{t_j,t_{j+1};\Cb}[f]=f(x)$, so reducing $N$ if necessary, without loss of generality we may assume that we have distinct  $0=t_0<t_1<\dots<t_N$.
Then  we can write the right hand side of \eqref{pi(c/a)} as a multiple contour integral \eqref{Mult-int} with $d=N$.We obtain
  \begin{multline*}
     \pi^{t_1,\dots,t_N;\Cab}\left[\bigotimes_{j=1}^N f_j\right]=
\frac{1-\Cab \Ca}{1-\A \Ca} \frac{1}{(2\pi \i)^N}
\oint_{|w_1|=\rho_1}\dots \oint_{|w_N|=\rho_N} \frac{1}{(1-\Ca w_1/t_1)}
\\ \times \left(\prod_{k=1}^{N} f_k\left(w_k+\tfrac{1}{w_k}\right) \frac{(1-w_k^2)(1-\A w_k t_{k-1}^2/t_k)}{(1-\A w_k t_k)\,\q_{\frac{t_{k-1}w_k}{t_k}}\left(w_{k-1}+\frac1{w_{k-1}}\right)}\right)
\frac{\q_{\A t_N}(w_N+\tfrac{1}{w_N})}{\q_{\Cab t_N}(w_N+\tfrac{1}{w_N})} \;
\frac{\d w_N}{w_N}  \dots \frac{\d w_1}{w_1}
 \end{multline*}

 We have
 $$\lim_{\Cab\to \A}\frac{1-\Cab \Ca}{1-\A \Ca} =1 \mbox { and by \eqref{q-fac} }
 \lim_{\Cab\to \A}\frac{\q_{\A t_N}(w_N+\tfrac{1}{w_N})}{\q_{\Cab t_N}(w_N+\tfrac{1}{w_N})}=\lim_{\Cab\to \A}\frac{(1-\A t_N w_N)(1-\A t_N/w_N)}{(1-\Cab t_Nw_N)(1-\Cab t_N/w_N)} =1.$$
 With $\Cab\leq \A$,  the modulus of the integrand is a bounded function of variables  $w_1,\dots,w_N$ on the circles $|w_k|=\rho_k$, $k=1,\dots,N$; here, two key properties in \eqref{Mult-int}   are $t_{k-1}\rho_k<t_k\rho_{k-1}$, $\rho_k<1$, $k=1,\dots,N$ and $\rho_N>\A t_N$.
 By the dominated convergence theorem we can pass to the limit under the integral. In the limit, we get the
 integral representation for
  $ \pi^{t_1,\dots,t_N;\A}\left[\bigotimes_{j=1}^N f_j\right]$,
 which ends the proof.
 \arxiv{
To confirm  that the denominators in the integrand are bounded away from 0, recall that in the proof of Proposition \ref{P:comp} we took $\rho_k=\A_k t_k$. We check the following.
\begin{itemize}
\item $|w_k|=\rho_k>\A t_1>0$, $k=1,\dots,N$.
\item $|1-\Ca w_1/t_1|>1-\Ca \rho_1/t_1=1-\A_1 \Ca>0$ as $\A_1\Ca<1$.
    \item $|1-\A w_k t_k|>1-\A\rho_k t_k=1-\A\A_k t_k^2>1-\A_1^2 t_N^2>0$,  $k=1,\dots,N$.
    \item For $k=2,\dots,N$ we have $| 1- t_{k-1} w_k w_{k-1}/t_k|>1- \rho_k \rho_{k-1} t_{k-1}/t_k
    =1-\A_k \A_{k-1} t_{k-1}^2> 1-\A_1^2 t_N^2>0$.
    \item For $k=2,\dots,N$ we have $| 1- t_{k-1} w_k /(w_{k-1}t_k)|>1- \frac{\rho_k  t_{k-1}}{ \rho_{k-1}/t_k}> 1-\max_{k}\frac{\A_k}{\A_{k-1}} >0$ as we chose $\A_k<\A_{k-1}$.
    \item With    $\Cab\leq \A$, we have  $|1-\Cab t_N w_N|>1-\Cab t_N \rho_N\geq 1-\A t_N \rho_N>1-\A_1^2 t_N^2 >0$.
    \item  With   $\Cab\leq \A$, we have $ |1-\Cab t_N/w_N|>1-\Cab t_N/\rho_N=1-\Cab/ \A_N\geq 1-\A/\A_N>0$  by our choice of $\A_N>\A$.
\end{itemize}
 }
 \end{proof}

\begin{proof}[Proof of Theorem \ref{Thm1-}]
In this proof, we fix $\A$ and $\Ca$ such that $\A\Ca<1$ and we fix an infinite  increasing sequence $t_1\leq t_2\leq \dots$ such that $\A t_j<1$. We will vary $\Cb$,  preserving the constraint $\A\Cb t_j^2<1$, which in particular allows us to take a limit as $\Cb\to \A$.

We proceed by induction on $N$. For $N=1$ the maximum  in \eqref{QQQ} is $\vv L_2(1)$,  so direct calculation with the geometric series shows that
$$\mathcal{G}\topp{\A,\Ca ,\Cb }_{1}(t_1)=\frac{1}{(1-\A \Ca)(1-\A \Cb  t_1^2)}.$$
Recalling \eqref{M-Trans}, from %
\eqref{KS-T4.1abc++} used with $v=\A t_1$ we get
 $$
 \pi^{t_1;\Cb}\left[\frac{1}{\q_{\A t_1}}\right]=\frac{1}{(1-\A \Ca )(1-\A  \Cb  t_1^2)},
 $$
provided $\A t_1<1$ and $\A  \Cb t_1^2<1$. Thus \eqref{S2I} holds.

Suppose that \eqref{S2I} holds for some $N$ and all $\Cb<1/(\A t_{N+1}^2)$. Comparing expressions \eqref{Gc2Ga} and \eqref{Ic2Ia}, we see that \eqref{S2I}  then holds for $N+1$, provided that $\A\ne \Cb $.  We then extend the formula to $\A=\Cb $ by continuity. In fact, by Lemma \ref{Lem:conv}, $\mathcal{G}\topp{\A,\Ca ,\Cb }_{N}(\vv t)$ is an analytic function of $\Cb $ and $I_N\topp{\Ca,\Cb}(\vv t)$  depends continuously on $\Cb$ by Lemma \ref{Lem-I-con}.
\end{proof}

\subsection{Proof of Proposition \ref{Prop1}}
This proof is an application of Theorem \ref{Thm1-}. Note that
\begin{equation}\label{G2Gconv}
    G_N(t)=\mathcal{G}_N\topp{\A,\Ca,\Cb}(\vv t),\quad \vv t=(t,t,\dots,t)
\end{equation}
Since $P_{x}^{t,t}[f]=f(x)$, from
\eqref{S2I} used with $t_j=t$ (or \eqref{G2I-r} with $d=1$)
we get \eqref{GenG-L}.

For $|z|$ small enough, the geometric series
$$\sum_{N=1}^\infty z^{N-1}\q_{\A t}^{-N}(y) =\sum_{N=1}^\infty\frac{z^{N-1}}{(1+\A^2t^2-\A t y)^N}=\frac{1}{1-z+\A^2t^2 - \A t y} $$ is a function in $\mathcal{A}_{1/R}$ with $R=\max\{1,\Ca/t,\Cb t\}$.  Furthermore, as in \eqref{q2B}, with $B+1/B=(1-z)/(\A t)+\A t$, we have
$$\frac{1}{1-z+\A^2t^2 - \A t y} = \frac{B}{\A t} \frac{1}{\q_B(y)}.$$
Therefore, from Proposition \ref{P:poly-appr}, we get
\[
    \sum_{N=1}^{\infty} z^{N-1} G_N(t)=
\frac{B}{\A t}\mathbb{L}\topp{\Ca/t,\Cb t}\left[ \frac{1}{\q_B} \right].
\] \color{black}
Formula \eqref{KS-T4.1abc++} gives \eqref{GenG} and completes the proof.

\subsection{Proof of Theorem \ref{Thm-PhD}}\label{Sec:ProofPhD}
We will compute the limit of the Laplace transform $\phi_N(s):=\EE[\exp(2s \vv L_1(N)/N)]$.
Combining \eqref{P-BCY}, \eqref{GG}, we obtain
$$\phi_N(s)=\frac{\mathcal{G}_N\topp{\A,\Ca,\Cb}(\vv t)}{\mathcal{G}_N\topp{\A,\Ca,\Cb}(\vv 1)}.$$
where $\vv t=(t_N,t_N,\dots,t_N)$ with $t_N=e^{s/N}$.
 Thus 
 \eqref{G2I-r} implies
\begin{equation}
    \label{phi0}
    \phi_N(s)=\frac{\mathbb{L}\topp{\Ca/t_N,\Cb t_N}\left[\q_{\A t_N}^{-N}\right]}{\mathbb{L}\topp{\Ca,\Cb}\left[\q_{\A}^{-N}\right]} .
\end{equation}
Note that, see \eqref{ZN},
 $$ \mathbb{L}\topp{\Ca,\Cb}\left[\q_{\A}^{-N}\right] =\mathcal{Z}\topp{\A,\Ca,\Cb}(N). $$

\subsubsection{Proof of Theorem \ref{Thm-PhD}(i)}\label{SubSec:PT14} In this proof, we focus on the interior of region I of Fig. \ref{Fig-PhD}.
At the  boundary, we still have $\frac1N\vv L_1(N)\to \A/(1-\A)$, but the argument need to be modified and is omitted.

If $\Ca,\Cb<1$, then  from \eqref{P1.17iii} we get
\begin{equation}
    \mathcal{Z}\topp{\A,\Ca,\Cb}(N)  \sim \tfrac{(1-\Ca\Cb)}{2\sqrt{\pi} \A^{3/2}(1-\Ca)^2(1-\Cb)^2} %
     \;\frac{1}{ N^{3/2} (1-\A)^{2N}}.
\end{equation}

Next, we analyze the asymptotic of the numerator in \eqref{phi0}.
Consider $s>0$  and $N$ large enough so that  we have  $0<\Ca/t_N,\Cb t_N<1 $. From \eqref{J2c0a} we get
 \begin{multline}
    \label{L2c0a}
    \mathbb{L}\topp{\Ca/ t_N,\Cb t_N}\left[\q_{\A t_N}^{-N}\right]=\frac{1-\Ca\Cb}{2\pi}\int_{-2}^2
   \frac{\sqrt{4-y^2}\, \d y}{(1+\A^2t_N^2-\A t_N y)^N(1+\Cb^2t_N^2-\Cb t_N y)(1+\Ca^2/t_N^2- y\Ca/t_N )}
   \\
   =\frac{1-\Ca\Cb}{2\pi}\int_{-2}^1
   \frac{\sqrt{4-y^2}\, \d y}{(1+\A^2t_N^2-\A  y t_N )^N(1+\Cb^2t_N^2-\Cb t_N y)(1+\Ca^2/t_N^2- y\Ca/t_N )}
   \\ +\frac{1-\Ca\Cb}{2\pi}\int_{1}^2
   \frac{\sqrt{4-y^2}\, \d y}{(1+\A^2t_N^2-\A y t_N )^N(1+\Cb^2t_N^2-\Cb t_N y)(1+\Ca^2/t_N^2- y\Ca/t_N )}=:I_N^-+I_N^+ \mbox{ (say)}.
\end{multline}
As in the proof of \eqref{P1.17iii}, we first  verify that $I_N^-=o\left(\mathcal{Z}\topp{\A,\Ca,\Cb}(N)\right) $ does not contribute to the limit of $\phi_N$.
To see this, we use a trivial bound $1+\A^2t_N^2-\A  y t_N \geq 1+\A^2t_N^2-\A  t_N = (1-\A t_N)^2+\A t_N$. Thus from \eqref{P1.17iii} we get
$$\frac{I_N^-}{\mathcal{Z}\topp{\A,\Ca,\Cb}(N)} \leq
C N^{3/2}\left(\frac{(1-\A)^{2}}{
(1-\A t_N)^2+\A t_N}\right)^N.
$$
Since $\frac{ (1-\A)^2 }{
(1-\A t_N)^2+\A t_N}\to \frac{ (1-\A)^2 }{
(1-\A )^2+\A}<1$, this proves that $I_N^-=o\left(\mathcal{Z}\topp{\A,\Ca,\Cb}(N)\right) $.

Thus
\begin{multline}
    \lim_{N\to\infty}\phi_N(s)=\lim_{N\to\infty} \frac{I_N^+}{\mathcal{Z}\topp{\A,\Ca,\Cb}(N)}
\\
= \lim_{N\to\infty} \frac{ \A^{3/2}(1-\Ca)^2(1-\Cb)^2  }{\sqrt{\pi}(1-\A)^3}\int_{1}^2
   \frac{(1-\A)^{2N}N^{3/2}\sqrt{4-y^2} \,\d y}{(1+\A^2t_N^2-\A y t_N )^N(1+\Cb^2t_N^2-\Cb t_N y)(1+\Ca^2/t_N^2- y\Ca/t_N )} .
\end{multline}
To compute this limit, we change the variable of integration  to $u=(2-y)N$.
 We get
\begin{equation}\label{Pre-lim1}
     \frac{ \A^{3/2}(1-\Ca)^2(1-\Cb)^2  }{\sqrt{\pi}(1-\A)^3}\int_{0}^\infty \mathbf{1}_{u\leq N}
 \frac{(1-\A)^{2N}}{((1-\A^2t_N)^2+\A u t_N/N )^N}  \frac{\sqrt{u}\sqrt{4+u/N} \,\d u}{((1-\Cb^2t_N)^2+\Cb t_N u/N)((1-\Ca^2/t_N^2)^2+ u\Ca/(Nt_N) )}.
\end{equation}

To use dominated convergence theorem, we need the following estimate:
\begin{lemma}
Recall that $t_N=e^{s/N}$ and $y=2-u/N$. Suppose $s\geq 0$. Then
there exists $N_0=N_0(\A,s)$ that does not depend on $y\in[1,2]$ (i.e. $u\in[0,N]$) such that for $N>N_0$ we have
\begin{equation}\label{dom-con-est}
    \frac{(1-\A)^{2N}}{(1+\A^2t_N^2-\A y t_N )^N}\leq  e^{3 \A s/(1-\A)}  \exp\left(-\tfrac{\A }{(1-\A )^2+\A } u/2 \right).
\end{equation}
\end{lemma}
\begin{proof}

    We rewrite the left hand side of \eqref{dom-con-est} as the product
    \begin{equation}
  \label{two-factors} \frac{(1-\A)^{2N}}{{(1-\A t_N)^{2N}}} \frac{{(1-\A t_N)^{2N}}}{(1+\A^2t_N^2-\A y t_N )^N}.
    \end{equation}

And we estimate the two factors separately.  Since $\frac{1-\A}{1-\A t_N}=1+\frac{\A s}{(1-\A) N} +O\left(s^2/N^2\right)$, we have
$$\frac{(1-\A)^{2 N}}{(1-\A t_N)^{2N}}\leq  \exp( 3 \A s/(1-\A))$$ for large $N$. This gives the first factor in \eqref{dom-con-est}.

To estimate the second factor, we use an elementary inequality
\begin{equation}
    \label{E-ineq}
    \frac{1}{1+\alpha v}\leq \exp\left(-\frac{\alpha}{1+\alpha}v\right), \quad 0\leq v\leq 1.
\end{equation}
    \arxiv{Indeed, consider $f(x)=\tfrac{e^{Cx}}{1+\alpha x}$. Then $f(0)=1$ and
    $$
    f'(x)=\frac{e^{C x} (-\alpha +\alpha  C x+C)}{(\alpha  x+1)^2}=\frac{\alpha ^2 e^{\frac{\alpha  x}{\alpha +1}}}{(\alpha +1) (\alpha  x+1)^2} (x-1)\leq 0
    $$
    when $C=\alpha/(1+\alpha)$, $\alpha>0$, $x\in [0,1]$. Thus $f(x)\leq 1$ for $x\in[0,1]$ as asserted.
    }
Writing $y=2-v$ with $v\in[0,1]$ and using the elementary inequality with
$\alpha=\tfrac{\A  t_N}{(1-\A t_N)^2}$, we get
\[\frac{(1-\A t_N)^2}{1+\A^2t_N^2-\A y t_N }=\frac{1}{1+\tfrac{\A v t_N}{(1-\A t_N)^2}}
\leq \exp\left( -\tfrac{\A t_N}{(1-\A t_N)^2+\A t_N} v \right).\]
Since $\tfrac{\A t_N}{(1-\A t_N)^2+\A t_N}\to \tfrac{\A }{(1-\A )^2+\A }>0$, and $v=u/N\in[0,1]$, we get the second factor in \eqref{dom-con-est} for $N$ large enough.
\end{proof}
In view of \eqref{dom-con-est}, we can pass to the limit in \eqref{Pre-lim1} under the integral.
We get
\begin{equation}
    \lim_{N\to\infty}\phi_N(s)=
      \frac{ 2 \A^{3/2} }{\sqrt{\pi}(1-\A)^3} \int_{0}^\infty
 \lim_{N\to\infty}\frac{(1-\A)^{2N}}{((1-\A t_N)^2+\A u t_N/N )^N}  \sqrt{u}  \,\d u.
\end{equation}
From \eqref{two-factors} we see that
$$\lim_{N\to\infty}
 \frac{(1-\A)^{2N}}{((1-\A t_N)^2+\A u t_N/N )^N} =  e^{2 \A s/(1-\A)} \exp\left(-\frac{\A u}{(1-\A)^2}\right).
$$
Thus %
\begin{equation}
    \lim_{N\to\infty}\phi_N(s)= e^{2 \A s/(1-\A)}
      \frac{ 2 \A^{3/2} }{\sqrt{\pi}(1-\A)^3} \int_{0}^\infty
      \sqrt{u}\exp\left( -{\A u}/{(1-\A)^2}\right)\,\d u
=  e^{2s\A/(1-\A)}.
      \end{equation}

\subsubsection{Proof of Theorem \ref{Thm-PhD}(ii)}
If $0<\Cb<\Ca$ and $\Ca>1$, %
then  from \eqref{P1.17i} we get
\begin{equation}
    \mathcal{Z}\topp{\A,\Ca,\Cb}(N) \sim   \frac{\Ca^2-1}{\Ca-\Cb}\;\frac{\Ca^{N-1}}{(\Ca-\A)^N(1-\A\Ca)^N}.
\end{equation}

Next, we use \eqref{J2c1a} to  analyze the asymptotic of the numerator in \eqref{phi0}. For $N$ large enough so that $\Ca>t_N$, we have
 \begin{multline}
    \label{L2c2a+}
    \mathbb{L}\topp{\Ca/t_N,\Cb t_N}\left[\q_{\A t_N}^{-N}\right]
   =\frac{1-\Ca\Cb}{2\pi}\int_{-2}^2
   \frac{\sqrt{4-y^2}\, \d y}{(1+\A^2t_N^2-\A t_N y)^N(1+\Cb^2t_N^2-\Cb t_N y)(1+\Ca^2/t_N^2- y\Ca/t_N )}
    \\  +
\frac{\Ca^2-t_N^2}{\Ca(\Ca-\Cb t_N^2)}\; \frac{\Ca^N}{(1-\A\Ca)^N(\Ca-\A t_N^2)^N} +
\mathbf{1}_{\Cb t_N>1} \frac{\Cb^2 t_N^2-1}{\Cb(\Cb t_N^2-\Ca)} \; \frac{\Cb^N}{(\Cb-\A)^N(1-\A\Cb t_N^2)^N}.
\end{multline}

As in the previous argument, the integral and the atom arising from $\Cb t_N>1$ are of lower order than the normalization constant and
do not contribute to the asymptotic of $\phi_N$. Therefore,
\begin{multline*}
  \phi_N(s)\sim \frac{\Ca-\Cb}{\Ca^2-1}\;\frac{(\Ca-\A)^N(1-\A\Ca)^N}{\Ca^{N-1}} \; \frac{\Ca^2-t_N^2}{\Ca(\Ca-\Cb t_N^2)}\;
  \frac{\Ca^N}{(1-\A\Ca)^N(\Ca-\A t_N^2)^N}
  \\=\frac{\Ca-\Cb}{\Ca-\Cb t_N^2}\; \frac{\Ca^2-t_N^2}{\Ca^2-1}\; \frac{(\Ca-\A)^N}{(\Ca-\A t_N^2)^N}  \sim  \exp\left(\frac{2 \A s}{\Ca-\A}\right),
\end{multline*}
which ends the proof. (See \cite[Theorem 2]{mukherjea2006note}.)
\arxiv{Indeed, with $t_N=e^{s/N}$ we have
\begin{equation}
    \label{exp2} \frac{\Ca-\A}{\Ca-\A t_N^2} =1+\frac{2 \A s}{N (\Ca-\A)}+\frac{2 s^2 \left(\A^2+\A \Ca\right)}{N^2 (\Ca-\A)^2}+O\left(s^3/N^3\right).
\end{equation}

}

\subsubsection{Proof of Theorem \ref{Thm-PhD}(iii)} %
If $0<\Ca<\Cb$ and $\Cb>1$, %
then  from \eqref{P1.17i} we get
\begin{equation}
    \mathcal{Z}\topp{\A,\Ca,\Cb}(N) \sim  \frac{\Cb^2-1}{\Cb-\Ca}\;\frac{\Cb^{N-1}}{(\Cb-\A)^N(1-\A\Cb)^N}.
\end{equation}
Next, we use \eqref{J2c1a} to  analyze the asymptotic of the numerator in \eqref{phi0}. For $N$ large enough so that $t_N\Cb>1$, we have
 \begin{multline}
    \label{L2c2a}
  \mathbb{L}\topp{\Ca/t_N,\Cb t_N}\left[\q_{\A t_N}^{-N}\right]\\=\frac{1-\Ca\Cb}{2\pi}\int_{-2}^2
   \frac{\sqrt{4-y^2} \,\d y}{(1+\A^2t_N^2-\A t_N y)^N(1+\Cb^2t_N^2-\Cb t_N y)(1+\Ca^2/t_N^2- y\Ca/t_N )}
    \\+
\frac{\Cb^2 t_N^2-1}{\Cb(\Cb t_N^2-\Ca)}\; \frac{\Cb^N}{(\Cb-\A)^N(1-\A\Cb t_N^2)^N}  +
\mathbf{1}_{\Ca>t_N}\frac{\Ca^2-t_N^2}{\Ca^2-\Ca \Cb t_N^2}\; \frac{\Ca^N}{(1-\A\Ca)^N(\Ca-\A t_N)^N} .
\end{multline}
Since $ (1+\A^2t_N^2-\A t_N y)^N$  attains its minimum at $y=2$,  the integral  is dominated by a constant multiple of  $(1-\A t_N)^{-2N}$ (with a constant that does not depend on $N$).

Since $\Cb+1/\Cb>2$, we have %
$$\lim_{N\to\infty} \frac{(1-\A/\Cb)(1-\A\Cb)}{(1-\A t_N)^2}= \frac{1+\A^2-\A(\Cb+1/\Cb)}{1+\A^2-2\A}
<1.$$
So $(1-\A t_N)^{-2N}=o\left(\mathcal{Z}\topp{\A,\Ca,\Cb}(N) \right)$ and the contribution of the integral to the limit of $\phi_N$ is negligible.

For $s>0$, the contribution from the atom arising when $\Ca>t_N$ is also negligible $o\left(\mathcal{Z}\topp{\A,\Ca,\Cb}(N) \right)$. Indeed,
$$
\lim_{N\to \infty} \frac{(1-\A/\Cb)(1-\A\Cb)}{(1-\A\Ca)(1-\A t_N/\Ca)}=\frac{1+\A^2-\A(\Cb+1/\Cb)}{1+\A^2-\A(\Ca+1/\Ca)}<1
$$
as $\A>0$ and $x+1/x$ is an increasing function for $x>1$ and in order to have an atom we must have  $1<t_N<\Ca<\Cb$.

Therefore,
$$\phi_N(s)\sim  \frac{(\Cb^2 t_N^2-1)(\Cb-\Ca)}{(\Cb^2-1)(\Cb t_N^2-\Ca)}
\left(\frac{1-\A\Cb}{1-\A\Cb t_N^2}\right)^N\sim e^{2 s \A\Cb/(1-\A\Cb)} \; \mbox{ as $N\to\infty$},
$$
which ends the proof of Theorem \ref{Thm-PhD}(iii). (See \cite[Theorem 2]{mukherjea2006note}.)
\arxiv{
Indeed, with $t_N=e^{s/N}$ we have
\begin{equation}
    \label{exp1}
\frac{1-\A\Cb}{1-\A\Cb t_N^2}=1+\frac{2 s \A \Cb }{N \left(1-\A \Cb\right)}+\frac{2 s^2 \left(\A^2 \Cb^2+\A \Cb\right)}{N^2 \left(1-\A
   \Cb\right){}^2}+O\left(s^3/N^3\right).
   \end{equation}

}

\subsubsection{Proof of Theorem \ref{Thm-PhD}(iv)} %
 If $\Ca=\Cb=\Cab>1$,    then  from \eqref{P1.17ii} we get
\begin{equation}\label{Zcc}
  \mathcal{Z}\topp {\A,\Cab,\Cab}(N)\sim \tfrac{\left(\Cab^2-1\right)^2 \A }{\Cab^2   (\Cab-\A) (1-\A\Cab)}\;N\; \left(\frac{\Cab }{(\Cab-\A) (1-\A  \Cab)}\right)^N.
\end{equation}

Next, we analyze the asymptotic of the numerator in \eqref{phi0}.
Consider $s>0$. Then  $t_N\searrow 1$, and for large $N$ we have  $1<\Cab/t_N<\Cab t_N $ with two atoms that are close  to each other.   From \eqref{J2c1a} we get
 \begin{multline}
    \label{L2c1a}
   \mathbb{L}\topp{\Cab/t_N,\Cab t_N}\left[\q_{\A t_N}^{-N}\right] =
   \frac{1-\Cab^2}{2\pi}\int_{-2}^2
   \frac{\sqrt{4-y^2} \,\d y}{(1+\A^2t_N^2-\A t_N y)^N(1+\Cab^2t_N^2-\Cab t_N y)(1+\Cab^2/t_N^2- y\Cab/t_N )}
   \\+
\frac{\Cab^2 t_N^2-1}{\Cab^2 \left(t_N^2-1\right)} \; \frac{\Cab^N}{(\Cab-\A)^N(1-\A\Cab t_N^2)^N}  +\frac{t_N^2-\Cab^2}{\Cab^2 \left(t_N^2-1\right)}\; \frac{\Cab^N}{(1-\A\Cab)^N(\Cab-\A t_N)^N} .
\end{multline}

Since for large $N$ we have $1+\A^2t_N^2-\A t_N y\geq (1-\A t_N)^2>(1-\A/\Cab)(1-\A\Cab)$ the integral term is $o\left(\mathcal{Z}\topp{\A,\Cab,\Cab}(N)\right)$ as $N\to \infty$, see \eqref{Zcc}.
\arxiv{Indeed,
$ 1+\A^2 t_N^2-2\A t_N> 1+\A^2-\A(\Cab+1/\Cab)$ as $t_N^2> 1$ and since  $t_N\to 1$ and $\Cab+1/\Cab>2$, we have $\Cab+1/\Cab>2t_N$ for large $N$ .
Thus \begin{multline*}
 \frac{1-\Cab^2}{2\pi}\int_{-2}^2
   \frac{\sqrt{4-y^2} \,\d y}{(1+\A^2t_N^2-\A t_N y)^N(1+\Cab^2t_N^2-\Cab t_N y)(1+\Cab^2/t_N^2- y\Cab/t_N )}
   \\
   \leq \frac{1-\Cab^2}{(1-\Cab t_N)^2(1-\Cab/t_N)^2} \frac{1}{(1-\A/\Cab)^N(1-\A\Cab)^N}\frac{1}{2\pi}\int_{-2}^2\sqrt{4-y^2}\,\d y
   \\=
   \frac{1-\Cab^2}{(1-\Cab t_N)^2(1-\Cab/t_N)^2} \frac{1}{(1-\A/\Cab)^N(1-\A\Cab)^N}
   \sim \frac{1+\Cab}{(1-\Cab)^3} \frac{1}{(1-\A/\Cab)^N(1-\A\Cab)^N}=o\left(\mathcal{Z}\topp{\A,\Cab,\Cab}(N)\right).
\end{multline*}

}
Therefore, only the two atoms contribute to the limit,
\begin{multline}
    \phi_N(s)\sim \frac{\Cab^{N-2}}{(t_N^2-1) \mathcal{Z}\topp{\A,\Cab,\Cab}(N)}\left(  \frac{\Cab^2 t_N^2-1}{(\Cab-\A)^N(1-\A\Cab t_N^2)^N}  - \frac{\Cab^2-t_N^2}{(1-\A\Cab)^N(\Cab-\A t_N)^N} \right)
    \\ \sim   %
  \tfrac{\Cab^2   (\Cab-\A) (1-\A\Cab)}{\left(\Cab^2-1\right)^2 \A N} \left(\frac{(\Cab-\A) (1-\A  \Cab)}{\Cab }\right)^N \times   \frac{\Cab^{N-2}}{(t_N^2-1)}\left(  \frac{\Cab^2 t_N^2-1}{(\Cab-\A)^N(1-\A\Cab t_N^2)^N}  - \frac{\Cab^2-t_N^2}{(1-\A\Cab)^N(\Cab-\A t_N)^N} \right)
  \\
  \sim
   \tfrac{    (\Cab-\A) (1-\A\Cab)}{\left(\Cab^2-1\right)^2 \A }       \frac{1}{N(t_N^2-1)}\left(  \frac{  (1-\A  \Cab)^N(\Cab^2 t_N^2-1)}{ (1-\A\Cab t_N^2)^N}  - \frac{ (\Cab-\A)^N  (\Cab^2-t_N^2)}{ (\Cab-\A t_N)^N} \right)
   \\
   \sim  \tfrac{    (\Cab-\A) (1-\A\Cab)}{\left(\Cab^2-1\right) \A }       \frac{1}{2s}\left(  \frac{  (1-\A  \Cab)^N}{ (1-\A\Cab t_N^2)^N}  - \frac{ (\Cab-\A)^N }{ (\Cab-\A t_N)^N} \right)
   \sim  \tfrac{    (\Cab-\A) (1-\A\Cab)}{\left(\Cab^2-1\right) \A }       \frac{1}{2s}\left(e^{\frac{2 \A \Cab s}{1-\A \Cab}} - e^{\frac{2 \A s}{\Cab-\A}} \right).
\end{multline}
\arxiv{
Indeed, with $t_N=e^{s/N}$, we have \eqref{exp1} and \eqref{exp2}:
$$\frac{  1-\A  \Cab}{ 1-\A\Cab t_N^2}=1+\frac{2 s (\A \Cab)}{N (1-\A \Cab)}+\frac{2 s^2 \left(\A^2 \Cab^2+\A \Cab\right)}{N^2 (1-\A \Cab)^2}+O\left(s^3/N^3\right)
,\quad
 \frac{ \Cab-\A  }{ \Cab-\A t_N}=1+\frac{2 \A s}{N (\Cab-\A)}+\frac{2 s^2 \left(\A^2+\A \Cab\right)}{N^2 (\Cab-\A)^2}+O\left(s^3/N^3\right).
 $$
}
Thus for $s> 0$ we have  $$\lim_{N\to\infty}\phi_N(s)=\frac{(\Cab-\A) (1-\A \Cab) }{\A \left(\Cab^2-1\right)} \; \frac{e^{\frac{2 \A \Cab s}{1-\A \Cab}}-e^{\frac{2 \A s}{\Cab-\A}}}{2  s}.$$
An elementary  calculation shows that this is the Laplace transform of the expression on the right-hand side of \eqref{U-lim},
$$
\int_0^1 \exp\left(2 s \left( \tfrac{\A}{\Cab-\A} u+\tfrac{\A \Cab}{1-\A \Cab}(1- u) \right)\right) \,\d u=\frac{    (\Cab-\A) (1-\A\Cab)}{\left(\Cab^2-1\right) \A }       \frac{1}{2s}\left(e^{\frac{2 \A \Cab s}{1-\A \Cab}} - e^{\frac{2 \A s}{\Cab-\A}} \right).
$$
Since convergence of the Laplace transforms of probability measures for $s$ in an open interval implies weak convergence (\cite[Theorem 2]{mukherjea2006note}), convergence in \eqref{U-lim} follows.

\subsection{Proof of Theorem \ref{Prop:Poiss}}\label{Sec:Proof:Poiss}
 The proofs of both parts are similar, but since the details differ, we provide both arguments.
\begin{proof}[Proof of  Theorem \ref{Prop:Poiss}(a)]  The plan of proof is to compute the limit of the multipoint  generating function.   We choose   $0=t_0<t_1<t_2<\dots <t_d<1/\Cb$ and $0=x_0<x_1<\dots<x_d=1$. We define $n_j=\floor{x_j N}-\floor{x_{j-1}N}$, $j=1,\dots, d$, $d=1,2,\dots$. The generating function is:
\begin{equation}\label{E2G/Z}
 \EE\left[\prod_{j=1}^d t_j^{2(\vv L_1(\floor{x_j N}) -\vv L_1(\floor{x_{j-1} N}))}\right] =  \frac{\mathcal{G}_N(\vv t^{\vv n})}{\mathcal{Z}\topp{\A(N),\Ca(N),\Cb}(N)},
\end{equation}
where
 we used notation  \eqref{tn}.
 We first verify that
 \begin{equation}\label{Z-lim-a}
  \mathcal{Z}\topp{\A(N),\Ca,\Cb(N)}(N)\sim 2^N N^Ne^{\la}
 \end{equation}
 For large enough $N$, the parameters are distinct, so we can use   \eqref{J2c1a}. We get
 \begin{multline}
    \mathcal{Z}\topp{\A(N),\Ca(N),\Cb}(N)\\ =
     \frac{1-\Ca \Cb}{2\pi} \int_{-2}^2 \frac{\sqrt{4-y^2}\,\d y}{\q_{\A}^N(y)\q_{\Ca}(y)\q_{\Cb}(y)}
     + \frac{\Ca^2-1}{(\Ca-\Cb)\Ca \q_{\A}^N(\Ca+1/\Ca)}
     +\mathbf{1}_{\Cb>1}\frac{\Cb^2-1}{(\Cb-\Ca)\Cb \q_{\A}^N(\Cb+1/\Cb)}
     =I_N+Z_N+B_N.
 \end{multline}
 We note that $Z_N\sim 2^N N^N  e^{\la}$, as
 $$
\q_{\A}^N(\Ca+1/\Ca)
   =(1-\A\Ca)^N(1-(\A/\Ca))^N=\left(1-\sqrt{\tfrac N{N+1}}\right)^N\left(1-\tfrac{\la}{\sqrt{N(N+1)}}\right)^N
   \\= \tfrac{\left(1-\frac{\la}{\sqrt{N(N+1)}}\right)^N}{\left(\sqrt{N+1}\left(\sqrt{N}+\sqrt{N+1}\right)\right)^N}
    \sim  \frac{e^{-\la}}{2^NN^N}.
$$
  For $y\in[-2,2]$ the minimum of $\q_\alpha$ is at $y=2$, whence by elementary bound \eqref{E-ineq}
 $$|I_N| \leq \frac{|1-\Cb \sqrt{N/\la}|}{\left(1-\sqrt{\la/(N+1)}\right)^{2N}(1-\Cb)^2(\sqrt{N/\la}-1)^2}=O\left(e^{2\sqrt{\la N}}/\sqrt{N}\right)=o(Z_N).$$
 Since $\Cb$ is fixed and $1/\q_{\A}^N(\Cb+1/\Cb)\leq (1-\A)^{-2N}\leq e^{2\sqrt{\la N}}$, we have $B_N=o(Z_N)$, and we proved \eqref{Z-lim-a}.

Next, to analyze the numerator $\mathcal{G}_N(\vv t^{\vv n})$ in \eqref{E2G/Z}, we introduce auxiliary functions $F_m$, of real variable $y_m$, $m=1,\dots,d$ defined recursively by:
$F_d\equiv 1$,
    \begin{equation}
        \label{F2F}
        F_m(y_m)=\mathbb{L}\topp{\Cb t_{m+1},\frac{t_m\UU(y_m)}{t_{m+1}},\frac{t_m}{t_{m+1}\UU(y_m)}} \left[ \frac{F_{m+1}}{\q_{\A t_{m+1}}^{n_{m+1}} }\right] = P_{y_m}^{t_m,t_{m+1};\Cb}\left[ \frac{F_{m+1}}{\q_{\A t_{m+1}}^{n_{m+1}} }\right],
        \quad m=d-1, d-2,\dots,1.
    \end{equation}
  (Compare \eqref{g-rec-0}.)
 We note that
    if $y_m\in[-2,2]$ then $\Cb t_{m+1},|\frac{t_m\UU(y_m)}{t_{m+1}}|,|\frac{t_m}{t_{m+1}\UU(y_m)}|<1$, so by \eqref{J2c0a} we have
    \begin{multline}\label{[-2,2]}
         F_m(y_m)=\mathbb{L}\topp{\Cb t_{m+1},\frac{t_m\UU(y_m)}{t_{m+1}},\frac{t_m}{t_{m+1}\UU(y_m)}} \left[ \frac{F_{m+1}}{\q_{\A t_{m+1}}^{n_{m+1}} }\right]
         \\=\frac{(1-\tfrac{t_m^2}{t_{m+1}^2}) \q_{\Cb t_m}(y_m)}{2\pi}  \int_{-2}^2  \frac{F_{m+1}(y_{m+1})}{\q_{\A t_{m+1}}^{n_{m+1}}(y_{m+1}) }\frac{\sqrt{4-y_{m+1}^2}}{\q_{\Cb t_{m+1}}(y_{m+1}) h_{m+1}(y_m,y_{m+1} )}\, \d y_{m+1},
        \end{multline}
     where
\begin{equation}\label{h(x,y)}
    h_j(x,y):=\q_{\tfrac{t_{j-1} \UU(x)}{t_{j}}}(y)\q_{\tfrac{t_{j-1}}{t_{j}\UU(x)}}(y) =\left(1-\tfrac{t_{j-1}^2}{t_j^2}\right)^2-\tfrac{t_{j-1}}{t_j}\left(1+\tfrac{t_{j-1}^2}{t_j^2}\right)\, x y + \tfrac{t_{j-1}^2}{t_j^2}(x^2+y^2).
\end{equation}
Iterating this procedure we get
        \begin{equation}\label{Fm2II}
     F_m(y_m)=     \frac{\q_{\Cb t_m}(y_m)}{(2\pi)^{d-m}}\int_{-2}^2\dots\int_{-2}^2
        \frac{1}{\q_{\Cb t_{d}}(y_{d})}\prod_{j=m+1}^d  \frac{ 1}{\q_{\A t_{j}}^{n_{j}}(y_{j}) }\frac{\left(1-\tfrac{t_{j-1}^2}{t_j^2}\right)\sqrt{4-y_{j}^2}}
        { h_j(y_{j-1},y_j)}\, \d y_d \dots  \d y_{m+1}.
    \end{equation}
\begin{claim}\label{Cl-a} For $k=1,\dots,d$ we have
    \begin{multline}
   \mathcal{G}_N(\vv t^{\vv n}) =\frac{\Ca^2-t_k^2}{\Ca(\Ca-\Cb t_k^2) }\,\frac{ F_k \left(\tfrac{\Ca}{t_k}+\tfrac{t_k}{\Ca}\right)}{\prod_{j=1}^k \left(\q_{\A t_j}
   \left(\tfrac{\Ca}{t_j}+\tfrac{t_j}{\Ca}\right)\right)^{n_j}}
  \\+\frac{1-\Ca\Cb}{2\pi \Ca^2}\sum_{j=1}^k\left(1-\tfrac{t_{j-1}^2}{t_j^2}\right)\left(\Ca^2-t_{j-1}^2\right)\int_{-2}^2 \frac{F_j(y_j)}{\left(\q_{\A t_j}(y_j)\right)^{n_j}} \, \frac{\sqrt{4-y_j^2}\,\d y_j}{\q_{\Cb t_j}(y_j)\q_{\frac{\Ca}{t_j}}(y_j)\q_{\frac{t_{j-1}^2}{\Ca t_j}}(y_j)}.
   \end{multline}
\end{claim}
\begin{proof}[Proof of Claim \ref{Cl-a}]
As in \eqref{g-rec-0}, we use  \eqref{G2I-r} and \eqref{pii} to represent the generating function
$\mathcal{G}_N(\vv t^{\vv n})$ of random vector  $(\vv L_1(\floor{x_j N}) -\vv L_1(\floor{x_{j-1} N}))_{j=1,\dots,d}$.

\begin{equation}\label{G2L}
    \mathcal{G}_N(\vv t^{\vv n}) :=
    \pi^{t_1,\dots,t_d}\left[\bigotimes_{j=1}^d \frac{1}{\q_{\A t_j}^{n_j}}\right]
   =  \pi^{t_1}\left[\frac{1}{\q_{\A t_{1}}^{n_{1}}  }F_1 \right]=
   \mathbb{L}\topp{\frac{\Ca}{t_1},\Cb t_1}  \left[ \frac{1}{\q_{\A t_{1}}^{n_{1}}  }F_1 \right],
\end{equation}
  We  apply \eqref{J2c1a} with  $\Cb t_d<1$  and
 $\Ca/ t_d>1$, which holds for large enough $N$. 
 We get
\arxiv{$$
\mathbb{L}\topp{\Ca/t_1,\Cb t_1} [f]= \frac{\Ca^2 -t_1^2}{(\Ca -\Cb t_1^2)\Ca} f( \tfrac{\Ca}{t_1}+\tfrac{t_1}{\Ca})  + \tfrac{1-\Ca\Cb}{2\pi}\int_{-2}^2\,f(y_1)\frac{\sqrt{4-y_1^2}}{\q_{\frac{\Ca}{t_1}}(y_1)\q_{\Cb t_1}(y_{1}) }\,\d y_1,
$$
which gives:
}
$$
\mathcal{G}_N(\vv t^{\vv n}) =
\frac{\Ca^2 -t_1^2}{(\Ca -\Cb t_1^2)\Ca}  \frac{F_1\left(\frac{\Ca}{t_1}+\frac{t_1}{\Ca}\right)}{(1-\A \Ca)^{n_1}\left(1-t_1^2\frac{\A}{\Ca} \right)^{n_{1}} }
+
\tfrac{1-\Ca\Cb}{2\pi}\int_{-2}^2 \frac{\sqrt{4-y_1^2}}{\q_{\A t_{1}}^{n_1}(y_1)\q_{\frac{\Ca}{t_1}}(y_1)\q_{\Cb t_1}(y_{1}) }F_1(y_1)\,\d y_1,
$$
which is \eqref{Fm2II} for $k=1$, as $t_0=0$ and $\q_0\equiv 1$.

Assuming  \eqref{Fm2II} holds   for some $k=1,\dots,d-1$, we now verify that it holds for $k+1$. Since    $\UU(\tfrac{\Ca}{t_{k}}+\tfrac{t_{k}}{\Ca})=\tfrac{t_{k}}{\Ca}$, from \eqref{F2F} we get
\begin{multline}
        F_k\left(\tfrac{\Ca}{t_k}+\tfrac{t_k}{\Ca}\right)=P_{\tfrac{\Ca}{t_k}+\tfrac{t_k}{\Ca}}^{t_k,t_{k+1};\Cb}\left[ \frac{F_{k+1}}{\q_{\A t_{k+1}}^{n_{k+1}} }\right]
        =\mathbb{L}\topp{\Cb t_{k+1},\tfrac{t_k^2}{\Ca t_{k+1}},\tfrac{\Ca}{t_{k+1}}}
    \left[\frac{F_{k+1}}{\q_{\A t_{k+1}}^{n_{k+1}}}\right]
    \\=\frac{\Ca^2-t_{k+1}^2}{\Ca-\Cb t_{k+1}^2} \; \frac{\Ca-\Cb t_{k}^2}{\Ca^2-t_{k}^2}\; \frac{F_{k+1}\left(\tfrac{\Ca}{t_{k+1}}+\tfrac{t_{k+1}}{\Ca}\right)}{\left(\q_{\A t_{k+1}}\left(\tfrac{\Ca}{t_{k+1}}+\tfrac{t_{k+1}}{\Ca}\right)\right)^{n_{k+1}}}
   \\ +
   \frac{(1-\Ca\Cb)\left(1-\tfrac{t_k^2}{t_{k+1}^2}\right)\left(1-\tfrac{\Cb}{\Ca}t_k^2\right)}{2\pi}
   \int_{-2}^2 \frac{F_{k+1}(y_{k+1})}{\left(\q_{\A t_{k+1}}(y_{k+1})\right)^{n_{k+1}}} \; \frac{\sqrt{4-y_{k+1}^2}\,\d y_{k+1}}
   {\q_{\Cb t_{k+1}}(y_{k+1})\q_{\frac{\Ca}{t_{k+1}}}(y_{k+1})\q_{\frac{t_{k}^2}{\Ca t_{k+1}}}(y_{k+1})}
\end{multline}
which ends the proof of the claim.
\end{proof}

From  Claim \ref{Cl-a}  for $k=d$, expanding $F_k$ according to \eqref{Fm2II}, we get
 \begin{equation}
     \mathcal{G}_N(\vv t^{\vv n}) = \sum_{k=0}^d J_k(N)
 \end{equation}
 where
 for $k=0,1,..,d-1$,
\begin{multline}
 J_k(N)= \frac{(1-\Ca\Cb)(\Ca^2-t_k^2)}{(2\pi)^{d-k} \Ca^2(\Ca-\Cb t_k^2)} \frac{1}{(1-\A\Ca)^{n_1+\dots+n_k}\prod_{j=1}^k\left(1-\tfrac{t_j^2\A}{\Ca}\right)^{n_j}}
 \\ \times\int_{-2}^2\,\frac{1}{\q_{\A t_{k+1}}^{n_{k+1}}(y_{k+1})} \frac{\left(1-\tfrac{t_k^2}{t_{k+1}^2}\right)\sqrt{4-y_{k+1}^2}}{ \q_{\frac{t_k^2}{t_{k+1}\Ca}}(y_{k+1})\q_{\frac{\Ca}{t_{k+1}}}(y_{k+1})} \int_{-2}^2  \dots \int_{-2}^2
 \frac{1}{\q_{\Cb t_{d}}(y_{d})}
\prod_{j=k+2}^d  \frac{ 1}{\q_{\A t_{j}}^{n_{j}}(y_{j}) }\frac{\left(1-\tfrac{t_{j-1}^2}{t_j^2}\right)\sqrt{4-y_{j}^2}}
        { h_j(y_{j-1},y_j)} \,\d y_d \dots   \d y_{k+1}
\end{multline}
and
 $$J_d(N)=\frac{\Ca^2-t_d^2} { \Ca(\Ca-\Cb t_d^2)} \frac{1}{(1-\A\Ca)^{N}\prod_{j=1}^d\left(1-\tfrac{t_j^2\A}{\Ca}\right)^{n_j}}.
 $$
\arxiv{We add some omitted details:
}
 It is clear that with $ {\A}/{\Ca}\sim  \la/N$ we get
 $$J_d(N) \sim 2^N N^N \prod_{j=1}^d e^{\la t_j^2(x_j-x_{j-1})}.$$

 Since $h_j(u+v,u-v)=(1-r_j^2)^2-(1-r_j)^2 r_j u^2+(1+r_j)^2r_j v^2$, we see that its smallest value under the integral is $h_j(2,2)=(1-t_{j-1}/t_j)^4>0$. Since the denominator of the integrand
 attains its minimum
 value at $y_j=2$. Consequently, the integral is bounded by $$C \exp\left(2\sqrt{N}\lambda\sum_{j=k+1}^d\,t_j(x_j-x_{j-1})\right)$$
 and for $k=0,1,\dots,d-1$  we have $|J_k(N)|\le C(2N)^{Nx_k}\exp\left(2\sqrt{N}\lambda\sum_{j=k+1}^d\,t_j(x_j-x_{j-1})\right)=o(Z_N)$,

 Therefore,
 \begin{multline*}
  \lim_{N\to\infty} \EE\left[\prod_{j=1}^d t_j^{2(\vv L_1(\floor{x_j N}) -\vv L_1(\floor{x_{j-1} N}))}\right] = \lim_{N\to\infty} \frac{\mathcal{G}_N(\vv t^{\vv n})}{\mathcal{Z}\topp{\A(N),\Ca(N),\Cb}(N)}
  \\=\lim_{N\to\infty} \frac{ J_d(N)}{Z_N} =   \exp\left({\la \sum_{j=1}^d (x_j-x_{j-1})(t_j^2-1)}\right)
  =\EE\left[\prod_{j=1}^d t_j^{2(\vv N_{x_j} -\vv N_{x_{j-1}})}\right],
 \end{multline*}
which ends the proof of  Theorem \ref{Prop:Poiss}(a).
\end{proof}

\color{black}

\begin{proof}[Proof of  Theorem \ref{Prop:Poiss}(b)] The plan of proof is similar to the previous part.  We choose   $0=t_0<\Ca<t_1<t_2<\dots <t_d$ and $0=x_0<x_1<\dots<x_d$. We define $n_j=\floor{x_j N}-\floor{x_{j-1}N}$, $j=1,\dots, d$, $d=1,2,\dots$. and consider the generating function
\eqref{E2G/Z}.

 We first verify that
 \begin{equation}\label{Z-lim-b}
   \lim_{N\to\infty}   \mathcal{Z}\topp{\A(N),\Ca,\Cb(N)}(N)= e^{\la}.
 \end{equation}
 For large enough $N$, the parameters are distinct, so we can use   \eqref{J2c1a}. We get
 \begin{multline}
    \mathcal{Z}\topp{\A(N),\Ca,\Cb(N)}(N)=
     \frac{1-\Ca \Cb}{2\pi} \int_{-2}^2 \frac{\sqrt{4-y^2}\,\d y}{\q_{\A}^N(y)\q_{\Ca}(y)\q_{\Cb}(y)}
    +\mathbf{1}_{\Ca>1} \frac{\Ca^2-1}{(\Ca-\Cb)\Ca \q_{\A}^N(\Ca+1/\Ca)}
     +\frac{\Cb^2-1}{(\Cb-\Ca)\Cb \q_{\A}^N(\Cb+1/\Cb)}
     \\=I_N+A_N+Z_N.
 \end{multline}
 For $y\in[-2,2]$ the minimum of $\q_\alpha$ is at $y=2$, we have
 $$|I_N| \leq \frac{|1-\Ca \la N^\theta|}{\left(1-\tfrac1{N^{\theta+1}}\right)^{2N}(1-\Ca)^2(\la N^\theta-1)^2}=O(\tfrac{1}{N^\theta})$$
and since 
$\q_{1/N^{\theta+1}}^N(\Ca+1/\Ca)  =(1-\Ca/N^{\theta+1})^N\left(1-\tfrac{1}{\Ca N^{\theta+1}}\right)^N\to 1$
we have $
|A_N|=O(1/N^\theta)
$. In view of factorization \eqref{q-fac} it is clear that $\lim_{N\to \infty}Z_N=e^\la$, which ends the proof of \eqref{Z-lim-b}.

\arxiv{We add some omitted details:
}

Next, we analyze the numerator $\mathcal{G}_N(\vv t^{\vv n})$ in \eqref{E2G/Z}.
 As previously, the generating function
$\mathcal{G}_N(\vv t^{\vv n})$ of random vector  $(\vv L_1(\floor{x_j N}) -\vv L_1(\floor{x_{j-1} N})_{j=1,\dots,d}$ is   given by
\eqref{G2L} with $F_k$ defined by \eqref{F2F}.

We note that for $\wt y_k:=\Cb t_{k}+\frac{1}{\Cb t_{k}}$ we have

\begin{equation}
    \label{F2Q}
    F_k(\wt y_k)=
\frac{F_{k+1}(\wt y_{k+1})}{\left(\q_{\A t_{k+1}}(\wt y_{k+1})\right)^{n_{k+1}}}=\frac1{\prod_{j=k}^d\,\left(\q_{at_j}(\wt y_j)\right)^{n_j}}=:Q_k(N)=:Q_k
\end{equation}
and if $y_k\in[-2,2]$, then (recall \eqref{h(x,y)}) we obtain
\begin{multline}\label{Fk1}
F_k(y_k)=\mathbb{L} \topp{\Cb t_{k+1},\tfrac{t_k}{t_{k+1}}\UU(y_k),\tfrac{t_k}{t_{k+1}\UU(y_k)}}
\left[\frac{F_{k+1} }{\q_{\A t_{k+1}} ^{n_{k+1}}}\right]
 = \frac{(\Cb^2t_{k+1}^2-1)\left(1-\frac{t_k^2}{t_{k+1}^2}\right)}{\Cb^2 t_{k+1}^2\,\q_{\tfrac{t_{k}}{\Cb t_{k+1}^2}}\left(\tfrac{y_{k}}{t_{k+1}}\right)} Q_{k+1}
 \\+\tfrac{\q_{\Cb t_k}(y_k)\,\left(1-\frac{t_k^2}{t_{k+1}^2}\right)}{2\pi}\,\int_{-2}^2\,  \frac{F_{k+1}(y_{k+1})}{\left(\q_{\A t_{k+1}}(y_{k+1})\right)^{n_{k+1}}}\,\frac{\sqrt{4-y_{k+1}^2}}{\q_{\Cb t_{k+1}}(y_{k+1})h_{k+1}(y_k,y_{k+1})}\,\d y_{k+1}.
\end{multline}


 \begin{claim}
     \label{Cl-b} For $k=1,\dots,d$ we have
      \begin{multline}\label{G2-lim-b}
   \mathcal{G}_N(\vv t^{\vv n})= \tfrac{\Cb^2t_1^2-1}{\Cb^2 t_1^2}Q_1+\tfrac{1-\Ca\Cb}{(2\pi)^k}\,\int_{-2}^2\dots\int_{-2}^{2}\,\tfrac{F_k(y_k)}{\q_{\frac{\Ca}{t_1}}(y_1)\q_{\Cb t_k}(y_k)}\prod_{i=1}^k\,\tfrac{\left(1-\frac{t_{i-1}^2}{t_i^2}\right)\sqrt{4-y_i^2}}{\left(\q_{\A t_i}(y_i)\right)^{n_i}h_i(y_{i-1},y_i)}\,\d y_k\dots\d y_2 \d y_1
   \\ +\sum_{j=2}^k\,\frac{\Cb^2t_j^2-1}{\Cb^2 t_{j}^2 } Q_j\,
    \tfrac{1-\Ca\Cb}{(2\pi)^{j-1}}
    \int_{-2}^2\dots\int_{-2}^{2}\,\tfrac{\left(1-\frac{t_{j-1}^2}{t_j^2}\right)}{\q_{\frac{\Ca}{t_1}}(y_1)\q_{\Cb t_{j-1}}(y_{j-1})\; \q_{\tfrac{t_{j-1}}{\Cb t_j^2}}\left(\tfrac{y_{j-1}}{t_j}\right)}\prod_{i=1}^{j-1}\,\tfrac{\left(1-\frac{t_{i-1}^2}{t_i^2}\right)\sqrt{4-y_i^2}}{\left(\q_{\A t_i}(y_i)\right)^{n_i}h_i(y_{i-1},y_i)}\,\d y_{j-1}\dots\d y_2 \d y_1.
 \end{multline}
 \end{claim}
\begin{proof}[Proof of Claim \ref{Cl-b}]
   We  apply \eqref{J2c1a} with  $a=\Ca/t_1<1$,
 $b=\Cb t_1>1$ and $c=0$, which holds for large enough $N$, to the last expression in \eqref{G2L}.
Using \eqref{F2Q} for $k=1$ we get
\begin{equation}\label{Git1}
\mathcal{G}_N(\vv t^{\vv n}) = \mathbb{L}\topp{\frac{\Ca}{t_1},\Cb t_1} \left[ \frac{F_1}{\q_{\A t_{1}}^{n_{1}} }\right]=\frac{\Cb^2 t_1^2-1}{\Cb^2t_1^2}Q_1+\frac{1-\Ca\Cb}{2\pi}\,\int_{-2}^2\,\frac{F_1(y_1)}{\left(\q_{\A t_{1}}(y_{1})\right)^{n_{1}} }\,\frac{\sqrt{4-y_1^2}}{\q_{\frac{\Ca}{t_1}}(y_1)\q_{\Cb t_1}(y_1)}\,\d y_1.
\end{equation}
Since $h_1(x,y)=1$, this establishes  \eqref{G2-lim-b} 
for $k=1$.

Now we assume that \eqref{G2-lim-b} holds for some $k\in\{1,\dots,d-1\}$ and we will prove that it holds for $k+1$. Applying \eqref{Fk1} to the second term in \eqref{G2-lim-b} we see that it assumes the form
\begin{multline}
  \frac{\Cb^2t_{k+1}^2-1}{\Cb^2 t_{k+1}^2} Q_{k+1}\tfrac{1-\Ca\Cb}{(2\pi)^k}\,\int_{-2}^2\dots\int_{-2}^{2}\,\frac{\left(1-\frac{t_k^2}{t_{k+1}^2}\right)}{\q_{\frac{\Ca}{t_1}}(y_1)\q_{\Cb t_k}(y_k)\q_{\tfrac{t_{k}}{\Cb t_{k+1}^2}}\left(\tfrac{y_{k}}{t_{k+1}}\right)}\prod_{i=1}^k\,\tfrac{\left(1-\frac{t_{i-1}^2}{t_i^2}\right)\sqrt{4-y_i^2}}{\left(\q_{\A t_i}(y_i)\right)^{n_i}h_i(y_{i-1},y_i)}\,\d y_k\dots\d y_2 \d y_1
    \\+\tfrac{1-\Ca\Cb}{(2\pi)^{k+1}}\,\int_{-2}^2\dots\int_{-2}^{2}\, \frac{F_{k+1}(y_{k+1})}{\q_{\frac{\Ca}{t_1}}(y_1)\q_{\Cb t_{k+1}}(y_{k+1})}\,\prod_{i=1}^{k+1}\,\tfrac{\left(1-\frac{t_{i-1}^2}{t_i^2}\right)\sqrt{4-y_i^2}}{\left(\q_{\A t_i}(y_i)\right)^{n_i}h_i(y_{i-1},y_i)}\,\d y_{k+1}\d y_k\dots\d y_2 \d y_1.
\end{multline}
This proves \eqref{G2-lim-b} for $k+1$. (Recall that $F_d\equiv 1$.)

\end{proof}
 From  Claim \ref{Cl-b}  for $k=d$,  we get
 \begin{equation}
     \mathcal{G}_N(\vv t^{\vv n}) = \sum_{k=0}^d J_k(N) ,
 \end{equation}
 where
 $$J_0(N)=\tfrac{\Cb^2t_1^2-1}{\Cb^2 t_1^2}Q_1(N),$$
 for $k=1,..,d-1$,
 $$J_k(N)=\tfrac{\Cb^2t_{k+1}^2-1}{\Cb^2 t_{k+1}^2} Q_{k+1}(N)\tfrac{1-\Ca\Cb}{(2\pi)^k}\,\int_{-2}^2\dots\int_{-2}^{2}\,\frac{\left(1-\frac{t_k^2}{t_{k+1}^2}\right)}{\q_{\frac{\Ca}{t_1}}(y_1)\q_{\Cb t_k}(y_k)\q_{\tfrac{t_{k}}{\Cb t_{k+1}^2}}\left(\tfrac{y_{k}}{t_{k+1}}\right)}\prod_{i=1}^k\,\tfrac{\left(1-\frac{t_{i-1}^2}{t_i^2}\right)\sqrt{4-y_i^2}}{\left(\q_{\A t_i}(y_i)\right)^{n_i}h_i(y_{i-1},y_i)}\,\d y_k\dots\d y_2 \d y_1
   $$
 and
 $$J_d(N)=\tfrac{1-\Ca\Cb}{(2\pi)^d}\,\int_{-2}^2\dots\int_{-2}^{2}\,\frac{1}{\q_{\frac{\Ca}{t_1}}(y_1)\q_{\Cb t_d}(y_d)}\prod_{i=1}^d\,\tfrac{\left(1-\frac{t_{i-1}^2}{t_i^2}\right)\sqrt{4-y_i^2}}{\left(\q_{\A t_i}(y_i)\right)^{n_i}h_i(y_{i-1},y_i)}\,\d y_d\dots\d y_2 \d y_1.$$
Since  by \eqref{q-fac} applied to
$\q_{\A t_k}\left( \Cb t_k+\frac{1}{\Cb t_k}\right)=(1-\A \Cb t_k^2) (1-\A/\Cb) \sim 1-\la/N$, it is clear that
$$ \lim_{N\to\infty}J_0(N) = \exp\left(\la \sum_{j=1}^d (x_j-x_{j-1})t_j^2\right).$$

More generally,   $Q_k(N)$ converge  for $k=1,2,\dots,d$, so   are bounded in $N$.
Next, we observe that   the integrals that appear  in $J_k(N) $  converge to $0$ for $k=1,\dots,d$. To see this, we note that the denominators of the integrands
attain their minimum values at $y_k=2$, $k=1,\dots,d$. In particular, under each integral
$\q_{\Cb t_k}(2)=(\Cb t_k-1)^2\sim \Cb^2 t_k^2\to \infty$ is of higher order than the coefficient
$(1-\Ca\Cb)$ in front of the integral, and the other factors are bounded.

Therefore,
 \begin{multline*}
  \lim_{N\to\infty} \EE\left[\prod_{j=1}^d t_j^{2(\vv L_1(\floor{x_j N}) -\vv L_1(\floor{x_{j-1} N}))}\right] = \lim_{N\to\infty} \frac{\mathcal{G}_N(\vv t^{\vv n})}{\mathcal{Z}\topp{\A(N),\Ca(N),\Cb}(N)}
  \\=\lim_{N\to\infty} \frac{ J_0(N)}{Z_N} =   \exp\left({\la \sum_{j=1}^d (x_j-x_{j-1})(t_j^2-1)}\right)
  =\EE\left[\prod_{j=1}^d t_j^{2(\vv N_{x_j} -\vv N_{x_{j-1}})}\right],
 \end{multline*}
which ends the proof of    Theorem \ref{Prop:Poiss}(b).
\end{proof}

\arxiv{
\subsubsection{Boundary cases}
As we cross the two boundary
 segments   $0< \Ca\leq 1$, $\Cb=1$ and $0<\Cb\leq 1$, $\Ca=1$  the limit of $\vv L_1(N)/N$ varies continuously  on the phase diagram in Fig \ref{Fig-PhD}. However, this
 is not covered by the above proofs, and some details differ.

 We note that  \eqref{P1.17vi} gives the following asymptotics for the normalization constant:
\begin{equation}
    \mathcal{Z}\topp{\A,\Ca,1}(N) \sim
    \frac{1}{(1-\Ca)\sqrt{\pi \A N}}\;\frac{1}{(1-\A)^{2N-1}}.
\end{equation}

Consider the boundary segment $\Ca<1$, $\Cb=1$, and take $s<0$ so that $t_N=e^{s/N}<1$.  Then the argument is similar to the atom-less case $\Ca,\Cb<1$,
as we can use again \eqref{J2c0a}. We have
 \begin{multline}
    \label{L2c0a*}
   \mathbb{L}_y\topp{\Ca/t_N, t_N}\left[(1+\A^2t_N^2-\A  y t_N)^{-N}\right]=
   \\
   =\frac{1-\Ca}{2\pi}\int_{-2}^1
   \frac{(1+\A^2t_N^2-\A  y t_N )^{-N}\sqrt{4-y^2}\, \d y}{(1+t_N^2- t_N y)(1+\Ca^2/t_N^2- y\Ca/t_N )}
   \\ +\frac{1-\Ca}{2\pi}\int_{1}^2
   \frac{(1+\A^2t_N^2-\A  y t_N )^{-N}\sqrt{4-y^2}\, \d y}{ (1+t_N^2-t_N y)(1+\Ca^2/t_N^2- y\Ca/t_N )}.
\end{multline}

As previously, only the second integral over $[1,2]$ contributes to the limit. After the substitution $y=2-u/N$, the contributing part of the integral  gives
\begin{multline*}
 \phi_N(s) \sim    \frac{(1-\Ca)^2\sqrt{ \A }}{\sqrt{\pi}(1-\A)}
 \int_0^N
 \left(\frac{ (1-\A)^2}{(1-\A^2t_N)^2+\A u t_N/N }\right)^N\frac{\sqrt{u} \,\d u}{ N ((1-t_N)^2+t_N u/N)((1-\Ca^2/t_N)^2+ \tfrac{u\Ca}{Nt_N} )}
 \\
 \sim
  \frac{ \sqrt{ \A }}{\sqrt{\pi}(1-\A)}  e^{2 \A s/(1-\A)}
 \int_0^\infty
  \frac{\exp\left(-\frac{\A u}{(1-\A)^2}\right)  \,\d u}{ \sqrt{u}}= e^{2 \A s/(1-\A)} .
\end{multline*}
To justify the use of the dominated convergence theorem we use a version of \eqref{dom-con-est} for $s<0$.(We  omit the details.)

Similarly, when $0<\Cb<1$, $\Ca=1$ we get $\frac1N\vv L_1(N)\to \A/(1-\A)$.
 }

\appendix
\section{Proofs of technical lemmas}\label{Sec:PoTL}

\begin{proof}[Proof of Lemma \ref{Lem:conv}]   We will sum over the end-points of the paths $\vv L$, so we recall the exact expression and a crude upper bound for the number of  paths with a given  end-point:
\begin{equation}
  \label{L-exact} \#\left\{\vv L: \vv L(N)=n\right\}=\begin{pmatrix}
  n+N-1 \\N-1
\end{pmatrix}\leq  (n+1)^N.
\end{equation}

  Note that
  \begin{equation}
    \label{min-BD}
    -\vv L_2(N)\leq \min_{1\leq j\leq N}(\vv L_1(j-1)-\vv L_2(j))\leq  \min\{-\vv L_2(1),(\vv L_1(N-1)-\vv L_2(N))\}.
  \end{equation}

As in the proof of \cite[Proposition 2.18]{barraquand2024stationary}, we consider separately the case $\Ca \Cb \geq 1$ and $\Ca \Cb <1$.

If $\Ca \Cb \geq 1$ then $x\mapsto (\Ca \Cb )^x$ is a nondecreasing function, so the left-hand side of inequality \eqref{min-BD} gives
  $$(\Ca \Cb )^{-\vv L_2(N)} \leq (\Ca \Cb )^{\min_{1\leq j\leq N}(\vv L_1(j-1)-\vv L_2(j))}.  $$
   So in this case the series \eqref{GG} is bounded by
$$
   \sum_{\vv L_1,\vv L_2} (\A \Cb  t_*^2 )^{\vv L_1(N)}  (\A \Ca ) ^{\vv L_2(N)}
  =\sum_{m,n=0}^\infty \sum_{\substack{\vv L_1,\vv L2\\
  \vv L_1(N)=m\\ \vv L_2(N)=n}} (\A \Cb  t_*^2 )^{m}  (\A \Ca ) ^{n}
  \leq \sum_{m,n=0}^\infty \, (m+1)^{N} (n+1)^{N} (\A \Cb t_*^2 )^m (\A \Ca )^n<\infty.
$$
Here we used the crude upper bound \eqref{L-exact} twice.

On the other hand, if $\Ca \Cb <1$ then  $x\mapsto (\Ca \Cb )^x$ is non-increasing, so the right-hand side of \eqref{min-BD} gives
  $$  (\Ca \Cb )^{\min\{-\vv L_2(1),(\vv L_1(N-1)-\vv L_2(N))\}}\leq (\Ca \Cb )^{\min_{1\leq j\leq N}(\vv L_1(j-1)-\vv L_2(j))}. $$
Consider
  \begin{multline}
    \label{XY-ends}
    A_1(m_1,m_2) =\{\vv L_1:\;\vv L_1(N-1)=m_1,\;\vv L_1(N)=m_1+m_2\}\quad
    \\ \mbox{and}\quad A_2(n_1,n_2)=\{\vv L_2:\;n_1=\vv L_2(1),\;\vv L_2(N)=n_1+n_2\},
  \end{multline}
 with $n_1,n_2,m_1,m_2\in\ZZ_{\geq 0}$. We note that
 $$
 \#\,A_1(m_1,m_2)=\binom{m_1+N-2}{m_1}\quad\mbox{and}\quad \#\,A_2(n_1,n_2)=\binom{n_2+N-2}{n_2}.
 $$
 Then
 \begin{multline*}
  \frac{1}{(\Ca \Cb )^{\min\{-\vv L_2(1),(\vv L_1(N-1)-\vv L_2(N))\}}}=(\Ca \Cb )^{n_1} \ind_{m_1>n_2}+(\Ca \Cb )^ {n_1+n_2-m_1}\ind_{n_2\geq m_1}
 \\
  \leq (\Ca \Cb )^{n_1} \ind_{m_1\geq n_2}+(\Ca \Cb )^ {n_1+n_2-m_1}\ind_{n_2\geq m_1}.
 \end{multline*}
  So in this case the series \eqref{GG}  is bounded by
     \begin{multline*}
       \sum_{\vv L_1,\vv L_2}          \frac{(\A \Cb t_*^2 )^{\vv L_1(N)} (\A \Ca ) ^{\vv L_2(N)}}{ (\Ca \Cb )^{\vv L_2(N)} (\Ca \Cb )^{\min\{-\vv L_2(1),(\vv L_1(N-1)-\vv L_2(N))\}} }
      \\ \le \sum_{\substack{m_1,m_2=0\\n_1,n_2=0}
      }^\infty \; \sum_{\substack{\vv L_1\in A_1(m_1,m_2),\\
                                                             \vv L_2\in A_2(n_1,n_2)}}
             \frac{ (\A \Cb t_*^2  )^{m_1+m_2} (\A \Ca ) ^{n_1+n_2}}{ (\Ca \Cb )^{n_1+n_2}}
       \left( (\Ca \Cb )^{n_1} \ind_{m_1\geq n_2}+(\Ca \Cb )^{n_1+n_2-m_1}\ind_{n_2\geq m_1}
   \right)
   \\
  = \sum_{\substack{m_1,m_2=0\\n_1,n_2=0}
      }^\infty \sum_{\vv L_1,\vv L_2} \frac{(\A \Cb  t_*^2 )^{m_1+m_2}(\A \Ca ) ^{n_1+n_2}}{ (\Ca \Cb )^{n_2}}
        \ind_{m_1\ge n_2}
        + \sum_{\substack{m_1,m_2=0\\n_1,n_2=0}
      }^\infty \sum_{\vv L_1,\vv L_2}  \frac{(\A \Cb  t_*^2 )^{m_1+m_2}(\A \Ca ) ^{n_1+n_2}}{ (\Ca \Cb )^{m_1 }} \ind_{n_2\geq m_1}
       \\=: S_1+S_2,
     \end{multline*}
     where the sum over $\vv L_1,\vv L_2$ is now over the pairs of paths    that satisfy \eqref{XY-ends}.
We have
\begin{multline*}
  S_1= \sum_{\substack{m_1,m_2=0\\n_1,n_2=0}
      }^\infty \binom{m_1+N-2}{N-2}\binom{n_2+N-2}{N-2} \frac{(\A \Cb  t_*^2 )^{m_2}(\A \Ca ) ^{n_1}\A^{n_2}}{ \Cb ^{n_2}}
        (\A \Cb  t_*^2 )^{m_1}\ind_{m_1\ge  n_2}
        \\ = \sum_{n_1,n_2,m_2=0}^\infty \binom{n_2+N-2}{N-2} \frac{(\A \Cb  t_*^2 )^{m_2}(\A \Ca ) ^{n_1}\A^{n_2}}{ \Cb ^{n_2}} \sum_{m_1=n_2}^\infty \binom{m_1+N-2}{N-2}\,(\A \Cb  t_*^2 )^{m_1}
        \\ = \tfrac1{(1-\A \Cb  t_*^2 )(1-\A \Ca )}\sum_{n_2=0}^\infty \binom{n_2+N-2}{N-2} \frac{\A^{n_2}}{ \Cb ^{n_2}} \sum_{m_1=n_2}^\infty \binom{m_1+N-2}{m_1}\,(\A\Cb t_*^2 )^{m_1}
        \\ = \tfrac1{(1-\A \Cb  t_*^2 )(1-\A \Ca )}\sum_{n_2=0}^\infty \binom{n_2+N-2}{N-2} (\A t_*)^{2n_2}\sum_{m=0}^\infty \binom{m+n_2+N-2}{N-2}\,(\A\Cb t_*^2 )^{m}
\end{multline*}
Since for $p_1,p_2\in(0,1)$
$$
\sum_{n\ge 0}\binom{N+n}{N}p_1^n\,\sum_{m\ge 0}\,\binom{N+n+m}{N}p_2^m=\sum_{n\ge 0}\binom{N+n}{N}p_1^n\,\sum_{k=0}^N\,a_k(N,p_2)\,n^k<\infty,
$$
we conclude that $S_1$ is finite.

Similarly,
\begin{multline*}
  S_2=\sum_{\substack{m_1,m_2=0\\n_1,n_2=0}
      }^\infty  \binom{m_1+N-2}{N-2}\binom{n_2+N-2}{N-2} \frac{(\A t_*^2 )^{m_1}(\A \Cb  t_*^2 )^{m_2}(\A \Ca ) ^{n_1}}{ \Ca ^{m_1 }} (\A \Ca ) ^{n_2}\ind_{n_2\geq m_1}
      \\ \tfrac{1}{(1-\A \Cb  t_*^2 )(1-\A \Ca )}\sum_{\substack{m_1=0}}^\infty  \binom{m_1+N-2}{N-2}\frac{(\A t_*^2 )^{m_1}}{ \Ca ^{m_1 }}\sum_{n_2=m_1}^\infty \binom{n_2+N-2}{N-2}   (\A \Ca ) ^{n_2}
      \\ \tfrac{1}{(1-\A \Cb  t_*^2 )(1-\A \Ca )}\sum_{\substack{m_1=0}}^\infty  \binom{m_1+N-2}{N-2}\,(\A^2 t_*^2 )^{m_1}\,\sum_{n=0}^\infty \binom{n+m_1+N-2}{N-2}   (\A \Ca ) ^{n}<\infty.
\end{multline*}

To prove analyticity in $\Ca ,\Cb $, we note that
$$\max_{1\le j\leq N}\{\vv L_1(j)-\vv L_2(j-1)\}=
\max_{1\le j\leq N}\{(n_1+\dots+n_j)-(m_1+\dots + m_j)\}=n_1+\max_{0\le j\leq N-1}\left\{\sum_{i=1}^j(n_{i+1}-m_i)\right\}.$$
We substitute $k_i=n_{i+1}-m_i\in\ZZ$ into the expression
\begin{multline}\label{Step0}
   \mathcal{G}_N\topp{\A,\Ca ,\Cb }(\vv t) =
   \sum_{\vv m,\vv n\in\ZZ_{\geq 0}^N}
   \left(\prod_{j=1}^N(\A  t_j^2)^{m_j} \A^{n_j}\right) (\Ca  \Cb )^{ n_1+\max_{0\leq j\leq N-1}\sum_{i=1}^{j}(n_{i+1}-m_i)} \Cb ^{m_N-n_1-\sum_{i=1}^{N-1} (n_{i+1}-m_i)}
   \\=
   \sum_{n_1,m_{N}\in\ZZ_{\geq 0}} (\A \Cb  t_N^2)^{m_N} (\A \Ca )^{n_1} \sum_{m_1,\dots,m_{N-1}\in\ZZ_{\geq 0}}\sum_{k_j\ge -m_j} \prod_{j=1}^{N-1} (\A t_j)^{2m_j}\A^{k_j}
 \Ca ^{\max_{0\leq j \leq N-1}\sum_{i=1}^j k_i} \Cb ^{\max_{0\leq j \leq N-1}\sum_{i=j}^{N-1}(-k_i)}
\end{multline}
as
$$\max_{0\leq j\leq N-1} \left\{\sum_{i=1}^jk_j\right\} - \sum_{j=1}^{N-1}k_j
=\max\left\{- \sum_{i=1}^{N-1}k_i,- \sum_{i=2}^{N-1}k_i,\dots,-k_{N-1},0\right\}.$$
Thus both $\Ca $ and $\Cb $ are raised to non-negative powers.
\end{proof}

\begin{proof}[Proof of Lemma \ref{Lem:S-rec}]  Denote $S_j=\sum_{i=1}^j k_i$ and $S_K^*=\max_{0\leq j\leq K}S_j$.
We rewrite the explicit sum \eqref{Step0} for the generating function, summing over $n_1$ and
$m_N$ and substituting $k_j=n_{j+1}-m_j\geq - m_j$, $j=1,\dots, N-1$. We get
\begin{multline}
  \mathcal{G}\topp{\A,\Ca ,\Cb }_{N}(t_1,\dots,t_{N})=
  \sum_{\vv m,\vv n\in\ZZ_{\geq 0}^{N}} \left(\prod_{j=1}^{N}(\A t_j^2)^{m_j} \A^{n_j}\right)
  (\Ca  \Cb )^{\max_{0\leq j\leq N-1}\left\{\sum_{i=1}^j (n_j-m_j)\right\}} c_s^{\sum_{i=1}^N(m_i-n_i)}
  \\=
  \frac{1}{(1-\A \Ca )(1-\A \Cb  t_N^2)}\sum_{\vv k\in\ZZ^{N-1}}
 (\Ca \Cb )^{S_{N-1}^*}\frac{\prod_{j=1}^{N-1} (\A t_j)^{k_j}}{\prod_{j=1}^{N-1} (\Cb t_j)^{k_j}} \sum_{m_j\geq (-k_j)\vee 0} \prod_{j=1}^{N-1}(\A t_j)^{2 m_j}
 \\=
   \frac{1}{(1-\A \Ca )(1-\A \Cb  t_N^2)}\sum_{\vv k\in\ZZ^{N-1}}
 (\Ca \Cb )^{S_{N-1}^*}\prod_{j=1}^{N-1} \left(\frac{ (\A t_j)^{k_j}}{(\Cb t_j)^{k_j}}  \sum_{m_j\geq (-k_j)\vee 0} (\A t_j)^{2 m_j}\right).
\end{multline}
Since
$$
 (\A t_j)^{k_j} \sum_{m_j\geq (-k_j)\vee 0} (\A t_j)^{m_j}=
\frac{1}{1-\A t_j}(\A t_j)^{k_j} (\A t_j)^{(-k_j)\vee 0}=\frac{(\A t_j)^{|k_j|}}{1-\A t_j}.
$$
we get
\begin{equation}
  \label{eq2use}
  \mathcal{G}\topp{\A,\Ca ,\Cb }_{N}(t_1,\dots,t_{N})=\frac{1}{(1-\A \Ca )(1-\A \Cb  t_N^2)\prod_{j=1}^{N-1}(1-\A^2 t_j^2)}\sum_{\vv k\in\ZZ^{N-1}}
 (\Ca \Cb )^{S_{N-1}^*}\frac{\prod_{j=1}^{N-1}(\A t_j)^{|k_j|}}{
 \prod_{j=1}^{N-1}(\Cb  t_j)^{k_j}}.
\end{equation}
In particular,
\begin{equation} \label{Ga}
  \mathcal{G}\topp{\A,\Ca ,\A}_{N}(t_1,\dots,t_{N})=\frac{1}{(1-\A \Ca ) \prod_{j=1}^{N}(1-\A^2 t_j^2)}\sum_{\vv k\in\ZZ^{N-1}}
 (\Ca \A)^{S_{N-1}^*}\frac{\prod_{j=1}^{N-1}(\A t_j)^{|k_j|}}{
 \prod_{j=1}^{N-1}(\A t_j)^{k_j}}.
\end{equation}
We now apply \eqref{eq2use} with $N+1$ instead of $N$.
We get
\begin{multline}\label{GcN+1}
   \mathcal{G}\topp{\A,\Ca ,\Cb }_{N+1}(t_1,\dots,t_{N},t_{N+1})=\frac{1}{(1-\A \Ca )(1-\A \Cb  t_{N+1}^2)\prod_{j=1}^{N}(1-\A^2 t_j^2)}\sum_{\vv k\in\ZZ^{N}}
 (\Ca \Cb )^{S_{N}^*}\frac{\prod_{j=1}^{N}(\A t_j)^{|k_j|}}{
 \prod_{j=1}^{N}(\Cb  t_j)^{k_j}}
 \\=
 \frac{1}{(1-\A \Ca )(1-\A \Cb  t_{N+1}^2)\prod_{j=1}^{N}(1-\A^2 t_j^2)}\sum_{\vv k\in\ZZ^{N-1}}
 \frac{\prod_{j=1}^{N-1}(\A t_j)^{|k_j|}}{
 \prod_{j=1}^{N-1}(\Cb  t_j)^{k_j}}
 \sum_{k_N\in\ZZ}(\Ca \Cb )^{S_{N}^*} (\A t_N)^{|k_N|}(\Cb  t_N)^{-k_N}.
\end{multline}
We now observe that
$$S_N^*=(S_{N-1}+k_N)\mathbf{1}_{k_N> S_{N-1}^*-S_{N-1}}+ S_{N-1}^*\mathbf{1}_{k_N\leq S_{N-1}^*-S_{N-1}}.$$
Since $S_{N-1}^*-S_{N-1}\geq 0$, we get
\begin{multline}\label{split3way}
  \sum_{k_N\in\ZZ}(\Ca \Cb )^{S_{N}^*} (\A t_N)^{|k_N|}(\Cb  t_N)^{-k_N}
  = \sum_{k_N=S_{N-1}^*-S_{N-1}+1}^\infty (\Ca \Cb )^{S_{N-1}+k_N}(\A t_N)^{k_N} (\Cb  t_N)^{-k_N}
  \\+\sum_{k_N=-1}^{-\infty} (\Ca \Cb )^{S_{N-1}^*} (\A t_N)^{-k_N}(\Cb  t_N)^{-k_N}
  +\sum_{k_N=0}^{S_{N-1}^*-S_{N-1}}  (\Ca \Cb )^{S_{N-1}^*} (\A t_N)^{k_N}(\Cb  t_N)^{-k_N}
 \\ =
  (\Ca \Cb )^{S_{N-1}}\sum_{k_N=S_{N-1}^*-S_{N-1}+1}^\infty (\A \Ca )^{k_N}
  +(\Ca \Cb )^{S_{N-1}^*}\sum_{k_N=-1}^{-\infty}  (\A t_N^2 \Cb )^{-k_N}
 +(\Ca \Cb )^{S_{N-1}^*}\sum_{k_N=0}^{S_{N-1}^*-S_{N-1}}   \left(\frac{\A}{\Cb} \right)^{k_N}\\
 =S_1+S_2+S_3.
\end{multline}
Summing the geometric series and the geometric sum we get
\begin{align*}
 S_1&=\frac{(\Ca \Cb )^{S_{N-1}}}{1-\A \Ca } (\A \Ca )^{S_{N-1}^*-S_{N-1}+1}
=\frac{\A \Ca }{1-\A \Ca } (\A \Ca )^{S_{N-1}^*} \prod_{j=1}^{N-1}\left(\frac{\Cb }{\A}\right)^{k_j}, \\
S_2&=\frac{\A t_N^2 \Cb }{1-\A \Cb  t_N^2}(\Ca \Cb )^{S_{N-1}^*}, \\
S_3&=(\Ca \Cb )^{S_{N-1}^*}\frac{1-\left(\frac{\A}{\Cb} \right)^{S_{N-1}^*-S_{N-1}+1}}{1-\A/\Cb }
= \frac{\Cb }{{\Cb -\A}}(\Ca \Cb )^{S_{N-1}^*} - \frac{\A}{\Cb -\A} (\A \Ca )^{S_{N-1}^*}\prod_{j=1}^{N-1}\left(\frac{\Cb }{\A}\right)^{k_j}.
\end{align*}
Thus
\begin{multline*} \sum_{k_N\in\ZZ}(\Ca \Cb )^{S_{N}^*} (\A t_N)^{|k_N|}(\Cb  t_N)^{-k_N}
= S_1+S_2+S_3
\\=  \frac{\A \left(\Ca  \Cb -1\right)}{\left(1-\A
   \Ca \right) \left(\Cb -\A\right)}  (\A \Ca )^{S_{N-1}^*}\prod_{j=1}^{N-1}\left(\frac{\Cb }{\A}\right)^{k_j}
 + \frac{\Cb  \left(1-\A^2
   t_N^2\right)}{\left(\Cb -\A\right)
   \left(1-\A \Cb  t_N^2\right)} (\Ca \Cb )^{S_{N-1}^*}=:A+B.
\end{multline*}
We put $A+B$ into the right-hand side of \eqref{GcN+1}. The first term, $A$,  gives
\begin{multline*}
  \frac{1}{(1-\A \Ca )(1-\A \Cb  t_{N+1}^2)\prod_{j=1}^{N}(1-\A^2 t_j^2)}\sum_{\vv k\in\ZZ^{N-1}}
 \frac{\prod_{j=1}^{N-1}(\A t_j)^{|k_j|}}{
 \prod_{j=1}^{N-1}(\Cb  t_j)^{k_j}} \frac{\A \left(\Ca  \Cb -1\right)}{\left(\A
   \Ca -1\right) \left(\A-\Cb \right)}  (\A \Ca )^{S_{N-1}^*}\prod_{j=1}^{N-1}\left(\frac{\Cb }{\A}\right)^{k_j}
   \\=
   \frac{\A \left(\Ca  \Cb -1\right)}{\left(1-\A
   \Ca \right) \left(\Cb -\A\right)(1-\A \Cb  t_{N+1}^2)}
    \frac{1}{(1-\A \Ca )\prod_{j=1}^{N}(1-\A^2 t_j^2)}\sum_{\vv k\in\ZZ^{N-1}}
 \frac{\prod_{j=1}^{N-1}(\A t_j)^{|k_j|}}{
 \prod_{j=1}^{N-1}(\A t_j)^{k_j}} (\A \Ca )^{S_{N-1}^*}
 \\= \frac{\A \left(\Ca  \Cb -1\right)}{\left(1-\A
   \Ca \right) \left(\Cb -\A\right)(1-\A \Cb  t_{N+1}^2)} \mathcal{G}\topp{\A,\Ca ,\A}_{N}(t_1,\dots,t_{N}),
\end{multline*}
see \eqref{Ga}. The second term, $B$, gives
\begin{multline*}
 \frac{1}{(1-\A \Ca )(1-\A \Cb  t_{N+1}^2)\prod_{j=1}^{N}(1-\A^2 t_j^2)}\sum_{\vv k\in\ZZ^{N-1}}
 \frac{\prod_{j=1}^{N-1}(\A t_j)^{|k_j|}}{
 \prod_{j=1}^{N-1}(\Cb  t_j)^{k_j}}
\frac{\Cb  \left(1-\A^2
   t_N^2\right)}{\left(\Cb -\A\right)
   \left(1-\A \Cb  t_N^2\right)} (\Ca \Cb )^{S_{N-1}^*}
   \\
   = \frac{\Cb   }{\left(\Cb -\A\right)
   (1-\A \Cb  t_{N+1}^2)}
   \frac{1}{(1-\A \Ca )\left(1-\A \Cb  t_N^2\right)\prod_{j=1}^{N-1}(1-\A^2 t_j^2)}\sum_{\vv k\in\ZZ^{N-1}}
 \frac{\prod_{j=1}^{N-1}(\A t_j)^{|k_j|}}{
 \prod_{j=1}^{N-1}(\Cb  t_j)^{k_j}}(\Ca \Cb )^{S_{N-1}^*}
   \\
   =\frac{\Cb   }{\left(\Cb -\A\right)
   (1-\A \Cb  t_{N+1}^2)}\mathcal{G}\topp{\A,\Ca ,\Cb }_{N}(t_1,\dots,t_{N}),
\end{multline*}
see \eqref{eq2use}. Thus \eqref{Gc2Ga} follows.
\end{proof}

\section{Integral representations of free Askey--Wilson functionals} \label{Sec:I-rep-of-AW} As mentioned in the introduction, Proposition \ref{P:J-rep} gives integral representations for the functionals \eqref{M-Trans+} when the parameters satisfy additional restrictions. We now make this representation explicit.

Recall that $\Ca,\Cb,s,t>0$.
From \eqref{J2c1a} we get the following.
\begin{proposition}\label{pi2Int}
  If $\Ca\ne \Cb t^2$, $\Cb t\ne 1$ and $\Ca\ne t$ then  functional $\pi^t$ extends to bounded measurable  functions $f$  and is given by
  \begin{equation}
     \pi^t[f]=\int_{\RR} f(y) \pi_{t^2}(dy),
  \end{equation}
 where the integral is with respect to the signed compactly supported mixed-type measure:
 \begin{multline*}
 \pi_{t^2}(dy)=\frac{1-\Ca\Cb}{2\pi} \mathbf{1}_{|y|\leq 2} \frac{\sqrt{4-y^2}}{(1+\Ca^2/t^2-\Ca y/t)(1+\Cb^2t^2-\Cb y t)}\,\d y
 \\+ \frac{1}{\Ca-t^2\Cb}\left( \frac{(\Ca^2-t^2)\vee 0}{\Ca} \delta_{\Ca/t+t/\Ca}(\d y)  -
   \frac{(\Cb^2t^2-1)\vee 0}{\Cb} \delta_{\Cb t+1/(\Cb t)}(\d y)\right).
\end{multline*}
\end{proposition}
(The appearance of $t^2$ is for consistency with the notation in \cite[(2.13)]{WWY-2024}.)

 Next, consider the functional $P_{x}^{s,t}$ for $0<s\leq t$ and $x\in\CC$.
 From \eqref{J2m0b-conj} we get $P_{x,}^{t,t}=\delta_x(\d y)$.
 From \eqref{J2c1a} we get the following.
 \begin{proposition}\label{Pst2Int}
     If  $0<s<t$, $\Cb t\ne 1$,
     $x$ is real, $x\ne\pm(\tfrac st+\tfrac ts)$, $x\ne \frac{s}{t^2 \Cb}+\frac{t^2 \Cb}{s}$,
     then  functional $P_{x}^{s,t}$ extends to bounded measurable functions $f$   and is given by
 \begin{equation}
     P_{x}^{s,t}[f] = \int_\RR f(y) P_{s^2,t^2}(x,\d y),
 \end{equation}
 where $P_{s^2,t^2}(x,\d y)$ is a compactly supported complex-valued mixed type measure:
 \begin{multline*}
   P_{s^2,t^2}(x,\d y) =  \mathbf{1}_{|y|\leq 2}\frac{(1+\Cb^2 s^2-\Cb s x)t^2(t^2-s^2)\sqrt{4-y^2}}{2\pi(1+\Cb^2t^2-\Cb t y)((t^2-s^2)^2+s^2 t^2 \left(x^2+y^2\right)-s t x y \left(s^2+t^2\right))}\,\d y
  \\+ \mathbf{1}_{\Cb t>1} \frac{\left(\Cb^2t^2-1\right)
    (t^2-s^2)}{\Cb^2 t^4+s^2 -\Cb st^2x}
  \; \delta_{\left(\Cb t+\tfrac{1}{\Cb t}\right)}(\d y)
   +
   \mathbf{1}_{t|\UU(x)|<s}\frac{ \left(s^2-t^2 \UU^2(x)\right)\left(\Cb  s \UU(x)-1\right)}{s \left(\UU^2(x)-1\right) \left(s-\Cb t^2 \UU(x)\right)} \;\delta_{\left(\tfrac{s }{t \UU(x)}+\tfrac{t \UU(x)}{s }\right)}(\d y).
 \end{multline*}
 \end{proposition}
(The appearance of $s^2,t^2$ in the integral is for consistency with the notation in \cite[(2.14)]{WWY-2024}.)
 \begin{proof}
   Recall Inverse Joukowsky mapping \eqref{u(z)}. The assumptions on the parameters are  $\Cb t\ne 1$, $s<t$,  $\UU(x)\ne \pm 1$, $|\UU(x)|\ne s/t$, $\UU(x)\ne \Cb t^2/s$ and $\UU(x)\ne s/(t \Cb)$. We use \eqref{J2c1a} with
   $a=\Cb t$, $b=s \UU(x)/t$, $c=s/(t \UU(x))$.
   A calculation gives
$$ \q_b(y)\q_c(y)=(1+b^2-by)(1+c^2-cy) =\frac{(t^2-s^2)^2+s^2 t^2 \left(x^2+y^2\right)-s t x y \left(s^2+t^2\right)}{t^4}.$$
Similar calculations give the weights for the atomic part of the measure.

 \end{proof}

\arxiv
{
\section{ When are the free Askey--Wilson functionals real-valued?}
In general, Askey--Wilson functionals are complex-valued. The following gives sufficient conditions for the functional to be real-valued.
\begin{propositiona}\label{Prop:J-real}
  If $f$ is analytic in the interior of an ellipse $\gamma_{\rho'}$ for some $\rho'<\rho$ and real on real axis, i.e.,  $\overline{f(z)}=f(\bar{z})$, then
  \begin{equation}
      \label{BarJ=J}
    \overline{\mathbb{L}\topp{a,b,c}[f]}= \mathbb{L}\topp{\bar{a},\bar{b},\bar{c}}[f].
  \end{equation} In particular if $a,b,c$  are real or if one is real and the other two form a complex-conjugate pair, then $\mathbb{L}\topp{a,b,c}[f]$ is a real number.
\end{propositiona}
\begin{proof}
  If   $\overline{f(z)}=f(\bar{z})$ then  from \eqref{J[f]++} we get %
  \[ \overline{\mathbb{L}\topp{a,b,c}[f]}=-\frac{1}{2\pi\i}\oint_{\gamma_\rho}
\frac{f(\bar{z}) \UU(\bar{z})(1-  \bar{a} \bar{b}\bar{ c} \UU(\bar{z})) \d \bar{z}}{(1- \bar{a} \UU(\bar{z}))
(1-\bar{b} \UU(\bar z))(1-\bar{c} \UU(\bar z))}.\]
Change of variable $z\mapsto\bar z$ reverses the orientation of $\gamma_\rho$, so after the change of sign, we get  \eqref{BarJ=J}.

If $a,b,c$ are real then the second part follows.  If one parameter, say $a$, is real and   $b=\bar{c}$  form   a complex conjugate pair, then \eqref{BarJ=J} says
$$\overline{\mathbb{L}\topp{a,b,\bar{b}}[f]}= \mathbb{L}\topp{a,\bar{b},b}[f] =\mathbb{L}\topp{a,b,\bar{b}}[f], $$
where the last equality is by invariance of $\mathbb{L}\topp{a,b,c}$ under permutation of  the parameters $a,b,c$.
\end{proof}

}

\arxiv{
\section{Analytic proofs of formulas in Proposition \ref{P:J-rep}}\label{Sec:An-Pr}
In the proof we  use  the following   observation.
 If  $f$ is an analytic function on the interior of an ellipse $\gamma_{\rho_0}$ for some $\rho_0\in(0,1)$, then \begin{equation}\label{f:uni-}
    \lim_{\eps\searrow 0} f((1-\eps) e^{\i \theta}+e^{-\i \theta}/(1-\eps))=f(2 \cos \theta) \mbox{ uniformly in } \theta\in \RR.
\end{equation}
}
\arxiv{
\begin{proof}[Proof of \eqref{f:uni-}]
We verify that %
   \begin{equation}\label{f:uni}
      \lim_{r\to 1} f(r e^{\i \theta}+e^{-\i \theta}/r)=f(2 \cos \theta) \mbox{ uniformly in } \theta\in \RR.
   \end{equation}
 We fix $\theta$ and denote $w=r e^{\i \theta}$. Expanding $f(z)$ into a power series near $z=2\cos\theta$   we get
\begin{multline*}
    |f(w+\tfrac1w)-a_0|=|f(w+\tfrac1w)-f(2\cos \theta)|=\left|\sum_{n=0}^\infty a_n (w+\tfrac1w-2\cos \theta)^n\right|
\leq \sum_{n=1}^\infty |a_n| \left(|w-e^{\i \theta}|+|\tfrac1w - e^{-\i \theta}|\right)^n
 \\ =\sum_{n=1}^\infty |a_n| (|1-r|+|\tfrac1r - 1|)^n =\sum_{n=1}^\infty |a_n|\; |1-r|^n(1+1/r)^n  \to 0 \mbox{ as } r\to 1.
\end{multline*}
Although the coefficients $a_n$ depend on $\theta$,  the radius of convergence of the series is bounded from below by the distance $\delta>0$ between the ellipse  $\gamma_{\rho_0}$ and the interval $[-2,2]$ which does not depend on $\theta$. Thus, $|a_n|\leq C/\delta^n$ and the convergence is uniform in $\theta$.
\end{proof}
}
\arxiv{
\begin{proof}[Analytic proof of \eqref{J2c1a} in Proposition \ref{P:J-rep}(i)]
Let $\rho_0<\rho<1/R$.
The integrand  in \eqref{J[f]-w} is analytic in the annulus $\rho_0<|w|<1$, except for up to three simple poles corresponding to $w=1/a, 1/b, 1/c$, if they lie in the unit disk. Thus,  enlarging the contour, we get
\begin{multline}
    \mathbb{L}\topp{a,b,c}[f]=\frac{1}{2\pi\i} \oint_{|w|=1-\eps}\frac{f(w+1/w)(1-abc w)(1-w^2)}{w(1-a w)(1-bw)(1-c w)} \,\d w
    \\
    - 1_{|a|>1/(1-\eps)} \mathrm{Res}_{ w=1/a}\left(\frac{f(w+1/w)(1-abc w)(1-w^2)}{w(1-a w)(1-bw)(1-c w)} \right)
   \\ - 1_{|b|>1/(1-\eps)} \mathrm{Res}_{ w=1/b}\left(\frac{f(w+1/w)(1-abc w)(1-w^2)}{w(1-a w)(1-bw)(1-c w)} \right)
   \\ - 1_{|c|>1/(1-\eps)} \mathrm{Res}_{w=1/c}\left(\frac{f(w+1/w)(1-abc w)(1-w^2)}{w(1-a w)(1-bw)(1-c w)} \right)
  \\ =\frac{1}{2\pi\i} \oint_{|w|=1-\eps}\frac{f(w+1/w)(1-abc w)(1-w^2)}{w(1-a w)(1-bw)(1-c w)} \,\d w
  \\+ 1_{|a|>1/(1-\eps)} \frac{\left(a^2-1\right)
    (1-b c)}{(a-b)
   (a-c)}f\left(a+\tfrac{1}{a}\right) +1_{|b|>1/(1-\eps)}\frac{\left(b^2-1\right) (1-a c)
  }{(b-a) (b-c)} f\left(b+\tfrac{1}{b}\right)
   +
   1_{|c|>1/(1-\eps)} \frac{\left(c^2-1\right) (1-a b)
  }{(c-a) (c-b)}  f\left(c+\tfrac{1}{c}\right).
\end{multline}
We now take the limit as $\eps\to 0$. The  terms obtained from the residua converge, and since $|a|,|b|,|c|\ne 1$, by the dominated convergence the integral also has a limit, see \eqref{f:uni-}.  We get
\begin{multline}
    \mathbb{L}\topp{a,b,c}[f]=\frac{1}{2\pi\i} \oint_{|w|=1}\frac{f(w+1/w)(1-abc w)(1-w^2)}{w(1-a w)(1-bw)(1-c w)} \,\d w
  \\+ 1_{|a|>1} \frac{\left(a^2-1\right)
    (1-b c)}{(a-b)
   (a-c)}f\left(a+\tfrac{1}{a}\right) +1_{|b|>1}\frac{\left(b^2-1\right) (1-a c)
  }{(b-a) (b-c)} f\left(b+\tfrac{1}{b}\right)
   +
   1_{|c|>1} \frac{\left(c^2-1\right) (1-a b)
  }{(c-a) (c-b)}  f\left(c+\tfrac{1}{c}\right).
  \end{multline}
  We now re-write the integral.
 Since
 $$\oint_{|w|=1}\,f(w+\tfrac1w)h(w)\tfrac{\d w}w\,=\i \int_{-\pi}^{\pi} \,f\left(e^{\i\theta}+e^{-\i\theta}\right)\,h\left(e^{\i\theta}\right)\,\d\theta=\i \int_{-\pi}^{\pi} \,f\left(e^{\i\theta}+e^{-\i\theta}\right)\,h\left(e^{-\i\theta}\right)\,\d\theta,
 $$
 we have
 $$
 \tfrac{1}{2\pi \i}\oint_{|w|=1}\,f(w+\tfrac1w)h(w)\tfrac{\d w}w=\tfrac1{4\pi} \int_{-\pi}^{\pi}\,f(2\cos\,\theta) \,(h(e^{\i\theta})+h(e^{-\i\theta}))\,\d\theta.
 $$
 For $$h(z)=\tfrac{(1-abcz)(1-z^2)}{(1-az)(1-bz)(1-cz)},$$ a computation gives
 $$
 h(e^{\i\theta})+h(e^{-\i\theta})=\tfrac{(1-ab)(1-ac)(1-bc) |1-e^{2\i\theta}|^2 }{(1-a e^{\i\theta}) (1-a e^{-\i\theta}) (1-b e^{\i\theta}) (1-b e^{-\i\theta})(1-c e^{\i\theta}) (1-c e^{-\i\theta})}
 =\tfrac{ (1-a b) (1-a c) (1-bc)(2-2 \cos (2\theta))}{(1+a^2-2a\cos \theta)  (1+b^2-2b\cos \theta) (1+c^2-2c\cos\theta)},
 $$
   which is clearly an even function.  Thus, referring to the identity $1-\cos\,2\theta=2\sin^2\theta$, we get
\begin{align}\nonumber
\mathbb{L}\topp{a,b,c}[f]
   = & \frac{ (1-a b) (1-a c) (1-b
   c)}{\pi} \int_0^{\pi}f(2\cos \theta) \frac{\sin\theta }{(1+a^2-2a\cos \theta)  (1+b^2-2b\cos \theta) (1+c^2-2c\cos\theta) } 2 \sin\theta\,\d \theta
  \\ \label{density}
  =&\frac{ (1-a b) (1-a c) (1-b
   c)}{2\pi} \int_{-2}^2 f(y)\frac{\sqrt{4-y^2}}{(1+a^2-ay)  (1+b^2-by) (1+c^2-cy) }\,\d y.
   \end{align}
   (Here we substituted $y=2\cos\ \theta$, $\d y=2 \sin \theta \,\d \theta$.)
\end{proof}
}

\arxiv
{
\begin{proof}[Analytic proof of \eqref{J2c0a} in  Proposition \ref{P:J-rep}(ii)]
   Inspecting the proof of   Proposition \ref{P:J-rep}(ii) we note that the assumption that parameters are distinct is only used in the analysis of the singularities resulting from the parameters that are not in the unit disk.
\end{proof}
}
\arxiv{
\begin{proof}[Analytic proof of \eqref{J2m0b-conj} in Proposition \ref{P:J-rep}(iii)]
 Without loss of generalize we assume $|a|\geq 1$.   The integrand  in \eqref{J[f]-w}  does not depend on $c$ and is analytic in the annulus $\rho_0<|w|<1$, except for a pole at $w=1/a$.
   Thus,  enlarging the contour, and subtracting the residue we get
   \[ \mathbb{L}\topp{a,1/a,c}[f]=  \frac{1}{2\pi\i}\int_{|w|=1-\eps} f(w+1/w)\frac{a \left(1-w^2\right)\d w}{w (a-w) (1-a w)} + \mathbf{1}_{|a|>1/(1-\eps)}f(a+\tfrac1a).\]
  We will need to consider several cases. If $|a|>1$, invoking  \eqref{f:uni-}  we can take the limit as $\eps\to 0$.
The  residue term converges to $f(a+1/a)\mathbf{1}_{|a|>1}$. The integral
converges to
\[
    \frac{1}{2\pi}\int_{-\pi}^\pi f(2\cos \theta) \frac{a(1-e^{2 \i \theta})}{(a -e^{\i \theta})(1-a e^{\i \theta})}\d \theta
=\frac{1}{2\pi}\int_{-\pi}^\pi f(2\cos \theta)\left(\frac{1}{1- ae^{\i  \theta }}-\frac{1}{1-a e^{-\i  \theta }}\right)\d \theta
=0,
\]
as the integrand is an odd function of $\theta$.

Suppose now that $a=e^{i\alpha}\ne \pm 1$. Then the residue term vanishes but as $\eps\to0$, the integral term encounters two singularities at $w=e^{\pm \i \alpha}$ which we will need to isolate. Fix $\delta$ such that $\eps<\delta<|a-\bar a|$.
 After replacing the circle $|z|=1-\eps$ by the closed contour formed by the  two arcs $\ell_1(\eps)\to\ell_1$ and $\ell_1'(\eps)\to\ell_1'$ of the  circle $|w|=1-\eps$ that lie outside of the circles $|w-a|=\delta$  and $|w-\bar a|=\delta$  and two arcs $\ell_\delta(\eps)\to\ell_\delta$ and $\ell_\delta'(\eps)\to\ell_\delta'$ of the circles $|w-a|=\delta$ and $|w-\bar a|=\delta$ that lie inside of the circle $|w|=1-\eps$  (oriented clockwise), we can pass to the limit with $\eps\to 0$.
We get, see Fig. \ref{Fig:ell-delta},
 \begin{multline*}
  \mathbb{L}\topp{a,1/a,c}[f]=
 \frac1{2\pi\i}\int_{\ell_1}f(w+1/w)\frac{a \left(1-w^2\right)\d w}{w (a-w) (1-a w)} +\frac1{2\pi\i}\int_{\ell_1'}f(w+1/w)\frac{a \left(1-w^2\right)\d w}{w (a-w) (1-a w)}
 \\+
 \frac{1}{2\pi \i} \int_{\ell_\delta}f(w+1/w)\frac{a \left(1-w^2\right)\d w}{w (a-w) (1-a w)}
 + \frac{1}{2\pi \i} \int_{\ell_\delta'}f(w+1/w)\frac{a \left(1-w^2\right)\d w}{w (a-w) (1-a w)}
\\= \frac{1}{2\pi\i}\int_{\alpha+2\theta_0}^{2\pi-\alpha-2\theta_0} f(2\cos \theta)\left(\frac{1}{1- e^{\i (\alpha+ \theta) }}-\frac{1}{1- e^{\i (\alpha- \theta) }}\right)\d \theta
+\frac{1}{2\pi\i} \int_{-\alpha+2\theta_0}^{\alpha-2\theta_0}  f(2\cos \theta)\left(\frac{1}{1- e^{\i (\alpha+ \theta) }}-\frac{1}{1- e^{\i (\alpha- \theta) }}\right)\d \theta
 \\+ \frac{1}{2\pi}\int_{\alpha+\pi/2+\theta_0}^{\alpha+3/2\pi-\theta_0}  f\left( \frac{1}{a+\delta  e^{i \theta }}+a+\delta  e^{i \theta }\right) \frac{a \left(a+\delta  e^{i \theta }-1\right) \left(a+\delta  e^{i \theta
   }+1\right)}{\left(a+\delta  e^{i \theta }\right) \left(a^2+a \delta  e^{i
   \theta }-1\right)} \d \theta
  \\
 + \frac{1}{2\pi}\int_{\alpha+\pi/2+\theta_0}^{\alpha+3/2\pi-\theta_0}  f\left( \frac{1}{\bar a+\delta  e^{i \theta }}+\bar a+\delta  e^{i \theta }\right) \frac{\bar a \left(\bar a+\delta  e^{i \theta }-1\right) \left(\bar a+\delta  e^{i \theta
   }+1\right)}{\left(\bar a+\delta  e^{i \theta }\right) \left(\bar{a}^2+\bar a \delta  e^{i
   \theta }-1\right)} \d \theta  \\
   =: I_1+I_1'++I_\delta+I_{\delta}',
 \end{multline*}
 where $\theta_0=2\arcsin(\delta/2)=O(\delta)$.

 {
Parametrizations of the arcs used in the proof:
\begin{align*}
    \ell_1&=\left\{e^{\i \theta}: \theta\in[\alpha+2\theta_0,2\pi -\alpha-2\theta_0]\right\} \\
        \ell_1'&=\left\{e^{\i \theta}: \theta\in[-\alpha+2\theta_0,\alpha  -2\theta_0]\right\} \\
    \ell_\delta&=\left\{e^{i\alpha}+\delta e^{\i \theta}: \theta\in[\alpha+3/2\pi -\theta_0,\alpha+\pi/2+\theta_0]\right\}\\
    \ell'_\delta&=\left\{e^{-i\alpha}+\delta e^{\i \theta}:  \theta\in[-\alpha+3/2\pi+\theta_0,-\alpha+\pi/2 -\theta_0]\right\}
\end{align*}
For the first two curves, we have $\d w/w=\i \,\d \theta$. For the  third curve, we have  $\d w/(a-w)=  -\i\, \d \theta$ with a negative sign  which swaps the limits of integration in $I_\delta$.
 }

We have $I_1'=0$, as the integrand is an odd function of $\theta$. After a change of variable of integration to $\pi+\theta$ in $I_1$, we get
$$
I_1=\frac{1}{2\pi\i}\int_{-\pi+\alpha+2\theta_0}^{\pi-\alpha-2\theta_0} f(-2\cos \theta)\left(\frac{1}{1+ e^{\i (\alpha+ \theta) }}-\frac{1}{1+ e^{\i (\alpha- \theta) }}\right)\d \theta =0,
$$
as again the he integrand is an odd function of $\theta$.
It remains to show that $I_\delta+I_\delta'=f(2\cos\alpha)+O(\delta)$. We have
\begin{multline*}
I_\delta=     \frac{1}{2\pi}\int_{\alpha+\pi/2+\theta_0}^{\alpha+3/2\pi-\theta_0}
f\left( 2 \cos \alpha+\frac{\delta  e^{i \theta -i \alpha } \left(\delta  e^{i \alpha +i \theta }+e^{2
   i \alpha }-1\right)}{e^{i \alpha }+\delta  e^{i \theta }}\right) \left(1+\frac{\delta  e^{i \theta }}{\left(e^{i \alpha }+\delta  e^{i \theta }\right)
   \left(\delta  e^{i (\alpha +\theta )}+e^{2 i \alpha }-1\right)}\right)\d \theta
  \\ =
  \frac{1}{2\pi}\int_{\alpha+\pi/2}^{\alpha+3/2\pi}
f\left( 2 \cos \alpha\right)  \d \theta+O(\delta)=\frac12 f(2\cos\alpha)+O(\delta).
\end{multline*}
Replacing $\alpha$ with $-\alpha$, from the above we get $I_\delta'=f(2\cos \alpha)+O(\delta)$. This proves \eqref{J2m0b-conj} for all the cases with $ab=1$, except the case when $a=\pm 1$.

Finally, we consider the case  $a=\pm 1$. In this case, one of the curves  $\ell_1$, $\ell_1'$  vanishes and  $\ell_\delta=\ell_\delta'$. For definiteness, lets consider $a=1$, i.e. $\alpha=0$.
By the previous argument, the integral over $\ell_1'$ vanishes, and
\begin{multline*}
    I_\delta=\frac{1}{2\pi\i} \int_{\ell_\delta}f(w+1/w)\frac{1+w}{w(1-w)}\d w
    \\=
    -\frac{1}{2\pi\i} \int_{\pi/2+\theta_0}^{3/2\pi-\theta_0}
    f\left(2+\frac{\delta ^2 e^{2 i \theta }}{1+\delta  e^{i \theta }}\right)
     \frac{2+\delta e^{\i \theta}}{(1+\delta e^{\i\theta}) (-\delta e^{\i \theta})}\i \delta e^{\i \theta}\d \theta
       \\=
    \frac{1}{2\pi} \int_{\pi/2+\theta_0}^{3/2\pi-\theta_0}
    f\left(2+\frac{\delta ^2 e^{2 i \theta }}{1+\delta  e^{i \theta }}\right)
     \frac{2+\delta e^{\i \theta}}{1+\delta e^{\i\theta}}\d \theta =
      \frac{1}{2\pi} \int_{\pi/2}^{3/2\pi}
    f\left(2\right) 2\; \d \theta + O(\delta) \to f(2) \mbox{ as } \delta\to 0.
\end{multline*}
The proof for the case $a=-1$ is similar and is omitted.
\end{proof}
}
\arxiv
{\newpage\clearpage   \arxiv{\newpage\clearpage
 \begin{figure}[tbh]
  \begin{tikzpicture}[scale=1.]
\draw[dotted] (0,0) circle(2.82843);
\draw[dotted] (2,2) circle(.5);

\draw[dotted] (2,-2) circle(.5);

\draw[thick] (1.61658, 2.32092)  arc (55.1418: 304.858:2.82843);

\draw[blue](2.32092, -1.61658) arc (-34.8582: 34.8582:2.82843);

\draw[magenta](1.61658,2.32092) arc (140.071: 309.929:.5);

\draw[magenta](2.32092,-1.61658) arc (50.0709: 219.929:.5);

\draw[fill] (0,0) circle(.05);
\draw[fill] (2,2) circle(.05);
\draw[fill] (2,-2) circle(.05);
\draw[-,thin] (0,0) to (2,2);
\draw[-] (0,0) to (1.61658,2.32092);
\draw[-] (0,0) to (2,-2);
\draw[-] (2,2) to (2,2.5);
\node[right]  at (2.0,2.25) {$\delta$};

\draw[-,thin] (-4,0) to (4,0);

\node[above]  at (-2.3,2.3) {$\ell_1$};
\node[right]  at (2.8,.6) {\color{blue}$\ell_1'$};
\node[above]  at (2.0,-1.5) {\color{magenta}$\ell_\delta'$};
\draw[-,dashed] (2.2,2.2) to (3,3);
\draw[-,dashed](1.61658,2.32092) to (2.42487, 3.48138);
\draw[->] (2.9, 2.9)  arc (45: 59:2.9);
\node[above]  at (2.9,3) {$2\theta_0$};

\draw[->,dashed](1,0) arc (0: 45:1);

\node  at (.5,0.2) {$\alpha$};

\draw[dashed](1,0) arc (0: 45:1);

    \end{tikzpicture}
    \caption{ Curves  $\ell_1$ and $\ell_1'$ are arcs of the unit circle  $(\cos \theta,\sin\theta)$.
    Curve $\ell_\delta$ is an arc of the circle $|z-e^{\i \alpha}|=\delta$ %
    that lies inside the unit disk.
     Curve $\ell_\delta'$ is an arc of the circle $|z-e^{-\i \alpha}|=\delta$  that lies inside the unit disk. As $\delta\to 0$, the two smaller circles do not intersect, provided $\alpha\ne 0, \pi$.
    }
    \label{Fig:ell-delta}
\end{figure}
\newpage\clearpage } \clearpage \newpage }
\arxiv
{
\begin{proof}[Proof of \eqref{J3aaa} in Proposition \ref{P:J-rep}] If $|a|<1$ then the result follows from \eqref{J2c0a}. We therefore consider the case $|a|>1$. Recall that $f$ is analytic in the interior of the ellipse $\gamma_{\rho_0}$.
Let $\rho_0<\rho<1/R=1/|a|$.  Then
\[ \mathbb{L}\topp{a,a,a}[f]=  \frac{1}{2\pi\i}\oint_{|w|=\rho} f(w+1/w)\frac{\left(1-w^2\right) \left(1-a^3 w\right)\d w}{w (1-a w)^3}.\]
The integrand    is analytic in the annulus $\rho_0<|w|<1$, except for the singularity of order 3 at $w=1/a$. Since the integrand is continuous   on the unit circle $|w|=1$, passing to the limit as $\eps\to0$ in integrals over the circles $|w=1-\eps$, we get
\[ \mathbb{L}\topp{a,a,a}[f]=  \frac{1}{2\pi\i}\oint_{|w|=1} f(w+1/w)\frac{\left(1-w^2\right) \left(1-a^3 w\right)\d w}{w (1-a w)^3}- \frac{1}{2\pi\i}\oint_{|w-1/a|=\delta} f(w+1/w)\frac{\left(1-w^2\right) \left(1-a^3 w\right)\d w}{w (1-a w)^3}.\]
As previously, see \eqref{density}, the integral over the unit circle yields the integral part of the expression \eqref{J3aaa}.
Taking into account the negative signs before the integral and in the integrand, the integral over the smaller circle $|w-1/a|=\delta$ yields  half of the second derivative
\begin{multline*}
    \frac12 \frac{\d^2}{\d w^2}\left[f(w+1/w)\frac{\left(1-w^2\right) \left(1-a^3 w\right) }{w a^3}\right]
\\=\frac{ \left(w^2-1\right) \left(a^3 \left(2 w^3+w\right)-w^2-2\right)
   f'\left(w+\frac{1}{w}\right)}{a^3 w^4}+
   \frac{\left(w^2-1\right)^3 \left(a^3 w-1\right)
   f''\left(w+\frac{1}{w}\right)}{2a^3 w^5}+\frac{\left(a^3 w^3+1\right)
   f\left(w+\frac{1}{w}\right)}{a^3 w^3}
\end{multline*}
at point $w=1/a$. This gives the answer
$$-\frac{\left(a^2-1\right)^2 \left(a^2+1\right)}{a^3}
   f'\left(a+\tfrac{1}{a}\right)-\frac{\left(a^2-1\right)^4}{2a^4}
   f''\left(a+\tfrac{1}{a}\right)+2   f\left(a+\tfrac{1}{a}\right), $$
which appears in \eqref{J3aaa}.
\end{proof}
}
\arxiv{
\begin{proof}[Proof of \eqref{J2m0a=1b=-1} in Proposition \ref{P:J-rep}(iv)]
   If $a=\pm 1$, $b=-a$, $|c|\ne 1$ then
   The integrand  in \eqref{J[f]-w} is analytic in the annulus $\rho_0<|w|<1$, except possible for a pole $w=1/c$ if $|c|>1$.
   Thus,  enlarging the contour, we get
\begin{multline}
   \mathbb{L}\topp{a,b,c}[f]=  \frac{1}{2\pi\i}\int_{|w|=\rho} f(w+1/w) \frac{(1+ c w)\d w}{w(1- c w)}
   \\= \frac{1}{2\pi\i}\int_{|w|=1-\eps} f(w+1/w) \frac{(1+ c w)\d w}{w(1- c w)} - \mathbf{1}_{|c|>1/(1-\eps)}\mathrm{Res}_{ w=1/c}\left(f(w+1/w) \frac{(1+ c w)}{w(1- c w)}\right)
     \\ \to
    \frac{1}{2\pi\i}\int_{|w|=1} f(w+1/w) \frac{(1+ c w)\d w}{w(1-c w)}+2 \mathbf{1}_{|c|>1} f(c+1/c),
    \end{multline}
where we used \eqref{f:uni-} again. We now re-write the integral over the unit circle. We get
    \begin{multline}
\frac{1}{2 \pi} \int_{-\pi}^\pi f(2 \cos \theta)  \frac{1+ c e^{\i \theta}}{1- c e^{\i \theta}}
 =\frac{1}{4 \pi} \int_{-\pi}^\pi f(2 \cos \theta)  \left( \frac{1+ c e^{\i \theta}}{1- c e^{\i \theta}} +
     \frac{1+ c e^{-\i \theta}}{1-c e^{-\i \theta}} \right)
     = \frac{1-c^2}{2 \pi} \int_{-\pi}^\pi f(2 \cos \theta)  \frac{\d\theta }{1+c^2-2 c \cos
   (\theta )}
    \\= \frac{1-c^2}{ \pi} \int_{0}^\pi f(2 \cos \theta)  \frac{\d\theta }{1+c^2-2 c \cos
   (\theta )}= \frac{1-c^2}{\pi}\int_{-2}^2 f(y) \frac{\d y}{(1+c^2-c y)\sqrt{4-y^2}}.
\end{multline}
\end{proof}
}
\arxiv{ \comment{This proof refers to formulas that are in expanded version only, so it cannot be part of the main paper}
\begin{proof}[Proof of \eqref{J2c2a} in Remark \ref{R:J-rep}] If $a=b$ and $|a|,|c|<L1$  then the then the result follows from \eqref{J2c0a}. If $|a|<1$  and $|c|>1$ then the proof is similar to the derivation of \eqref{J2c1a} and is omitted.

We therefore consider the case $|a|>1$.  Recall that $f$ is analytic in the interior of the ellipse $\gamma_{\rho_0}$.
Let $\rho_0<\rho<1/R=1/(|a|\vee |c|)$.
The starting point is  \eqref{J[f]-w} which takes the form
\[ \mathbb{L}\topp{a,a,c}[f]=  \frac{1}{2\pi\i}\oint_{|w|=\rho} f(w+1/w)\frac{\left(1-w^2\right) \left(1-a^2c w\right)\d w}{w (1-a w)^2(1-c w)} .
\]
 In the annulus $\rho_0<|w|<1$. the integrand has singularity of order 2 at $w=1/a$ and possibly a pole at $w=1/c$. Since $c\ne a$ we can choose $\delta>0$ such that $|c-a|>2\delta$ and (taking again the limit as $\rho\to 1$) we write
\begin{multline*}
     \mathbb{L}\topp{a,a,c}[f]=  \frac{1}{2\pi\i}\oint_{|w|=1} f(w+1/w)\frac{\left(1-w^2\right) \left(1-a^2c w\right)\d w}{w (1-a w)^2(1-c w)}
     \\
     -\frac{1}{2\pi\i}\oint_{|w-1/a|=\delta} f(w+1/w)\frac{\left(1-w^2\right) \left(1-a^2c w\right)\d w}{w (1-a w)^2(1-c w)}
     - \mathbf{1}_{|c|>1} \frac{1}{2\pi\i}\oint_{|w-1/c|=\delta} f(w+1/w)\frac{\left(1-w^2\right) \left(1-a^2c w\right)\d w}{w (1-a w)^2(1-c w)}  =:I-I_a-I_c.
\end{multline*}
The integral $I$ over the unit circle  is given by formula \eqref{density} with $b=a$.

The integral $I_c$ over  $|w-1/c|=\delta$ is given by the residue:
\[\mathrm{Res}_{w=1/c}
\left( f(w+1/w)\frac{\left(1-w^2\right) \left(1-a^2c w\right) }{w (1-a w)^2(1-c w)} \right)  =-\frac{\left(1-a^2\right) \left(c^2-1\right)}{(a-c)^2} f(c+\tfrac1c).
\]
The integral $I_a$ is given as the derivative:
\[
   I_a=\frac{\d}{\d w} \left[ f(w+1/w)\frac{\left(1-w^2\right) \left(1-a^2c w\right) }{a^2 w  (1-c w)} \right]_{w=1/a}
   =-\frac{\left(a^2-1\right)^2   (1-a c) }{a^2(a-c)}f'\left(a+\tfrac{1}{a}\right)-\frac{
   (a-c)^2+(1-a c)^2
   }{(a-c)^2} f\left(a+\tfrac{1}{a}\right).
\]
\end{proof}
}
\arxiv
{
\begin{proof}[Sketch of another Proof of \eqref{J2c2a} in Remark \ref{R:J-rep}]
By  Proposition \ref{P:poly-appr}, from \eqref{Dabb} we have
\begin{equation}\label{neww}
    \mathbb{L}\topp{a,a,c}[f]=\frac{c^2(1-a^2)}{(a-c)^2}\mathbb{L}\topp{c}[f]+ \frac{a(a-c-c(1-ac))}{(a-c)^2}\mathbb{L}\topp{a}[f]+
    \frac{a^2(1-a c)}{a-c} \frac{\partial}{\partial a}\mathbb{L}\topp{a}[f]
\end{equation}
Recalling \eqref{Jck2} and the bound $|\q_a(y)|\geq (1-|a|)^4 $, in view of
$$
\frac{\partial}{\partial a}\frac1{\q_a(y)}=\frac{\partial}{\partial a}\frac{1}{1+a^2-ay}=\frac{y-2a}{\q_a^2(y)}=-\frac1{a\q_a(y)}+\frac{1-a^2}{a \q_a^2(y)},
$$
differentiating under the integral $\mathcal L\topp{0}$, we get \comment{How do we justify this? Better re-fresh the complex analysis proof!}
\begin{multline*}
\frac{\partial}{\partial a}\mathbb{L}\topp{a}[f]=\frac{\partial}{\partial a}\left(\mathbf{1}_{|a|>1} \left(1-\tfrac{1}{a^2}\right) f\left(a+\tfrac1a\right)+\mathcal{L}\topp 0\left[\frac{f}{\q_a}\right]\right)\\
=\mathbf{1}_{|a|>1}\left(\frac{2}{a^3}f(a+\tfrac1a)+\frac{(a^2-1)^2}{a^4}f'\left(a+\tfrac1a\right)\right)-\frac1a\mathcal{L}\topp 0\left[\frac{f}{\q_a}\right] +\frac{1-a^2}{a}\mathcal{L}\topp{0}\left[\frac{f}{\q_a^2}\right]
\end{multline*}

Referring again to \eqref{Jck2} for $\mathbb{L}\topp{a}[f]$ and for $\mathbb{L}\topp{c}[f]$ now can rewrite \eqref{neww} as
\begin{multline}
\mathbb{L}\topp{a,a,c}[f]=\mathbf{1}_{|c|>1} \tfrac{(c^2-1)(1-a^2)}{(a-c)^2} f\left(c+\tfrac1c\right)+ \mathbf{1}_{|a|>1} \left(
    \left(1+ \tfrac{(1-a c)^2}{(a-c)^2} \right)f\left(a+\tfrac1a\right)
    +\tfrac{(a^2-1)^2(1-a c)}{a^2(a-c)} f'\left(a+\tfrac1a\right) \right)
     \\+\mathcal{L}\topp{0}\left[\left(\frac{c^2(1-a^2)}{(a-c)^2\q_c}    + \frac{a c(a^2-1)}{(a-c)^2\q_a} + \frac{a(1-a^2)(1-ac)}{(a-c)\q_a^2}\right) f\right]
\end{multline}
Collecting the terms under $\mathcal{L}\topp{0}$ we get \eqref{J2c2a}.

\end{proof}
}

\subsection*{Acknowledgements}
W.B. thanks Guillaume Barraquand for an advanced copy of    \cite{Barraquand-2024-integral}. The authors thank Xiaolin Zeng for discussions.
 This research was partially supported by Simons Grant~(703475) to W.~B.,
NCN grant Weave-UNISONO [BOOMER 2022/04/Y/ST1/00008] to K.~S.  and National Science Center Poland [project no.  2023/51/B/ST1/01535] to J.~W.
\bibliographystyle{apalike}
\bibliography{vita,LPP-2024}
\end{document}